\documentclass[12pt]{article} 
\usepackage{array,epsfig,fancyhdr,rotating}
\usepackage[]{hyperref}  
\usepackage{sectsty, secdot}
\renewcommand{\theequation}{\thesection\arabic{equation}}
\subsectionfont{\fontsize{12}{14pt plus.8pt minus .6pt}\selectfont}

\usepackage{amsmath}
\usepackage{amssymb}
\usepackage{amsfonts}
\usepackage{multirow}
\usepackage{amsthm}

\usepackage{natbib}
\usepackage{comment}
\usepackage{enumerate}
\usepackage[ruled,vlined]{algorithm2e}
\usepackage{hyperref} 
\usepackage{algpseudocode}
\usepackage{float}
\usepackage{booktabs}
\usepackage{color}
\usepackage{xcolor}
\usepackage{mathtools}
\mathtoolsset{showonlyrefs}

\newcommand{\field}[1]{\mathbb{#1}}
\newcommand{\R}{\field{R}}

\newcommand{\WCr}{\textcolor{black}}
\newcommand{\p}{\field{P}}
\newcommand{\N}{\field{N}}
\newcommand{\Z}{\field{Z}}

\newcommand{\E}{\field{E}}

\newcommand{\FF}{\mathcal{F}}

\newcommand{\WC}{\textcolor{black}}

\def\argmin{\mathop{\mbox{argmin}}}

\def\lf{\lfloor}
\def\rf{\rfloor}
\def\lc{\lceil}
\def\rc{\rceil}

\newtheorem{theorem}{Theorem}
\newtheorem{lemma}{Lemma}

\newtheorem{proposition}{Proposition}
\theoremstyle{definition}

\newtheorem{example}{Example}
\newtheorem{remark}{Remark}
\newtheorem{assumption}{Assumption}[section]
\begin{document}


\renewcommand{\baselinestretch}{2}

\markright{ \hbox{\footnotesize\rm 
}\hfill\\[-13pt]
\hbox{\footnotesize\rm
}\hfill }

\markboth{\hfill{\footnotesize\rm FIRSTNAME1 LASTNAME1 AND FIRSTNAME2 LASTNAME2} \hfill}
{\hfill {\footnotesize\rm FILL IN A SHORT RUNNING TITLE} \hfill}

\renewcommand{\thefootnote}{}
$\ $\par


\fontsize{12}{14pt plus.8pt minus .6pt}\selectfont \vspace{0.8pc}
\centerline{\large\bf Confidence surfaces for the mean  of locally stationary }
\vspace{2pt} 
\centerline{\large\bf  functional time series}
\vspace{.4cm} 
\centerline{Holger Dette, Weichi Wu} 
\vspace{.4cm} 
\centerline{\it Ruhr-Universit\"at Bochum and Tsinghua University}
 \vspace{.55cm} \fontsize{9}{11.5pt plus.8pt minus.6pt}\selectfont


\begin{quotation}
\noindent {\it Abstract:}
The problem of constructing a simultaneous confidence surface for the 2-dimensional mean function of a non-stationary functional time series is challenging as these bands can not be built on classical limit theory for the maximum absolute deviation between an estimate and the time-dependent regression function. In this paper, we propose a new bootstrap methodology to construct such a region. Our approach is based on a  Gaussian approximation for the maximum norm of sparse high-dimensional vectors approximating the maximum absolute deviation which is suitable for nonparametric inference of high-dimensional time series. The elimination of the zero entries produces (besides the time dependence) additional dependencies such that the ''classical'' multiplier bootstrap is not applicable. To solve this issue we develop a novel multiplier bootstrap, where blocks of the coordinates of the vectors are multiplied with random variables, which mimic the specific structure between the vectors appearing in the Gaussian approximation. We prove the validity of our approach by asymptotic theory, demonstrate good finite sample properties by means of a simulation study and illustrate its applicability by analyzing a data example.

\vspace{9pt}
\par
\end{quotation}\par

\def\thefigure{\arabic{figure}}
\def\thetable{\arabic{table}}

\renewcommand{\theequation}{\thesection.\arabic{equation}}

\fontsize{12}{14pt plus.8pt minus .6pt}\selectfont

\section{Introduction}
\label{sec1} 
\def\theequation{1.\arabic{equation}}
\setcounter{equation}{0}

In the big data era  data gathering technologies provide enormous amounts of data with complex structure.
In many applications the observed  data exhibits certain degrees of
dependence and smoothness and thus may naturally be regarded as discretized
functions. A major tool for the  statistical analysis of  such data is  functional data analysis 
(FDA) which  has
found considerable  attention in the  statistical literature
\citep[see, for example, the monographs of][among others]{bosq2000,RamsaySilverman2005,FerratyVieu2010,HorvathKokoskza2012,hsingeubank2015}.
In FDA the considered parameters, such as the mean or the (auto-)covariance (operator) 
are functions themselves, which makes the development of statistical methodology 
challenging.  
Most of the literature considers  Hilbert space-based methodology for which  there exists by now a well-developed theory. 
In particular, this  approach   allows the application of
dimension reduction techniques such as (functional) principal components \citep[see, for example,][]{shang2014}.
On the other hand, in many applications data is observed on a very fine gird and it is reasonable to assume that functions are at least continuous
\citep[see also][for a discussion of the integral role of smoothness]{RamsaySilverman2005}. In such cases 
fully functional methods can prove advantageous
and have been recently, developed by   \cite{HorvathKokoszkaRice2014}, \cite{BucchiaWendler2015}, 
\cite{Aue2015DatingSB},   \cite{dettekokotaue2020}
and \cite{ dettekokot2020} among others.

In this paper we are interested in statistical inference regarding the smooth mean functions
of a not necessarily stationary functional time series $(X_{i,n})$, $1\leq i\leq n$ in the space 
$L^2[0,1]$ of square integrable functions on the interval $[0,1]$. 
As we do not assume stationarity, the mean function
$t \to   \mathbb{E} [X_{i,n} (t) ] $
is changing with $i$ and we assume that it is given by   $    \mathbb{E} [X_{i,n} (t) ] = m(\tfrac{i}{n},t) $, where $m$ is a smooth function on the unit square. Our goal is  the construction of simultaneous confidence 
surfaces (SCSs) for the (time dependent) mean function $(u,t) \to m (u,t)$ of the locally stationary functional time series $\{X_{i,n} \}_{i=1,\ldots,n}$.
As an illustration we display in	Figure \ref{Fig-Introduction}  the implied volatility of an SP$500$ index  as a function of moneyness ($t$)  at different times to maturity ($u$, which is scaled to the interval $[0,1]$). These functions are quadratic and known as ``volatility smiles'' in the literature on option pricing. They seem to slightly vary  in time. In practice, it is important to assess whether these ``smiles'' are time-invariant.    We refer the interested reader to Section \ref{sec3} for a more detailed discussion (in particular we construct there  a confidence surface for the function $(u,t) \to m (u,t)$).
\begin{figure}
	\centering
	\includegraphics[width=12cm,height=6cm]{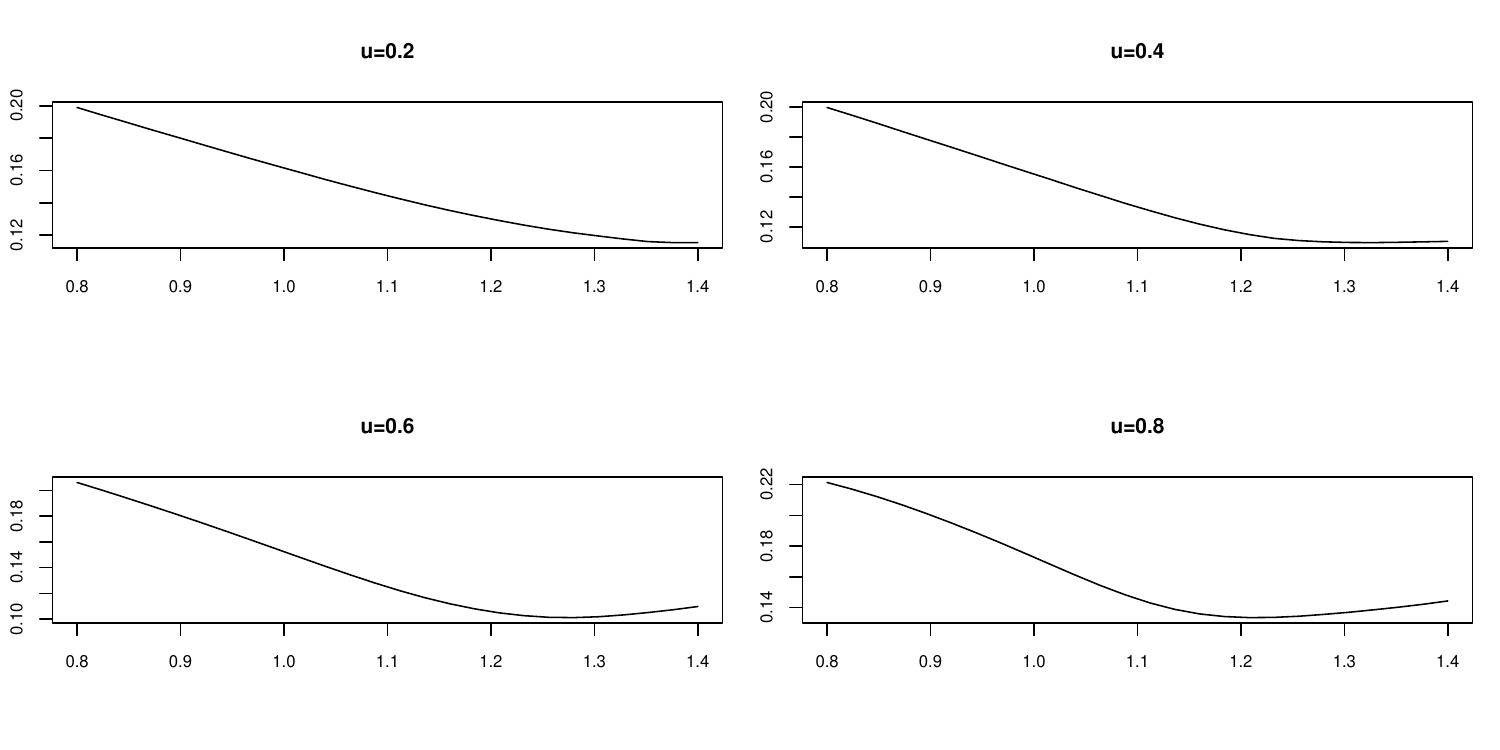}
	\caption{ \it 
		Volatility Smile at different (one minus) times to maturity ($u =0.2$, $0.4$, $0.6$ and $0.8$). The x-axis corresponds to 'moneyness'. 
	}
	\label{Fig-Introduction}
\end{figure}
To our best knowledge  confidence bands have only been considered in the stationary case, where $m(u,t) = m(t)$. Under the 
assumption of stationarity 
they can be constructed using the sample 
mean $\bar X_n = \frac{1}{n} \sum_{i=1}^n X_{i,n}  $ and the weak convergence of 
$\sqrt{n} (\bar X_n - m) $ to a centered Gaussian process \citep[see, for example][who either assume that data is observed on a dense  grid or that the full trajectory can be observed]{degras2011,caoyangtodem2012,degras2017,dettekokotaue2020}.
More recently, alternative 
simultaneous  confidence (asymptotic)  bands have been 
constructed by \cite{lieblreim2019,telschwartz2022}  using  the  Gaussian Kinematic formula.  

On the other hand, 
although non-stationary 
functional time series have found considerable interest in the recent literature 
\citep[see, for example,][]{anne2017,delftaue2020,buchdetthein2020,kurisu2021nonparametric,kurisu2021estimation,anneholger}, the problem 
of constructing a confidence surface for the mean function has not been considered in the literature so far.
A potential  explanation that  a solution is still not available, consists in the fact that 
due to 	non-stationarity 
smoothing is required to estimate the function 
$u \to m(u,t)$ (for a fixed $t$).
This results  in an estimator  converging with a   $1/\sqrt{b_n n}$ rate (here $b_n$ denotes a bandwidth).
On the other hand, in the stationary case (where $m$ 
does not depend on $u$),
the sample mean 
$ \bar X_n $ can be used, resulting  in a $1/\sqrt{n}$ rate.

As a consequence, a weak convergence result for  the 
sample mean in the non-stationary case is not available and  
the construction of SCSs 
for the regression function $(u,t)\to m(u,t)$  is challenging.
In this paper we propose a general solution to this problem, which is not based on weak convergence results.
As an alternative  to ``classical'' limit theory  (for which it is not clear if it exists in the present situation) 
we develop  Gaussian  approximations for the maximum absolute deviation  between the estimate and  
the  regression function.  These results are then  used to construct  a non-standard multiplier bootstrap procedure  
for the construction  of SCSs for the mean function of a locally stationary  functional
time series. 
Our approach is based on approximating 
the maximal absolute deviation $\hat  \Delta = \max_{u,t \in [0,1]} |\hat m_{b_n}(u,t) - m (u,t) | $ by a maximum taken over a discrete grid, which becomes dense with increasing sample size. 
We thus  relate $\hat  \Delta $ to the maximum norm of a {\bf sparse} high-dimensional vector.
We then further develop Gaussian approximations for the maximum norm of {\bf sparse}  high-dimensional random vectors based on the methodology  proposed by  \cite{chernozhukov2013gaussian} \cite{zhang2018gaussian}. Finally, the covariance structure of this vector 
(which is actually a high-dimensional long-run variance) is mimicked  by a multiplier bootstrap.  Our approach is non-standard in the following sense: due to the sparsity, the Gaussian approximations in the cited literature cannot be directly used. In order to make these applicable we reduce the dimension by deleting vanishing entries. However, this procedure produces additional spatial dependencies (besides the dependencies induced by time series), such that the common multiplier bootstrap is not applicable. Therefore we propose a novel multiplier bootstrap, where instead of the full vector, individual 
blocks of the vector are multiplied with independent random variables,
such that  for different vectors a certain amount (depending on the lag) of these multipliers coincide. { Our proposed Gaussian approximation and the bootstrap scheme are suitable for nonparametric inference of means of high-dimensional  time series. In fact, we first discretize our functional time series to nonstationary high-dimensional time series  and then utilize the above-mentioned Gaussian approximation and bootstrap scheme for simultaneous inference.}

The remaining part  of the paper is organized as follows. The statistical model is introduced in Section \ref{sec2}, where we
describe our approach in an informal way and  propose two 
confidence surfaces  for the mean function $m$.
Section  \ref{sec4} is devoted to   rigorous statements under which 
conditions our method  provides valid (asymptotic) confidence 
surfaces.  
As a by-product of our approach, we also derive in the online supplement of this paper
new confidence bands for the functions $t \to m(u,t)$ (for fixed $u$) and $u \to m(u,t)$ (for fixed $t$), which provide efficient alternatives to the commonly used confidence bands for stationary functional data or real-valued locally stationary data, respectively
(see Section \ref{secA} for details).
\WCr{Although our main 
focus  is on  SCSs excluding the boundary (as most work in the literature does), we also 
provide  -  as  a complement -  simultaneous inference at the boundary, which  is of independent interest, see  Remark \ref{Boundary-Remark}.}  
In Section \ref{sec3} 
we demonstrate  the usefulness of our approach by means of analyzing a data example.
\WCr{Finally, all technical results are deferred to  the online supplement. There, we also give   remarks regarding noisy and multivariate locally stationary functional time series and  some concrete examples of locally stationary functional time series. Moreover,  the online supplement  contains  implementation details  and  a  simulation study illustrating the finite 
sample properties of the  asymptotic results.}

\section{ SCSs for  non-stationary time series} 
\label{sec2}
\def\theequation{2.\arabic{equation}}
\setcounter{equation}{0}

Throughout this paper  we consider the model 
\begin{align} \label{1.1}
	X_{i,n} (t) =  m(\tfrac{i}{n},t) + \varepsilon_{i,n}(t) ~,~~i=1, \ldots , n ~,
\end{align}
where $(\varepsilon_{i,n})_{i=1, \ldots ,n} $ is a centered locally stationary process in	$L^2[0,1]$ of square integrable functions on the interval $[0,1]$
(see Section \ref{sec4} for a precise mathematical  definition)
and $m: [0,1] \times [0,1] \to \R $ is a smooth   mean function. 
This means that at each time point ``$i$'' we observe  a function $t \to X_{i,n} (t)$.  

\WCr{
Let $\mathcal C^{a,b}$ denote  the set  of  functions $f: [0,1]^2 \to \mathbb{R} $, which  are  $a$-times and $b$-times  partially differentiable with respect to  the first and second coordinate, respectively, 
 such that for fixed $t$  and for fixed $u$ the functions 
$u \to \frac{\partial^a}{\partial u^a}f(u,v)$ and $t \to \frac{\partial^b}{\partial t^b}f(u,t)$ are Lipschitz continuous with a uniformly bounded Lipschitz constant on the interval  $[0,1]$.   
In this paper, 
 we are interested in a  SCS 
\begin{equation}
	\label{2.3}
	{\cal C}_n =     \big \{ f \in \mathcal C^{3,0}  ~| ~~ \hat L_1 (u,t) \leq f(u,t)  \leq  \hat U_1 (u,t) ~~\forall u, t ~ \big \} 
\end{equation}}
for the mean function $(u,t) \to m(u,t)$,
where $\hat L_1$ and $\hat U_1$  are  appropriate lower and upper bounds calculated from the data. 
The methodology developed in this paper  allows the construction
of SCSs for the functions 
$t \to m(u,t)$ (for fixed $u$) and $u \to m(u,t)$  (for fixed $t$),
which 
are developed in Section \ref{secA} of the online supplement for the sake of completeness.
Moreover,  in the main part of this paper we investigate  SCSs 
for  $u\in[b_n,1-b_n]$. SCSs on the boundary
can be constructed in a similar spirit,  even though their form is distinct from their interior counterparts, see Remark \ref{Boundary-Remark}.

Our approach is based on the maximum deviation
\begin{align}  \label{m2}
	\hat \Delta_n = 
	\sup_{t,u}|\hat m_{b_n}(u,t)-m(u,t)|~,
\end{align}
where   for $u\in [b_n,1-b_n]$
\WCr{\begin{align} 	\label{estimating-m} 
	\hat m_{b_n}(u,t) & =\sum_{i=1}^n X_{i,n}(t)K\Big(\frac{\tfrac{i}{n}-u}{b_n}\Big)/\sum_{i=1}^nK\Big(\frac{i/n-u}{b_n}\Big)
\end{align}
denotes the  common Nadaraya-Watson estimate with kernel $K$.  For 
$u\in [0,b_n)$ and for $u\in (1-b_n,1]$  
we use boundary kernels, say  $K_l(\cdot)$ and $K_r(\cdot)$, respectively, in the definition of $\hat m_{b_n}(u,t)$. 
Throughout this paper, we make  the following assumptions on these kernels.
\begin{assumption}
\rm 
\label{asskern}
The kernel $K(\cdot)$ is  symmetric continuous, supported on the interval $[-1,1]$ and satisfies $\int_{\mathbb R}  K(x)dx=1$,  $\int_{\mathbb R}  K(v)v^2dv=0$ and $\int_{\mathbb R} K(v)v^4dv>0$. 
$K_l$ 
(used for estimation on  $[0,b_n)$)  is supported on the interval $[0,1]$ and satisfies $K_l(0)=K_l(1)=0,$ $ \int K_l(x)x^jdx=0$ for $j=1,2$, $\int K_l(x)dx=1, \int K_l(x)x^3dx>0$. Additionally, both $K$ and $K_l$ are twice differentiable on their support, respectively and  $K''$ and $K_l''$ are Lipschitz continuous. 
The kernel $K_r$ (used  for estimation  on  $(1-b_n, 1]$) is given by  $K_r(x)=K_l(-x)$.
\end{assumption}} 

\WCr{ 
As a consequence of Assumption \ref{asskern}, the bias of  the  estimator \eqref{estimating-m} is of order 
 $O(b_n^4)$ in  the interval $[b_n,1-b_n]$,  and of order $O(b_n^3)$ at the region $[0,b_n)$ and $(1-b_n,1]$ if the function  $u \to  \frac{\partial^3}{\partial u^3}m(u,v)$ is Lipschitz continuous with bounded Lipschitz constant. As alternative one could consider the local  polynomial regression estimate of order $3$. As this will make the  theoretical analysis  even  more technical,  we leave the study of the statistical properties of  local polynomial estimator of higher order for a  $2$-dimensional mean function of a non-stationary functional time series for  future work.}

\subsection{Confidence  surfaces with fixed width}
\label{sec21} 

Note that under smoothness assumption the deterministic term in
\eqref{estimating-m}  approximates $m(u,t)$.  
For an increasing sample size $n$, we can 
approximate the maximum deviation on a discrete grid,  
i.e.  
\begin{align}
	\nonumber
	\hat \Delta_{b_n} &:=& 
	\max_{\substack{ b_n \leq u\leq 1-b_n,\\
			0\leq t\leq 1}} \sqrt {nb_n}  | \hat \Delta(u,t) |  \approx 
	\max_{\substack{
			\lceil nb_n\rceil\leq l\leq n-\lceil nb_n \rceil	\\  
			1\leq k\leq p }
	}
	\sqrt{nb_n} |	\hat \Delta(\tfrac{l}{n},\tfrac{k}{p})|  
	\\
	\label{delta1}
	&			\approx  & 
	\max_{\substack{
			\lceil nb_n\rceil\leq l\leq n-\lceil nb_n \rceil	\\  
			1\leq k\leq p 
	} }
	\Big |\frac{1}{\sqrt {nb_n}}\sum_{i=1}^n  \varepsilon_{i,n} (\tfrac{k}{p})		K\Big(\frac{\tfrac{i}{n}-\tfrac{l}{n}}{b_n}\Big) \Big | , ~~~~~~~
\end{align}
where $p$ is increasing with $n$  as well. 
Therefore, the bootstrap procedure
will be based on a  Gaussian approximation of the right hand side of \eqref{delta1}, which is the maximum norm of high-dimensional sparse vector. 
In this section our approach 
will be stated in a rather informal way, rigorous  statements can be found in Section \ref{sec4}. 

To be precise, define for $1\leq i\leq n$ the $p$-dimensional vector
\begin{align} \label{3.1}
	Z_{i}(u) & = (Z_{i,1}(u), \ldots , Z_{i,p}(u))^\top \\
	\nonumber 
	& = K \Big (\frac{\tfrac{i}{n}-u}{b_n} \Big ) 	\Big (\varepsilon_{i,n} (\tfrac{1}{p})
	,\varepsilon_{i,n} (\tfrac{2}{p})
	,\ldots ,\varepsilon_{i,n} (\tfrac{p-1}{p}), \varepsilon_{i,n} (1)
	\Big )^\top ~,
\end{align}
where  $K(\cdot)$ and $b_n$ are the interior kernel and bandwidth used in the estimate \eqref{estimating-m}, respectively.
Next we  define the $p$-dimensional vector  
\begin{equation}
	Z_{i,l}=Z_i(\tfrac{l}{n}) = (Z_{i,l,1}, \ldots Z_{i,l,p})^\top~,
	\label{hol20}  
\end{equation}
where  
$$
Z_{i,l,k} = \varepsilon_{i,n} (\tfrac{k}{p}) K \Big (\frac{\tfrac{i}{n}-\tfrac{l}{n}}{b_n}
\Big ) ~~~~
(1\le k\le p)~.
$$
Note that, by \eqref{delta1}, 
\begin{align} \label{delta11}
\hat \Delta_{b_n} 
\approx 
\max_{\substack{	\lceil nb_n\rceil\leq l\leq n-\lceil nb_n \rceil	\\  
	1\leq k\leq p }}
\Big | 	\frac{1}{\sqrt{nb_n}}  \sum_{i=1}^n Z_{i,l,k} \Big | 	\approx \Big | 
\frac{1}{\sqrt{nb_n}} 
\sum_{i=1}^n(Z_{i,{\lceil nb_n\rceil }}^\top,\ldots  ,Z_{{i,n-\lceil nb_n\rceil}}^\top)^\top \Big |_\infty~,
\end{align}
where $ | a |_\infty $ denotes the maximum norm of a finite dimensional vector $a$ 
(the dimension will always be clear from context).
The 	entries in  the vector $Z_{i,l}$ are  zero whenever  ${|i-l|}/{(nb_n)} \geq 1$. Therefore, the high-dimensional vector $(Z_{i,\lceil nb_n \rceil}^\top,\ldots ,Z_{i,n-\lceil nb_n \rceil}^\top)^\top$ is sparse and  common Gaussian approximations for its  maximum  norm \citep[see, for example,][]{chernozhukov2013gaussian}, \cite{zhang2018gaussian}  are not  applicable.

To address this issue 
we reconstruct high-dimensional vectors, say $\tilde Z_j$,  by eliminating   vanishing entries in the  vectors  $Z_{i,l}$  and rearranging  the nonzero ones. 
While this approach is very natural it produces additional dependencies, which require a substantial modification  of the common multiplier bootstrap 
as considered, for example, in
\cite{zhou2010simultaneous}, \cite{zhou2013heteroscedasticity}, \cite{karmakar2021simultaneous} or  \cite{mies2021functional}  for 
(low dimensional) locally stationary  time series. More precisely, we define
the $(n-2\lceil nb_n\rceil +1)p$-dimensional vectors
$\tilde Z_1 , \ldots , \tilde {Z}_{2\lceil nb \rceil-1}$  by 
\begin{align}\label{tildeZi}
\tilde Z_i= \big (
Z_{i,\lceil nb_n\rceil}^\top,
Z_{i+1,\lceil nb_n\rceil+1}^\top, 
\ldots ,
Z_{n-2\lceil nb_n\rceil+i , n-\lceil nb_n \rceil  }^\top
\big)^\top  ~.
\end{align}
We also put $\tilde Z_{2\lceil nb_n\rceil }=0$ and note 
that 
\begin{align}\label{tilde-Zi-max}
\Big |  \frac{1}{\sqrt{nb_n}} 
\sum_{i=1}^n \bar Z_i \Big |_\infty
= \Big  |\frac{1}{\sqrt{nb_n}}\sum_{i=1}^{2\lceil nb_n\rceil-1 }\tilde Z_i \Big  |_\infty ,
\end{align} 
where $\bar Z_i:=(Z_{i,\lceil  nb_n \rceil }^\top, Z_{i,\lceil  nb_n \rceil+1 }^\top, \ldots ,Z_{i,n-\lceil  nb_n \rceil}^\top)^\top$.
Note that the right hand side of \eqref{tilde-Zi-max} is a sum of the $(n-2\lceil nb_n\rceil-1)p$ dimensional vectors
\begin{equation}
\begin{split}
\label{rep1}
\tilde Z_1&=K\big (\tfrac{1-\lceil nb_n\rceil}{nb_n}\big )(\vec{ \varepsilon}_1,\vec \varepsilon_2, \ldots ,\vec \varepsilon_{n-2\lceil nb_n\rceil+1})^\top,\\
\tilde Z_2&=K\big(\tfrac{2-\lceil nb_n\rceil}{nb_n}\big)(\vec{ \varepsilon}_2,\vec \varepsilon_3, \ldots ,\vec \varepsilon_{n-2\lceil nb_n\rceil+2})^\top,\\
&\vdots\\
\tilde Z_{2\lceil nb_n\rceil-1}&=K\big (\tfrac{\lceil nb_n\rceil-1}{nb_n}\big )
\big (\vec{ \varepsilon}_{2\lceil nb_n\rceil-1},\vec \varepsilon_{2\lceil nb_n\rceil}, \ldots ,\vec \varepsilon_{n-1} \big )^\top,
\end{split} 
\end{equation}
where $\vec{\varepsilon}_i=(\varepsilon_{i,n}(\tfrac{1}{p}), \ldots ,\varepsilon_{i,n}(\tfrac{p}{p})).$   On the other hand the left hand side of \eqref{tilde-Zi-max} is a sum of the {\bf sparse}
vectors 
\begin{equation}
\label{rep1a}
\begin{split}
\bar Z_1&=\Big({K\big (\tfrac{1-\lceil nb_n\rceil}{nb_n}\big )\vec \varepsilon_1}~,~~~~~~~~~0~~~~~~~~~~,~0~,~\ldots  ~,0~,~~~~~~~~~~~~~ 0~~~~~~~~ \Big)^\top,\\
\bar Z_2&=\Big(K\big (\tfrac{2-\lceil nb_n\rceil}{nb_n}\big )\vec \varepsilon_2 ~,~{K\big (\tfrac{1-\lceil nb_n\rceil}{nb_n}\big )\vec \varepsilon_2}~,~0~,~\ldots ~,~0~,~~~~~~~~~~~~~ 0~~~~~~~~\Big)^\top,\\
& \vdots\\
\bar Z_{n-1}&=\Big(~~~~~~~~~0~~~~~~~~ , ~~~~~~~~~~0 ~~~~~~~~~,~ 0~,~\ldots
~,~0~,~K\big (\tfrac{\lceil nb_n\rceil-1}{nb_n}\big )\vec\varepsilon_{n-1}\Big)^\top.
\end{split}
\end{equation}
Although, the vectors on both sides of \eqref{tilde-Zi-max} are very different, and the  number of terms in the sum is different, the  non-vanishing elements over which 
the maximum is taken on both sides  coincide. We note that this transformation yields some computational advantages and, even more important,  it allows the development of a Gaussian approximation and a corresponding multiplier bootstrap, which is explained next.

To be precise, observing \eqref{delta11},
we see that  the right hand side of \eqref{tilde-Zi-max} is an  approximation of the maximum  absolute deviation 
$	\max_{u,t} \sqrt {nb_n} | \hat \Delta(u,t) |$.  In Theorem \ref{SCB} in Section \ref{sec42} 
we will show that  the vectors 
$\tilde Z_1 , \ldots , \tilde {Z}_{2\lceil nb \rceil-1}$  in \eqref{tilde-Zi-max} can
be replaced by Gaussian vectors. More precisely we prove the existence of  
$(n-2\lceil nb_n\rceil +1)p$-dimensional centered Gaussian vectors
$\tilde Y_1, \ldots , \tilde Y_{2\lc nb_n\rc-1}$  with the same auto-covariance structure as  the vector $\tilde Z_i$ in \eqref{tildeZi} such that
\begin{align} 	
\sup_{x\in \mathbb R} \Big |\p \Big ( 
\max_{\substack{b_n \leq u\leq 1-b_n \\ 0\leq t\leq 1}}\sqrt{nb_n}|\hat \Delta(u,t)|\leq x \Big )-\p 
\Big (\Big |\frac{1}{\sqrt{nb_n}}\sum_{i=1}^{2\lc nb_n\rc-1}&\tilde Y_i \Big  |_\infty\leq x\Big )\Big |
\notag\\&= o(1)\label{det1a}
\end{align}
if  $ p $ is an appropriate sequence converging  to infinity  with the  sample size 
(for example, $p= \sqrt{n}$).

The estimate \eqref{det1a}	is the basic tool for the  construction of a SCS  for the regression function $m$. 
For its application it is necessary to generate  Gaussian random vectors $\tilde Y_i$ with the same auto-covariance structure as  the vector  $\tilde Z_i$ 	in \eqref{tildeZi}, which is not trivial.   To see this, note
that   the common multiplier bootstrap approach for approximating the distribution of $\frac{1}{\sqrt{nb_n}}\sum_{i=1}^{2\lceil nb_n\rceil-1}\tilde Z_i$ replaces 
the  $\tilde Z_i$ by 
block sums multiplied with independent 
random variables, such as
$R_i \sum_{s=i}^{i+m}\tilde Z_s/\sqrt{m} $ 
\citep[see][]{zhang2014bootstrapping}
or  $R_i\sum_{s=i}^{i+m}(Z_s-\frac{1}{2\lceil nb_n\rceil-1}\sum_{s=1}^{2\lceil nb_n\rceil-1}\tilde Z_i)/\sqrt{m}$
\citep[see][]{zhou2013heteroscedasticity},
where $R_1, \ldots ,R_{2\lceil nb_n\rceil}$ are independent standard normally distributed random variables.
However, this would not yield to valid approximation due to the additional dependencies between $\tilde Z_1, \ldots ,\tilde Z_{2\lceil nb_n\rceil-1}.$   As an alternative we therefore propose a multiplier bootstrap, which also mimics this dependence structure by multiplying $p$-dimensional blocks of block sums of  $\tilde Z_i$ by standard normally distributed random variables,
which reflects the specific dependencies of these vectors. In other words the vectors
$\tilde Z_1, \tilde Z_2, \tilde Z_3, \ldots $  in \eqref{rep1} are replaced by 
\begin{equation}
\begin{split}
&K\big (\tfrac{1-\lceil nb_n\rceil}{nb_n}\big )(\vec{ \varepsilon}_{1:1+m_n} R_1,\vec \varepsilon_{2:2+m_n} R_2 , \ldots ,\vec \varepsilon_{n-2\lceil nb_n\rceil+1:n-2\lceil nb_n\rceil+1+m_n } R_{n-2\lceil nb_n\rceil+1} )^\top,\\
&K\big(\tfrac{2-\lceil nb_n\rceil}{nb_n}\big)(\vec{ \varepsilon}_{2:2+m_n} R_2,\vec \varepsilon_{3:3+m_n} R_3 , \ldots ,\vec \varepsilon_{n-2\lceil nb_n\rceil+2 : n-2\lceil nb_n\rceil+2 +m_n } R_{n-2\lceil nb_n\rceil+2})^\top,\\
&K\big(\tfrac{3-\lceil nb_n\rceil}{nb_n}\big)(\vec{ \varepsilon}_{3:3+m_n} R_3,\vec \varepsilon_{4:4+m_n} R_4 , \ldots ,\vec \varepsilon_{n-2\lceil nb_n\rceil+3 : n-2\lceil nb_n\rceil+3 +m_n } R_{n-2\lceil nb_n\rceil+3})^\top,\\
& ~~~~~ ~~~~~ ~~~~~ ~~~~~ ~~~~~ ~~~~~ ~~~~~ ~~~~~   \vdots
\end{split} 
\label{mimic}
\end{equation}
respectively, where  
\begin{align}
\label{det41}
\vec \varepsilon_{j:j+m_n}=\frac{1}{\sqrt m_n} \sum_{r=j}^{j+\lf m_n/2\rf -1}\vec{\varepsilon}_r- 
\frac{1}{\sqrt m_n}\sum_{r=j+\lf m_n/2\rf }^{j+2\lf m_n/2\rf -1}\vec{\varepsilon}_r.
\end{align}
Here we consider local block sums (of increasing length) to mimic the dependence structure of the error process.
A difference of local block sums is used 
to mitigate the effect of bias  if all elements of the 
unknown errors $\vec{ \varepsilon}_i$ are replaced by  corresponding nonparametric residuals $\hat \varepsilon_{i,n}(s/p)$, $1\leq s\leq p$ where
\WCr{
$	\hat \varepsilon_{i,n}(t)=X_{i,n}(t)-\hat m_{b_n} (\tfrac{i}{n},t)$,
  where $\hat m_{b_n}$ is defined 
 \eqref{estimating-m}. 
We emphasize that the use of  boundary kernels  in this estimate (see Assumption \ref{asskern}) allows us to construct SCSs for  the boundary region as well. Details are given in  Remark \ref{Boundary-Remark} below.}
With the residuals we define the $p$-dimensional vector\begin{align}
\label{3.10}
\hat Z_{i}(u)
& = (\hat Z_{i,1}(u), \ldots , \hat Z_{i,p}(u))^\top \\
\nonumber 
& = K \Big (\frac{\tfrac{i}{n}-u}{b_n} \Big ) \big  (\hat \varepsilon_{i,n}(\tfrac{1}{p})
,\hat \varepsilon_{i,n}(\tfrac{2}{p})
,\ldots ,\hat \varepsilon_{i,n} (\tfrac{p-1}{p}),
\hat \varepsilon_{i,n}(1)
\big )^\top
\end{align}	
as an analog of \eqref{3.1}. Similarly, we define the 
analog of  \eqref{tildeZi} by
\begin{align}\label{3.11}
\hat{\tilde Z}_j= \big (
\hat Z_{j,\lceil nb_n\rceil}^\top,
\hat	Z_{j+1,\lceil nb_n\rceil+1}^\top, 
\ldots ,
\hat	Z_{n-2\lceil nb_n\rceil+j, n-\lceil nb_n \rceil  }^\top
\big)^\top ,
\end{align}
where 
$
\hat Z_{i,l}= 	\hat Z_i(\tfrac{l}{n}) = (	\hat Z_{i,l,1}, \ldots 	\hat Z_{i,l,p})^\top~.
$ 
Note that we have replaced $Z_{i,l}$ in \eqref{tildeZi} by  $\hat Z_{i,l}$, which can be calculated from the data.
These vectors will be used 
in Algorithm \ref{algorithm1} to 
define empirical versions of the vectors in \eqref{mimic}, which then mimic the dependence structure of the 
vectors $\tilde Y_1, \ldots , \tilde Y_{2\lc nb_n\rc-1}$ in the Gaussian approximation  \eqref{det1a}
(see equations \eqref{hol10} and \eqref{3.13} in Algorithm \ref{algorithm1} ).
The SCS   for the mean function $m$ is finally defined by
\begin{equation} \label{n02}
{\cal C}_n =     \big \{ f\in \mathcal C^{3,0} : [0,1]^2 \to \R | ~ \hat L_1 (u,t) \leq f(u,t)  \leq  \hat U_1 (u,t) ~\forall   u \in  [b_n,1-b_n] ~\forall   t \in  [0,1] \big \}
~,~~~			\end{equation}
where the definition of functions $\hat L_1,  \hat U_1 : [0,1]^2 \to \mathbb{R}  $ is given in  Algorithm  
\ref{algorithm1}. Finally, Theorem \ref{Thm2.1-2021} in Section \ref{sec4}  shows that ${\cal C}_n$ defines a valid asymptotic 
$(1-\alpha)$ confidence surface for the regression function $m$ in model \eqref{1.1}.

\begin{remark}\label{Remarksmooth} 
\textcolor{blue}{In this paper we assume that at each time point  the full trajectory is observed. Therefore, smoothing with respect to the variable $t$ is not necessary.   Smoothing with respect to both variables  becomes necessary if the trajectories are observed with measurement error.
 Our method is also applicable to dense and discrete observations from the trajectory. In these cases   smoothing with respect  to the variable $t$  yields  a further  bias.
 Another scenario when smoothing is important is the situation where the trajectory is observed at sparse discrete points. This case  is beyond the scope of our paper because it requires a different theory.
 }
\end{remark}

\begin{algorithm}
%

\begin{itemize} 
\item[(a)]  Calculate  the $(n-2\lceil nb_n\rceil +1)p$-dimensional vectors   $\hat{\tilde Z}_i$
in \eqref{3.11} 
\item[(b)] For window size $m_n$, 
let $m_n'=2\lfloor m_n/2\rfloor$, define
the   vectors 
\begin{align}	\label{hol10}
	\hat S_{jm_n'}= \frac{1}{\sqrt{m_n'}} \sum_{r=j}^{j+\lfloor m_n/2\rfloor-1 } 
	\hat{\tilde  Z}_{r} - \frac{1}{\sqrt{m_n'}}
	\sum_{r=j+ \lfloor m_n/2\rfloor}^{j+ m_n^\prime -1 }  \hat{\tilde  Z}_{r}
\end{align}

and denote by $\hat{\vec \varepsilon}_{j:j+m_n',k}$
the
$p$-dimensional sub-vector of the vector $\hat S_{jm_n'}$  in \eqref{hol10}
containing its $(k-1)p+1${st }  $-$  $kp${th}
components.
\item[(c)]   \textbf{For} r=1, \ldots , B, \textbf{do}
\smallskip
Generate i.i.d. $N(0,1)$ random variables $\{R^{(r)}_i\}_{i=1, \ldots  , n-m_n'}$.
Calculate 
\begin{align}  \label{3.13}
	T_k^{(r)}& =\sum_{j=1}^{2\lceil nb_n \rceil-m_n'} 
	\hat{\vec \varepsilon}_{j:j+m_n',k}
	R^{(r)}_{k+j-1} ~,~~ 1\leq k\leq n-2\lceil nb_n\rceil+1  , \\
	T^{(r)} & =
	\max_{1\leq k\leq n-2\lceil nb_n\rceil+1} |T_k^{(r)}|_\infty.
	\nonumber 
\end{align} 
\textbf{end}

\smallskip
\item [(d)]
Define $T_{\lf (1-\alpha)B\rf}$ as the empirical $(1-\alpha)$-quantile of the bootstrap sample
$T^{(1)}, \ldots , T^{(B)}$ and  
\begin{align*}
	\hat   L_1 (u,t)   = \hat m_{b_n}(u,t) - \hat r_1 ,~~~~
	\hat  U_1 (u,t) =  \hat m_{b_n}(u,t) + \hat r_1
\end{align*}
where 
$$
\hat r_1 =\frac{\sqrt 2 T_{\lf (1-\alpha)B\rf} }{\sqrt{nb_n}\sqrt{ 2\lceil nb_n\rceil-m'_n}.
} 
$$
{\bf Output:}  SCS \eqref{n02} 
for the mean function $m$.
\end{itemize}

\caption{}\label{algorithm1}
\end{algorithm}


\subsection{Confidence  surfaces with varying width}
\label{sec22}

The confidence surface in 
Algorithm \ref{algorithm1} has  a constant width and does not reflect the variability of the estimate $\hat m$ at the point $(u,t)$. In this section we will 
construct a SCS  adjusted by	an estimator of the   long-run variance  (see  equation \eqref{lrv}  in Section \ref{sec41}  for the exact definition).  Among others, this approach has been proposed by \cite{degras2011} and \cite{zhengyanghardle2014}   for
repeated measurement data from independent subjects where a  variance estimator is used for standardization. It has also been considered by \cite{zhou2010simultaneous}
who derived a simultaneous confidence tube for the parameter of a
time varying coefficients linear model with a (real-valued) 
locally stationary error process. 
In the situation of non-stationary functional data
as 	considered here this task is challenging as
an estimator of  the long-run variance is required, which is  uniformly consistent on the square $[0, 1]^2$.   

In  order to define such an estimator  let $H$  denote the Epanechnikov kernel   and define for some bandwidth $\tau_n\in(0,1)$ the weights 
$$
\bar \omega(t,i)=H\Big(\frac{\tfrac{i}{n}-t}{\tau_n}\Big) \Big / \sum_{ {i}=1}^n  H\Big(\frac{\tfrac{i}{n}-t}{\tau_n}\Big).
$$ 
Let $S^X_{k,r}=\frac{1}{\sqrt r}\sum_{i=k}^{k+r-1}X_{i,n}$
denote the normalized partial sum of the data 
$X_{k,n}, \ldots , X_{k+r-1,n}$ (note that these are functions) 
and define for $w\geq 2$ 
\begin{align*} 
\Delta_j(t)=\frac{S^X_{j-w+1,w}(t)-S^X_{j+1,w}(t)}{\sqrt w}.
\end{align*}
An  estimator of the long-run variance (where the exact definition is in \eqref{lrv}) is then  defined 
by 
\begin{align}\label{2018-5.4}
\hat \sigma^2(u,t)=\sum_{j=1}^n\frac{w\Delta_j^2(t)}{2}\bar\omega(u,j),
\end{align} 
if $u\in[w/n,1-w/n]$.  
For $u\in[0,w/n)$ and   $u\in(1-w/n,1]$ we define it as 
$\hat \sigma^2(u,t)=\hat \sigma^2(w/n,t)$  and 
$\hat \sigma^2(u,t)=\hat \sigma^2(1-w/n,t)$, respectively.
We will show in Proposition
\ref{prop-lrv} in Section \ref{sec4} that this 
estimator is  uniformly  consistent.

To state the bootstrap algorithm for a SCS of the form \eqref{2.3}  with varying  width, we introduce the following notation
\begin{align*}
\hat Z^{\hat \sigma}_{i}(u)
& = (\hat Z^{\hat \sigma}_{i,1}(u), \ldots , \hat Z^{\hat \sigma}_{i,p}(u))^\top \\
\nonumber 
& = K \Big (\frac{\tfrac{i}{n}-u}{b_n} \Big ) \Big  (\frac{\hat \varepsilon_{i,n} (\tfrac{1}{p})}{\hat \sigma(\tfrac{i}{n},\tfrac{1}{p})}
,\frac{\hat \varepsilon_{i,n}(\tfrac{2}{p})}{\hat \sigma(\tfrac{i}{n},\tfrac{2}{p})}
,\ldots ,\frac{\hat \varepsilon_{i,n} (\tfrac{p-1}{p})}{\hat \sigma(\tfrac{i}{n},\tfrac{p-1}{p})},
\frac{\hat \varepsilon_{i,n}(1)}{\hat \sigma(\tfrac{i}{n},1)}
\Big )^\top  
\end{align*}
and consider the normalized   analog  
\begin{align}\label{3.11-sigma}
\hat{\tilde Z}_j^{\hat \sigma}= \big (
\hat Z_{j,\lceil nb_n\rceil}^{\hat \sigma, \top},
\hat	Z_{j+1,\lceil nb_n\rceil+1}^{\hat \sigma, \top}
\ldots ,
\hat	Z_{n-2\lceil nb_n\rceil+j, n-\lceil nb_n \rceil  }^{\hat \sigma, \top}
\big)^\top 
\end{align}
of the  vector 
$	\hat{\tilde Z}_j$ in \eqref{3.11}, 
where 
$
\hat Z^{\hat \sigma}_{i,l}= 	\hat Z^{\hat \sigma}_i(\tfrac{l}{n}) = (	\hat Z^{\hat \sigma}_{i,l,1}, \ldots 	\hat Z^{\hat \sigma}_{i,l,p})^\top~
$.
The SCS
with varying width for the mean function $m$ is then defined by
\begin{equation} 
\label{n1}
{\cal C}^{\hat \sigma}_n =     \big \{ f\in\mathcal C^{3,0} : [0,1]^2 \to \R ~| ~ \hat L^{\hat \sigma}_2 (u,t) \leq f(u,t)  \leq  \hat U^{\hat \sigma}_2 (u,t) ~~\forall   u \in  [b_n,1-b_n] ~\forall   t \in  [0,1] \big \}, 
\end{equation}		
where the functions $\hat L_2 $ and $\hat U_2 $ are
constructed in Algorithm  
\ref{algorithm3}. Theorem \ref{sigma-bootstrap} 
in Section \ref{sec42} shows that this  defines a valid asymptotic 
$(1-\alpha)$ confidence surface for the function $m$ in model \eqref{1.1}.

\WCr{We also emphasize that we can construct similar
SCSs for noisy and multivariate locally stationary functional time series; see  Remarks \ref{noisydata} and \ref{remark4.1-multiple} in the supplementary  material for details.}

\begin{algorithm}

\begin{itemize} 
\item[(a)] Calculate the the estimate of the long-run variance  $\hat \sigma^2$
in \eqref{2018-5.4}.

\item[(b)]  Calculate the
$(n-2\lceil nb_n\rceil +1)p$-dimensional 
vectors   $\hat{\tilde Z}^{\hat \sigma}_i$
in \eqref{3.11-sigma}.
\item[(c)] For window size $m_n$, 
let $m_n'=2\lfloor m_n/2\rfloor$, define 
\begin{align*}
	\hat S^{\hat \sigma}_{jm_n'}=\frac{1}{\sqrt{m_n'}}	\sum_{r=j}^{j+\lf m_n/2\rf -1 }\hat{\tilde  Z}^{\hat \sigma}_{r}-\frac{1}{\sqrt{m_n'}}\sum_{r=j+\lf m_n/2\rf}^{j+ m_n'-1 }\hat{\tilde  Z}^{\hat \sigma}_{r}
\end{align*} 
and denote by  
$\hat{\vec \varepsilon}^{\ \hat \sigma}_{j:j+m_n',k}$ be the	$p$-dimensional sub-vector of the vector $\hat S^{\hat \sigma}_{jm_n'}$ 
containing its $(k-1)p+1${st} $ - $ $kp${th}
components.

\item[(d)]  \textbf{For} {$r=1, \ldots , B$}  
\smallskip
\textbf{do} Generate i.i.d. $N(0,1)$ random variables $\{R^{(r)}_i\}_{i =1 , \ldots , n-m_n'}$.
~ Calculate 
\begin{align*} 
	T_k^{\hat \sigma,(r)}& =\sum_{j=1}^{2\lceil nb_n \rceil-m_n'}\hat{\vec \varepsilon}^{\ \hat \sigma}_{j:j+m_n',k} R^{(r)}_{k+j-1} ~,~~ k= 1 , \ldots  , n-2\lceil nb_n\rceil+1  , \\
	T^{\hat \sigma,(r)} & =
	\max_{1\leq k\leq n-2\lceil nb_n\rceil+1} |T_k^{\hat \sigma,(r)}|_\infty.
	\nonumber 
\end{align*} 
\textbf{end}

\smallskip
\item [(e)]
Define $T^{\hat \sigma}_{\lf (1-\alpha)B\rf}$ as the empirical $(1-\alpha)$-quantile of the  sample
$T^{\hat \sigma,(1)}, \ldots , T^{\hat \sigma,(B)}$ and 
\begin{align*}
	\hat   L^{\hat \sigma}_2 (u,t)  & = \hat m_{b_n}(u,t) -
	\hat r_2 (u,t),~~~~
	\hat  U^{\hat \sigma}_2 (u,t)  =  \hat m_{b_n}(u,t) +  \hat r_2 (u,t) ,
\end{align*}
where
$$
\hat r_2 (u,t) =  \frac{\hat \sigma(u,t)\sqrt 2 T^{\hat \sigma } _{\lf (1-\alpha)B\rf} }{\sqrt{nb_n}\sqrt{ 2\lceil nb_n\rceil-m'_n}}
$$
{\bf Output:} 
SCS 
\eqref{n1}  with varying width for
the mean function $m$.
\end{itemize}
\caption{}
\label{algorithm3}
\end{algorithm}

\section{Theoretical justification} \label{sec4}
\def\theequation{4.\arabic{equation}}
\setcounter{equation}{0}
In this section, we first present the locally stationary functional time series model for which 
the theoretical results of
this paper are derived (Section \ref{sec41}). 
We also describe under which conditions Algorithm \ref{algorithm1}
and \ref{algorithm3} provide valid asymptotic $(1-\alpha )$ confidence surfaces for the regression function $m$ in model \eqref{1.1} (Section \ref{sec42}  and \ref{sec43}).
Throughout this paper we use the notation  
\begin{equation*}
\Theta(a,b)=a\sqrt{1\vee\log ( (b/a))}
\end{equation*}
for positive constants $a,b$, and  
the notation $a\vee b$ denotes the maximum of the real numbers $ a$ and $b$.

\subsection{Locally stationary processes and physical dependence} \label{sec41}
We begin with an assumption for the mean function $m$ in model \eqref{1.1}.\WCr{\begin{assumption} 
{\rm 
	\label{assreg}  $m \in \mathcal C^{3,0}$.
}
\end{assumption}
}
In fact, in the  proof of Theorem \ref{sigma-asy}, we show that the difference  between $\hat m_{b_n}(u,t)$ and $m(u,t)$ can be uniformly approximated by a weighted
sum of the random variables  $\varepsilon_{1,n}(t) , \ldots , \varepsilon_{n,n}(t)$. As a consequence, an approximation of the form \eqref{delta1} for an increasing number of points $\{t_1,\ldots,t_p\}$ is guaranteed by an appropriate smoothness condition on the error process  $\{ \varepsilon_{i,n}(t) \}_{i=1, \ldots , n}$, which will be introduced next. 

\begin{assumption} 
{\rm 
	\label{asserr}
	The error process has the form
	\begin{equation*}
		\varepsilon_{i,n} (t)  =     G(\tfrac{i}{n},t,\FF_i)~,~~i=1, \ldots , n
	\end{equation*}
	where $\FF_i=( \ldots , \eta_{i-1},\eta_i)$, $(\eta_i)_{i\in \Z} $ is a sequence of independent identically distributed
	random variables in some measurable space $\mathcal S$ 
	and $G: [0,1] \times [0,1] \times {\mathcal S}^{\Z } \to \R$ denotes a filter with the following properties:
	\begin{description}
		\item (1) There exists a constant $t_0>0$ such that
		\begin{equation} \label{2.7a} 
			\sup_{u,t\in [0,1]}\E(t_0\exp(G(u,t,\FF_0)))<\infty.
		\end{equation}
		\item (2)  
		Let $(\eta'_i)_{i\in \mathbb N } $ denote a sequence of independent identically distributed random variables which is  independent of  but has the same distribution as $(\eta_i)_{i\in \mathbb Z}$. Define $\FF^*_i=( \ldots , \eta_{-1},\eta_0', \eta_{1}, \ldots ,  \eta_i)$ and consider for some $q> 2$ the  dependence measure
		\begin{align} \label{2.7}
			\delta_{q}(G, i)=\sup_{u,t\in [0,1]}\|G(u,t,\FF_i)-G(u,t,\FF^*_i)\|_{q}.
		\end{align}
		There exists a constant  $\chi\in(0,1)$ such that for $i\geq 0$
		\begin{align}\label{2.8}
			\delta_{q}(G, i)=O(\chi^i)~.
		\end{align}
		\item (3) For the same  constant $q$ as in (2)  there exists a positive constant $M$ such that 
		\begin{align*}
			\sup_{t\in[0,1],u_1,u_2\in [0,1]}\|G(u_1,t,\FF_i)-G(u_2,t,\FF_i)\|_q\leq M|u_1-u_2|.
		\end{align*}
		\item	(4)   The {\it long-run variance}
		\begin{align} \label{lrv}
			\sigma^2(u,t):=\sum_{k=-\infty}^\infty {\rm Cov} (G(u,t,\FF_0),G(u,t,\FF_k))
		\end{align}
		of the  process $(G(u,t,\FF_i))_{i\in \Z} $ satisfies 
		$$\inf_{u,t\in [0,1]} \sigma^2(u,t)>0.
		$$
	\end{description}	
}
\end{assumption}

\noindent
Assumption \ref{asserr}(2)  requires  that the dependence measure is  geometrically decaying.  Similar results
as presented in this  section can be obtained  under summability assumptions  with substantially more intensive mathematical arguments and complicated notation, see
Remark \ref{remrk3.1}(ii) in the supplemental material for some details. 
Assumption \ref{asserr}(3)  means that the locally stationary functional time series is smooth in $u$,  while the smoothness in $t$ is provided in the next assumption. They are crucial for constructing SCSs of the form \eqref{2.3}.

\begin{assumption}
\label{asssim}
{\rm 
	The filter $G$ in Assumption \ref{asserr}
	is differentiable with respect to  $t$.  If  
	$G_2(u,t,\FF_i)=\frac{\partial }{\partial t}G(u,t,\FF_i),$ $G_2(u,0,\FF_i)=G_2(u,0+,\FF_i)$, $G_2(u,1,\FF_i)=G_2(u,1-,\FF_i)$, we assume 
		that there exists a constant $q^*>2$ such that	for some $\chi \in (0,1)$ and $i\geq0$,
		$$
		\delta_{q^*}(G_2, i)=O(\chi^i).
		$$
}
\end{assumption}

In the online supplement we present several examples of locally stationary functional time series satisfying these assumptions (see Section \ref{secex}).

\subsection{Theoretical analysis of the methodology in Section \ref{sec21}  } \label{sec42}

The bootstrap methodology introduced in Section \ref{sec2} is based on the  
Gaussian approximation \eqref{det1a}, which will be stated rigorously in Theorem \ref{SCB}.
Theorem \ref{Thm2.1-2021}  shows under which conditions
the confidence surface  \eqref{n02} has asymptotic level $(1-\alpha)$.

\begin{theorem}[Justification of Gaussian approximation \eqref{det1a}]
\label{SCB}
Let Assumptions \ref{asskern}, \ref{assreg} -   \ref{asssim} be satisfied and  assume that   $n^{1+a}b_n^9=o(1)$, $n^{a-1}b_n^{-1}=o(1)$ for  some $0<a<4/5$.
Then there exists  a sequence of centered 
$(n-2\lceil nb_n\rceil +1)p$-dimensional centered Gaussian vectors
$\tilde Y_1, \ldots , \tilde Y_{2\lc nb_n\rc-1}$  with the same auto-covariance structure as  the vector $\tilde Z_i$ in \eqref{tildeZi} such that  the distance $\mathfrak{P}_n$ defined in \eqref{det1a} satisfies 
\begin{align*}
\mathfrak{P}_n& 
=O\Big ((nb_n)^{-(1-11\iota)/8} +\WCr{\Theta \Big (\sqrt{nb_n}(b_n^4+\frac{1}{n}), np \Big )} \\
& ~~~~~~~~~~~~~~~~~~~~~~~~~~
+\Theta \big(((np)^{1/q^*}((nb_n)^{-1}+1/p))^{\frac{q^*}{q^*+1}}, n p\big) \Big )
\nonumber 
\end{align*}	
for any sequence $ p \to \infty$ with 
$np = O (\exp(n^{\iota}))$ 	 for some $0\leq \iota <1/11$. 
In particular, for the choice  $p=n^c$ with $c>0$ we have 
$$
\mathfrak{P}_n= o(1)
$$ 
if the constant $q^*$ in   Assumption \ref{asssim}  is sufficiently large.
\end{theorem}

\WCr{In Section \ref{CheckThm1}  of the supplementary material we investigate  the finite sample properties of the approximation in  Theorem \ref{SCB} by means of a simulation study.}  Moreover, Theorem \ref{SCB} is the main ingredient to prove the validity of the bootstrap SCS $ \mathcal C_n$ defined 
in \eqref{n02}  by Algorithm \ref{algorithm1}. More precisely, 
we have the following result.


\begin{theorem}
\label{Thm2.1-2021} 
Assume that the conditions of Theorem \ref{SCB} hold. \WCr{ Recall that $m_n$ is the block size defined in \eqref{mimic} }
Define
$$
\vartheta_n=\frac{\log^2n}{m_n}+\frac{m_n\log n}{nb_n}+\sqrt{\frac{m_n}{nb_n}}(np)^{4/q}.
$$
If 
$ p \to \infty$ such that  
$np = O (\exp(n^{\iota}))$ 	 for some $0\leq \iota <1/11$
and 
$$\vartheta^{1/3}_n
\Big \{1\vee \log \Big(\frac{np}{\vartheta_n}\Big)\Big\}^{2/3}+\Theta \Big (\Big(\sqrt{m_n\log np} \Big (\frac{1}{\sqrt{nb_n}}+b_n^3 \Big )(np)^{\frac{1}{q}}\Big)^{q/(q+1)},np \Big )=o(1),$$
then the SCS  \eqref{n02} constructed by  Algorithm \ref{algorithm1} satisfies
\begin{align*}
\lim_{n\rightarrow \infty} \lim_{B\rightarrow \infty}  \p(m \in \mathcal C_n  ~ |~{\cal F}_n)=1-\alpha
\end{align*}
in probability.
\end{theorem} 
Theorem \ref{Thm2.1-2021} can also be built on alternative assumptions, such as the polynomial decaying instead of geometric decaying dependence measure as in	Assumption \ref{asserr}(2). Due to page limit, we relegate further discussions on conditions to Section \ref{remrk3.1} of the supplemental material.

\subsection{Theoretical analysis of the methodology in Section \ref{sec22} } \label{sec43}

In this section  we will prove that the 
surface \eqref{n1} defines an asymptotic $(1-\alpha)$ confidence surface with varying width for the mean function $m$. 
If the long-run variance in \eqref{lrv} would be known, 
a confidence surface 
could be based on the 
``normalized'' maximum deviation of 
\begin{equation*}
\label{det1}\hat \Delta^\sigma (u,t)= \frac{\hat m_{b_n}(u,t)-m(u,t)}{\sigma(u,t)}.
\end{equation*}
Therefore we will  derive a 
Gaussian approximation 
for the vector
$\big ( \hat \Delta^\sigma (\frac{l}{n},\frac{k}{p})
\big ) $, ${l=1, \ldots  ,n ; k=1, \ldots , p}$
first  	and define  
for $1\leq i\leq n$ the $p$ dimensional vector
\begin{align*} 
Z^\sigma_{i}(u) & = (Z^\sigma_{i,1}(u), \ldots , Z^\sigma_{i,p}(u))^\top \\
\nonumber 
& = K \Big (\frac{\tfrac{i}{n}-u}{b_n} \Big ) 	\Big (\varepsilon_{i,n}^\sigma (\tfrac{1}{p})
,\varepsilon_{i,n}^\sigma (\tfrac{2}{p})
,\ldots , \varepsilon_{i,n}^\sigma (\tfrac{p-1}{p}), \varepsilon_{i,n}^\sigma (1)
\Big )^\top,
\end{align*}
where $\varepsilon_{i,n}^\sigma (t)=	
\varepsilon_{i,n}(t)/\sigma(\tfrac{i}{n},t)$.
Similarly as in Section \ref{sec22}  we consider the $p$-dimensional vector  
\begin{equation*}
Z^\sigma_{i,l}=Z^\sigma_i(\tfrac{l}{n}) = (Z^\sigma_{i,l,1}, \ldots Z^\sigma_{i,l,p})^\top~,
\end{equation*}
where  
$$
Z^\sigma_{i,l,k} = \varepsilon_{i,n}^\sigma (\tfrac{k}{p})K \Big (\frac{\tfrac{i}{n}-\tfrac{l}{n}}{b_n}\Big) ~~~~
(1\le k\le p).
$$
Finally, we define 
the $(n-2\lceil nb_n\rceil +1)p$-dimensional vectors 
$\tilde Z^\sigma_1 , \ldots , \tilde {Z}^\sigma_{2\lceil nb \rceil-1}$by 
\begin{align}\label{tildeZisigma}
\tilde Z^\sigma_j= \big (
Z_{j,\lceil nb_n\rceil}^{\sigma,\top},
Z_{j+1,\lceil nb_n\rceil+1}^{\sigma,\top}, \ldots ,
Z_{n-2\lceil nb_n\rceil+j, n-\lceil nb_n \rceil  }^{\sigma,\top}
\big)^\top  ~
\end{align}
and obtain the following result.
\begin{theorem}\label{sigma-asy}
Let the Assumptions of Theorem \ref{SCB} be satisfied and  assume that  the partial derivative 
$\frac{\partial^2 \sigma(u,t)}{\partial u\partial t}$ exists and is bounded  on $[0,1]^2$. Then  
there exist 
$(n-2\lceil nb_n\rceil +1)p$-dimensional centered Gaussian vectors
$\tilde Y^\sigma_1, \ldots , \tilde Y^\sigma_{2\lc nb_n\rc-1}$  with the same
auto-covariance structure as  the vector $\tilde Z^\sigma_i$ in \eqref{tildeZisigma} such that 
\begin{align*}  
\mathfrak{P}^\sigma_n& :=
\sup_{x\in \mathbb R} \Big |\p \Big ( 
\max_{b_n \leq u\leq 1-b_n,0\leq t\leq 1}\sqrt{nb_n}|\hat \Delta^\sigma(u,t)|\leq x \Big )-\p 
\Big (\Big |\frac{1}{\sqrt{nb_n}}\sum_{i=1}^{2\lc nb_n\rc-1}\tilde Y^\sigma_i \Big  |_\infty\leq x\Big )\Big | \\
& =O\Big ((nb_n)^{-(1-11\iota)/8}  +\Theta \big (\sqrt{nb_n}(b_n^4+\frac{1}{n}), np \big )
\\
\nonumber &
+\Theta \big  (\big[(np)^{1/q^*}((nb_n)^{-1}+1/p)\big]^{\frac{q^*}{q^*+1}}, n p\big)+\Theta \big  (b_n^{\frac{q-2}{q+1}},np \big) \Big )
\nonumber 
\end{align*}
for any sequence $ p \to \infty$ with 
$np = O (\exp(n^{\iota}))$ 	 for some $0\leq \iota <1/11$. 
In particular, for the choice $p=n^c$ for any $c>0$ we have 
$
\mathfrak{P}^\sigma_n= o(1)
$ 
if the constant  $q^*$ in Assumption \ref{asssim} is sufficiently large, such that $$\Theta \big  (\big[(np)^{1/q^*}((nb_n)^{-1}+1/p)\big]^{\frac{q^*}{q^*+1}}, n p\big)=o(1).$$ 
\end{theorem}

The next result shows that 	the estimator $\hat \sigma$ 
defined by \eqref{2018-5.4}  
is  uniformly  consistent. Thus
the unknown long-run variance $\sigma^2$  in  Theorem \ref{sigma-asy}
can be replaced by $\hat \sigma^2$ and the result
can be used to prove the validity 
of the confidence surface \eqref{n1} defined by  Algorithm \ref{algorithm3}.

\begin{proposition}\label{prop-lrv}
Let the assumptions of Theorem \ref{SCB} be satisfied and  assume that  the partial derivative 
$\frac{\partial^2 \sigma(u,t)}{\partial^2 u}$ exists on the square $[0,1]^2$, is bounded and  Lipschitz continuous in $u \in (0,1)$. If  $w\rightarrow \infty$, $w=o(n^{2/5})$, $w=o(n\tau_n)$, $\tau_n\rightarrow 0$ and $n\tau_n\rightarrow \infty$ we have that
\begin{align*}
& \Big \|\sup_{\substack{u\in [\gamma_n,1-\gamma_n] \\ t\in[0,1]}}
|\hat \sigma^2(u,t)-\sigma^2(u,t)| \Big \|_{q'} =O \Big (g_n+\tau_n^2
\Big  ),\\
& \Big \|\sup_{\substack {u\in [0,\gamma_n)\cup(1-\gamma_n,1] \\ t\in[0,1]
}}|\hat \sigma^2(u,t)-\sigma^2(u,t)| \Big \|_{q'}
= O(g_n+\tau_n)
~,
\end{align*}
where
\begin{equation} \label{gan}
g_n=\frac{w^{5/2}}{n}\tau_n^{-1/q'}+w^{1/2}n^{-1/2}\tau_n^{-1/2-2/q'}+w^{-1}~,
\end{equation}
$\gamma_n=\tau_n+w/n$, $q'=\min(q,q^*)$ and  $q,q^*$ are defined in Assumptions \ref{asserr} and \ref{asssim}, respectively.
\end{proposition}
 \WCr{We investigate the finite sample performance of  the long-run variance estimator $\hat \sigma^2$ in Section \ref{checkLRV}
of the supplementary material by means of a simulation study.} 
Proposition \ref{prop-lrv} and Theorem
\ref{sigma-asy} yield that $\mathfrak P_n^{\hat \sigma}=o_p(1)$ provided that $ \mathfrak{P}_n^{\sigma}=o(1)$.

\begin{theorem}\label{sigma-bootstrap}
Assume that the conditions of  Theorem
\ref{Thm2.1-2021}, Proposition \ref{prop-lrv} hold, that  $p=n^c$ for some $c>0$, and
that  $q^*$ in Theorem \ref{sigma-asy} satisfies 
$\Theta \big  (\big[(np)^{1/q^*}((nb_n)^{-1}+1/p)\big]^{\frac{q^*}{q^*+1}}, n p\big)=o(1)$.
Further assume
there exists  a sequence $\eta_n\rightarrow \infty$ such that 
$$
\Theta\big(\big(\sqrt{m_n\log np}{(g_n+\tau_n)}\eta_n(np)^{\frac{1}{q}}\big)^{q/(q+1)},np\big)+
\eta_n^{-q'}=o(1)~,
$$
where  $\gamma_n$, $g_n$ and $q'$ are defined   in Proposition \ref{prop-lrv}. 
Then the SCS \eqref{n1} defined by   Algorithm \ref{algorithm3} satisfies
\begin{align*}
\lim_{n\rightarrow \infty} \lim_{B\rightarrow \infty}  \p(m \in \mathcal C^{\hat \sigma}_n  ~ |~{\cal F}_n)=1-\alpha
\end{align*}
in probability.  
\end{theorem}
\begin{remark}\label{Boundary-Remark}
\WCr{
Note that the SCS  derived so far 
exclude the  boundary region $[0,b_n)\cup(1-b_n, b_n]$ for the varable $u$.  This is common practice in  the context of simultaneous inference for kernel based estimates,
as inference at the boundary is very difficult due to the inaccurate estimation and sophisticated statistical properties of most nonparametric estimators at the boundary.
The problem  of simultaneous confidence bands  including the boundary region has even not been thoroughly investigated for one-dimensional responses \citep[see for example][]{zhou2010simultaneous,wu2017nonparametric}. 
In the following discussion we provide a first  solution to this problem. 
Since the bias of the usual local linear estimates at the boundary is of order $O(b_n^2)$ and therefore  too large for simultaneous inference,  we use  higher order one sided kernel for the boundary region. 
Simple calculations show that the bias of the  NW estimator with  this kernel can be of order  $O(\frac{1}{n}+b_n^3)$. We discuss the constant width SCS for the boundary, while the varying SCS could be constructed in a similar way. Let $\tilde K_l(v)=\sum_{i=1}^nK_l(\frac{i/n-v/n}{b_n})$ and   $\tilde K_r(v)=\sum_{i=1}^nK_r(\frac{i/n-v/n}{b_n})$.
Then,  similar to \eqref{3.11}, we define
\begin{align}
	\hat Z^l_{i}(u)= K_l \Big (\frac{\tfrac{i}{n}-u}{b_n} \Big ) \big  (\hat \varepsilon_{i,n}(\tfrac{1}{p}),	\varepsilon_{i,n}(\tfrac{2}{p}),\ldots ,\hat \varepsilon_{i,n} (\tfrac{p-1}{p}),
	\hat \varepsilon_{i,n}(1)
	\big )^\top
		\\	\hat Z^r_{i}(u)= K_r \Big (\frac{\tfrac{i}{n}-u}{b_n} \Big ) \big  (\hat \varepsilon_{i,n}(\tfrac{1}{p})
		,\hat \varepsilon_{i,n}(\tfrac{2}{p})
		,\ldots ,\hat \varepsilon_{i,n} (\tfrac{p-1}{p}),
		\hat \varepsilon_{i,n}(1)
		\big )^\top
	\end{align}	
	and consider  for $1\leq s\leq \lc nb_n\rc$ the $\lc nb_n\rc$ and $\lc nb_n\rc+1$ dimensional vectors $\hat {\tilde Z}_s^l$and 	$\hat {\tilde Z}_s^r$ as 
	\begin{align*}
		\hat {\tilde Z}_s^l=( \hat Z^{l\top}_{s}(\tfrac{1}{n})/\tilde K_l(1), \hat Z^{l\top}_{s+1}(\tfrac{2}{n})/\tilde K_l(2), \dots, Z^{l\top}_{s+\lc nb_n\rc-2}(\tfrac{\lc nb_n\rc-1}{n})/\tilde K_l(\lc nb_n\rc-1))^\top\\
			\hat {\tilde Z}_s^r=( \hat Z^{r\top}_{n'+s}(\tfrac{n'+\lc nb_n\rc}{n})/\tilde K_r(n'+\lc nb_n\rc), \hat Z^{r\top}_{n'+s+1}(\tfrac{n'+\lc nb_n\rc+1}{n})/\tilde K_r(n'+\lc nb_n\rc+1), \dots,\\ Z^{r\top}_{n'+s+\lc nb_n\rc-1}(1)/\tilde K_r(n))^\top
	\end{align*}
	where $n'=n-2\lc nb_n\rc+1$. For window size $m_n$, 
	let $m_n'=2\lfloor m_n/2\rfloor$, define
	the   vectors 
	\begin{align}
		\hat S^l_{jm_n'}= \frac{1}{\sqrt{m_n'}} \sum_{r=j}^{j+\lfloor m_n/2\rfloor-1 } 
		\hat{\tilde  Z}^l_{r} - \frac{1}{\sqrt{m_n'}}
		\sum_{r=j+ \lfloor m_n/2\rfloor}^{j+ m_n^\prime -1 }  \hat{\tilde  Z}^l_{r}.
	\end{align}
	Let $\hat{\vec \varepsilon}^{\ l}_{j:j+m_n',k}$ be
	the $p$-dimensional sub-vector of the vector $\hat S^l_{jm_n'}$  
	containing its $(k-1)p+1${st }  $-$  $kp${th}
	components. Then for $v=1, \ldots , B$, we 
	generate i.i.d. $N(0,1)$ random variables $\{R^{(v)}_i\}_{i=1, \ldots  ,2\lc nb_n\rc-m_n'}$.
	and calculate  for  $k= 1 , \ldots  , \lceil nb_n\rceil$ 
	\begin{align} 
		T_k^{l, (v)}& =\sum_{j=1}^{\lceil nb_n \rceil-m_n'} 
		\hat{\vec \varepsilon}^{\ l}_{j:j+m_n',k}
		R^{(v)}_{k+j-1} ~,~~
		T^{l,(r)}  =
		\max_{1\leq k\leq \lceil nb_n\rceil} |T_k^{(r)}|_\infty.
		\nonumber 
	\end{align} 
	Similarly, using $\hat{\tilde Z}_s^r$ and another sequence  of $i.i.d.$ $N(0,1)$ random variables $\{V^{(v)}_i\}_{i=1, \ldots  ,2\lc nb_n\rc-m_n'}$ we could generate	$T^{r,(v)}$ for $v=1,...B$, where $\{V^{(v)}_i\}$ and $\{R^{(v)}_i\}$ are independent. Define $T_{\lf (1-\alpha)B\rf}$ as the empirical $(1-\alpha)$-quantile of the bootstrap sample $T^{l,(1)}, \ldots , T^{l,(B)}$,  $T^{r,(1)}, \ldots , T^{r,(B)}$, then the lower and upper bound of the $(1-\alpha)$-SCS for $u\in [0,b_n)\cup(1-b_n, 1]$ are given by  
$
		\hat   L_{\text{boundary}} (u,t)   = \hat m_{b_n}(u,t) - \hat r_{\text{boundary}}
  $ and $
		\hat  U_{\text{boundary}} (u,t) =  \hat m_{b_n}(u,t) + \hat r_{\text{boundary}}$, respectively
	where 
	$
	\hat r_{\text{boundary}}  =\frac{\sqrt{nb_n} T_{\lf (1-\alpha)B\rf} }{\sqrt{ \lceil nb_n\rceil-m'_n}.
	} 
	$ and $u \in [0,b_n) \cup (1-b_n,1]$, $t \in [0,1]$.
	We examine the empirical coverage probabilities of this SCS in Section \ref{Boundary-check} of the online supplement.
}
\end{remark}

\def\theequation{3.\arabic{equation}}
\setcounter{equation}{0}

\section{Real data}\label{sec3}

In this section we illustrate the proposed methodology analyzing the implied volatility (IV) of the European call option of SP500. These options are contracts such that their holders have the right to buy the SP500 at a specified price (strike price) on a specified date (expiration date).
The implied volatility is derived from the observed SP500 option prices, directly observed parameters, such as risk-free
rate and expiration date,  and option pricing methods, and is widely used in the studies of quantitative finance. For more details, we refer to \cite{hull2003options}. 

We collect the implied volatility and the strike price from the `optionmetrics' database  and the SP500 index from the CRSP database.  Both databases can be accessed from Wharton Research Data Service (WRDS). We calculate the SCSs  for the implied volatility surface, which is a two variate function of time (more precisely time to maturity) and moneyness, where the moneyness is calculated using strike price divided by SP500 indices.  The options are collected from December 21, 2016 to July 19, 2019, and the expiration date is December 20, 2019. Therefore the length of time series is 647. Within each day we observe the volatility curve, which is IV as a function of moneyness.

\begin{figure}[h]
\centering
\includegraphics[scale=0.3]{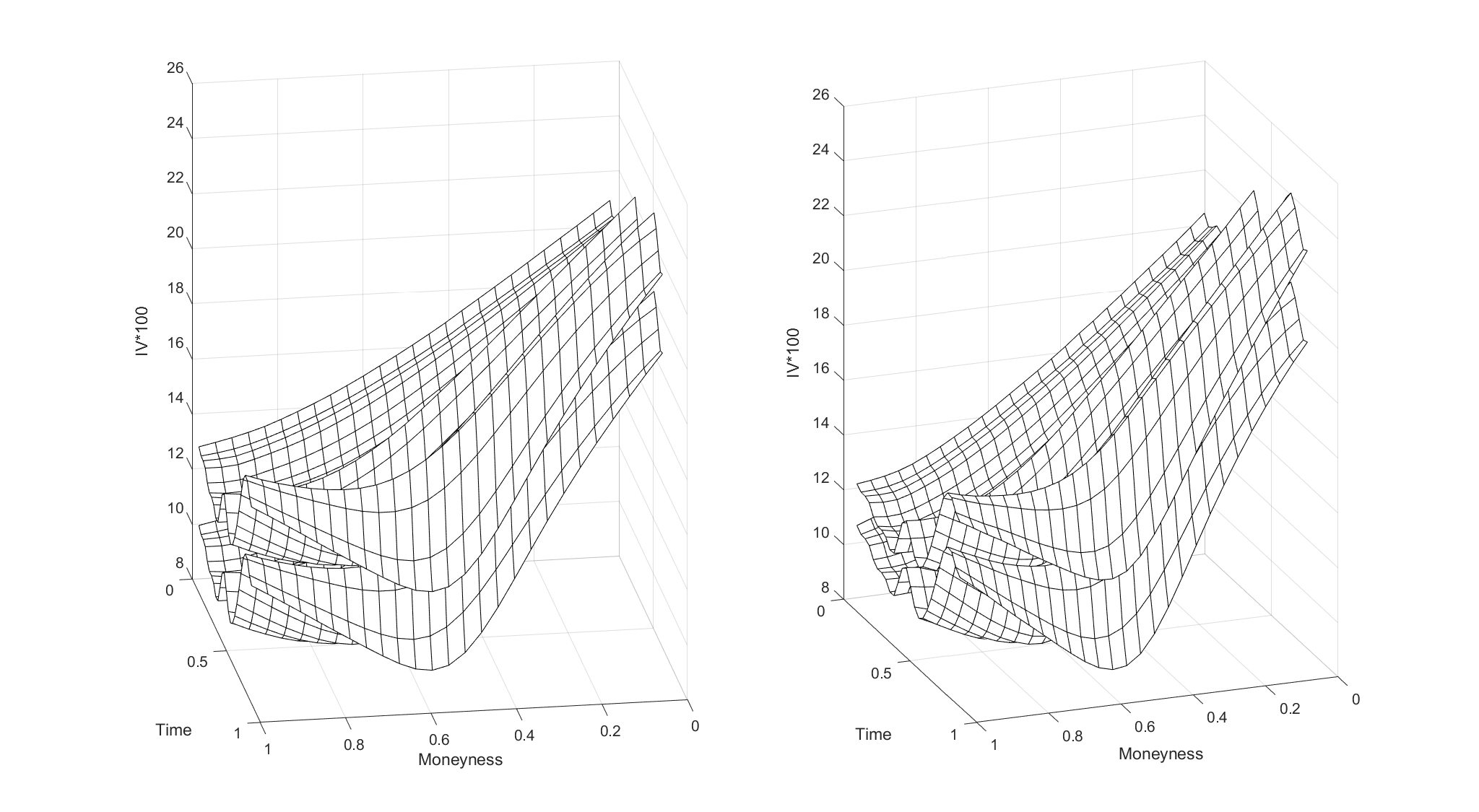}
\vspace{-0.0cm}
\footnotesize{\caption{\it 95\% SCS of the form   \eqref{2.3} for the  IV surface.
Left panel: constant width
(Algorithm   \ref{algorithm1}); Right panel: variable width (Algorithm    \ref{algorithm3}).}}
\label{Vol-SCB}
\end{figure}

Recently,  \cite{liu2016convolutional} models IV via functional time series. 
Following their perspective, we shall study the IV data via  model \eqref{1.1}, where $X_{i,n}(t)$  represents the observed volatility curve at a day $i$, with total sample size $n=647.$ 
We consider the options with moneyness in the range of $[0.8, 1.4]$,
corresponding to options that have been actively traded in this period
(note that, our methodology was developed for functions on the interval $[0,1]$, but it is straightforward to extend this to an arbitrary 
compact interval $[a,b]$).
The number of observations for each day varies from 34 to 56,   and we smooth the implied volatility using linear interpolation and constant extrapolation.

In practice it is important to determine whether the volatility curve changes with time, i.e., to test $H_0: m(u,t)\equiv m(t).$ As pointed out by \cite{daglish2007volatility}, the volatility surface of an asset would be flat and unchanging if the assumptions of Black–Scholes \citep{black2019pricing} hold.  In particular, \cite{daglish2007volatility} demonstrate that  for most assets the volatility surfaces are not flat and are  stochastically changing in practice.
We can provide an inference tool for such a conclusion 
using  the SCSs developed in Section \ref{sec3}.
For example, note, that by the duality between confidence regions
and hypotheses tests, an asymptotic level $\alpha$ 
test for the hypothesis $H_0: m(u,t)\equiv m(t)$
is obtained by rejecting the null hypothesis, whenever the 
surface of the form 
$m(u,t) = m(t) $ is not 
contained in an $(1-\alpha)$ SCS of the form \eqref{2.3}.

Therefore we construct the 95\% SCS
for  the regression function $m$ 
with constant and varying width using Algorithm \ref{algorithm1} and  Algorithm \ref{algorithm3}, respectively. 
Following Section \ref{sec51} of supplemental material 
we choose $b_n=0.1$ and $m_n=36$.  The results are depicted in Figure \ref{Vol-SCB} (for a better illustration the  $z$-axis  shows $100\times$ implied volatility). We observe
from both figures that the SCSs do  not contain a surface of the form $m(u,t)=m(t)$ and therefore  reject the 
null hypothesis  (at significance level $5\%$). In the supplement, we construct the simultaneous confidence bands of IV w.r.t. fixed $t$ and fixed $u$. 
 
\section*{Supplementary Materials} contains further 
results for \WCr{noisy and multivariate locally stationary functional time series}, the confidence bands for the functions $t \to m(u,t) $ (fixed $u$) and $u \to m(u,t) $ (fixed $t$), implementation details, simulation and additional data analysis results, examples of locally stationary functional time series. It also includes all the detailed proofs.  

\bigskip

\noindent
\section*{Acknowledgements} Holger Dette gratefully acknowledges Collaborative Research Center “Statistical modeling
of nonlinear dynamic processes” (SFB 823, Project A1, C1)
of the German Research Foundation (DFG). Weichi Wu gratefully acknowledges
NSFC no.12271287. Weichi Wu is the corresponding author.
\bigskip

\bibliographystyle{apalike}
\begin{footnotesize}
\setlength{\bibsep}{1pt}
\bibliography{references}
\end{footnotesize}

\def\theequation{S\arabic{section}.\arabic{equation}}
\def\thesection{S\arabic{section}}
\def\theproposition{S\arabic{proposition}}
\def\thetheorem{S\arabic{theorem}}
\def\thelemma{S\arabic{lemma}}
\def\thecorrolary{S\arabic{corollary}}

\def\theremark{S\arabic{remark}}

\def\theexample{S\arabic{example}}
\def\thedefinition{S\arabic{definition}}

\def\thetable{S\arabic{table}}
\def\thefigure{S\arabic{figure}}
\renewcommand{\thealgocf}{S\arabic{algocf}}
\renewcommand{\baselinestretch}{2}

\markright{ \hbox{\footnotesize\rm Statistica Sinica: Supplement
	}\hfill\\[-13pt]
	\hbox{\footnotesize\rm
	}\hfill }

\markboth{\hfill{\footnotesize\rm Holger Dette \& Weichi Wu} \hfill}
{\hfill {\footnotesize\rm FILL IN A SHORT RUNNING TITLE} \hfill}

\renewcommand{\thefootnote}{}
$\ $\par \fontsize{12}{14pt plus.8pt minus .6pt}\selectfont


\centerline{\large\bf Supplemental Material For}
\vspace{2pt}
\centerline{\large\bf  "Confidence surfaces for the mean  of locally stationary}
\vspace{2pt}
\centerline{\large\bf functional time series"}
\vspace{.25cm}
\author{Holger Dette, Weichi Wu}
\vspace{.4cm}
\centerline{\it Ruhr-Universit\"at Bochum and Tsinghua University}
\vspace{.55cm}
\centerline{\bf Supplementary Material}
\vspace{.55cm}
\fontsize{9}{11.5pt plus.8pt minus .6pt}\selectfont
\noindent
Section \ref{Sec-rmks} provides additional remarks regarding the noisy and multivariate locally stationary functional time series. Section \ref{finitesample} provides approaches to the selection of tuning parameters and simulation results, including the simulation results for both boundary and interior regions, and the simulation results checking the Gaussian approximation and the long-run variance function estimator. Section  \ref{remrk3.1} discusses possible alternative assumptions for Theorem \ref{Thm2.1-2021}.	Section \ref{secA}  contains some details about simultaneous 
confidence bands for the regression function 
in model \eqref{1.1}, where one of the arguments
is fixed (Section \ref{seca1}) including additional numerical results for this case (see Sections \ref{sec81} and \ref{Section-Real Data}). 
In Section \ref{secex}  we provide  examples of locally stationary functional processes, illustrating our approach of modeling non-stationary 
functional data. Section \ref{proof-main} contains the proof of  all Theorems, 
while Section \ref{Proof-Prop} provides propositions. 
Finally, Section \ref{aux} presents  auxiliary results for the proofs.
\par

\setcounter{section}{0}
\setcounter{equation}{0}

\fontsize{12}{14pt plus.8pt minus .6pt}\selectfont
\section{Additional remarks}\label{Sec-rmks}
In this section, we provide two remarks which briefly discuss  how to build simultaneous confidence surface for noisy data and multivariate locally stationary functional time series using our method.
\begin{remark}\label{noisydata}
	Indeed several authors consider (stationary) functional data models 
	with  noisy observation \citep[see][among others]{caoyangtodem2012,chensong2015}
	and we expect  that the results presented in this section can be extended to this scenario. More precisely, consider the model
	$$
	Y_{ij}=X_{i,n}({
		\tfrac{j}{N}})+\sigma({\tfrac{j}{N}})z_{ij},~1\leq i\leq n, 1\leq j\leq N~,
	$$
	where $X_{i,n}$ is  the functional time series defined in \eqref{1.1},  $\{z_{ij}\}_{i=1, \ldots , n, j=1,\dots, N} $  is an array of centered independent identically distributed observations and $\sigma(\cdot)$ is a positive function on the interval $[0,1]$. {This means that one
		can not observe the full trajectory of $\{ X_{i,n} (t) ~| ~t \in [0,1] \}$, but only the function  $X_{i,n}$ evaluated at the discrete time points $1/N, 2/N, \ldots ,(N-1)/N,1$ subject to some random error.
		If $N\rightarrow \infty$ as $n\rightarrow \infty$, and the regression function $m$ in \eqref{1.1}  is sufficiently smooth, we expect that  we can construct 
		simultaneous confidence bands and surfaces by applying the
		procedure described in this section
		to smoothed trajectories. 

		For example, we 
		can consider the  smooth estimate 
		\begin{align}
			\tilde m(u,\cdot )=\argmin_{g\in {\cal S}_p }\sum_{i=\lf nu-\sqrt n\rf}^{\lceil nu+\sqrt n\rceil}\sum_{j=1}^N\big( {
				Y_{i,j}}
			-g(\tfrac{j}{N})\big)^2~,
		\end{align}
		where  ${\cal S}_p$ denotes the set of splines of order $p$, which depends on the  smoothness of the function $t\to m(u,t)$.
		We can now construct 
		confidence bands applying the
		methodology  to the 
		data $\tilde X_{i,n} (\cdot ) =\tilde m(\tfrac{i}{\sqrt n},\cdot )$, $i=1, \ldots , \sqrt n$
		due to the asymptotic efficiency of the spline estimate
		\citep[see Proposition 3.2-3.4 in ][]{caoyangtodem2012}. }

	Alternatively,   we can also
	obtain smooth
	estimates 
	$t \to \check X_{i,n} (t)$
	of the  trajectory
	using local polynomials, 
	and we expect that the proposed methodology 
	applied to the data $\check  X_{1,n} , \ldots , \check X_{n,n}$ will yield valid 
	simultaneous confidence bands 
	and surfaces,
	where the range for the variable $t$ 
	is restricted to the interval 
	$ [c_n,1-c_n]$ and $c_n$ denotes the bandwidth of the local polynomial estimator used in smooth estimator of the trajectory.
\end{remark}
\begin{remark}\label{remark4.1-multiple}
	{\rm 
		The methodology presented so far can be extended to  construct a simultaneous confidence surfaces  for the vector of mean functions of a multivariate locally stationary functional time series. For simplicity we  consider  a  $2$-dimensional series of the form 
		\begin{align} \label{det42}
			\left (\begin{matrix}
				X^1_{i,n}(t) \\ X^2_{i,n}(t)  
			\end{matrix} \right)
			=	
			\left (\begin{matrix}
				m_1(\tfrac{i}{n},t) \\ m_2(\tfrac{i}{n},t)  
			\end{matrix}
			\right) ~+~
			\left (\begin{matrix}
				\varepsilon^1_{i,n}(t) \\ \varepsilon^2_{i,n}(t)
			\end{matrix}
			\right)~,
		\end{align}
		and define for  $a=1,2$ 
		\begin{align*}
			\hat Z^{a,\hat \sigma}_{i}(u)
			& = (\hat Z^{a,\hat \sigma}_{i,1}(u), \ldots , \hat Z^{a,\hat \sigma}_{i,p}(u))^\top \\
			\nonumber 
			& = K \Big (\frac{\tfrac{i}{n}-u}{b_n} \Big ) \Big  (\tfrac{\hat \varepsilon^a_{i,n}(\tfrac{1}{p})}{\hat \sigma_a(\tfrac{i}{n},\tfrac{1}{p})}
			,\tfrac{\hat \varepsilon^a_{i,n}(\tfrac{2}{p})}{\hat \sigma_a(\tfrac{i}{n},\tfrac{2}{p})}
			,\ldots ,\tfrac{\hat \varepsilon^a_{i,n} (\tfrac{p-1}{p})}{\hat \sigma_a(\tfrac{i}{n},\tfrac{p-1}{p})},
			\tfrac{\hat \varepsilon^a_{i,n}(1)}{\hat \sigma_a(\tfrac{i}{n},1)}
			\Big )^\top, 
		\end{align*}
		where $\hat \varepsilon^a_{i,n}(t)=X^a_{i,n}(t)-\hat m_{a}(\frac{i}{n},t)$ and 
		$\hat \sigma^2_a(\frac{i}{n},t)$ is the estimator of long-variance of $\varepsilon^a_{i,n}$ defined in \eqref{2018-5.4}. Next we 
		consider  the $2(n-2\lceil nb_n\rceil+1 )p$-dimensional  vector 
		\begin{align*}
			\hat{\tilde Z}_j^{\hat \sigma}= \big (
			\hat Z_{j,\lceil nb_n\rceil}^{\hat \sigma, \top},
			\hat	Z_{j+1,\lceil nb_n\rceil+1}^{\hat \sigma, \top}
			\ldots ,
			\hat	Z_{n-2\lceil nb_n\rceil+j, n-\lceil nb_n \rceil  }^{\hat \sigma, \top}
			\big)^\top ~,
		\end{align*}
		where 
		$
		\hat Z^{\hat \sigma}_{i,l}= 	\hat Z^{\hat \sigma}_i(\tfrac{l}{n}) = (	\hat Z^{1,\hat \sigma}_{i,l,1},\hat Z^{2,\hat \sigma}_{i,l,1} \ldots 	\hat Z^{1,\hat \sigma}_{i,l,p},\hat Z^{2,\hat \sigma}_{i,l,p})^\top~
		$
		contains information from both components.
		Define for $a=1,2$
		\begin{align*}
			\hat L_{3,a}^{\hat\sigma}(u,t)=\hat m_{a}(u,t)-\hat r_{3,a} (u,t), ~~  \hat U_{3,a}^{\hat\sigma}(u,t)=\hat m_{a}(u,t)+\hat r_{3,a}(u,t)  ~,
		\end{align*}
		where 
		$$
		\hat r_{3,a} (u,t) =  \frac{\hat \sigma_a(u,t)\sqrt 2 T^{\hat \sigma } _{\lf (1-\alpha)B\rf} }{\sqrt{nb_n}\sqrt{ 2\lceil nb_n\rceil-m'_n}}
		$$
		and $T^{\hat \sigma } _{\lf (1-\alpha)B\rf}$ is generated in the same way as in step (d) of Algorithm \ref{algorithm3} with $p$ replaced by $2p$, $\hat m_a $ is the kernel estimator of $m_a $  defined in \eqref{estimating-m}. 
		Further, define for $a=1,2$ the set of functions
		\begin{align*} 
			{\cal C}^{\hat \sigma}_{a,n} =     \big \{ f\in \mathcal C^{3,0} : [0,1]^2 \to \R ~| ~~ \hat L_{3,a} (u,t) \leq f(u,t)  \leq  \hat U_{3,a} (u,t) \notag\\~~\forall   u \in  [b_n,1-b_n] ~\forall   t \in  [0,1] \big \} .
		\end{align*}
		Suppose that the mean   functions and error processes of $X^1_{i,n}(t)$ and $X^2_{i,n}(t)$  satisfy the conditions of Theorem \ref{sigma-bootstrap}, then it can be proved that 
		the set $\mathcal C^{\hat \sigma}_{1,n} \times \mathcal C^{\hat \sigma}_{2,n}$
		defines an asymptotic $(1-\alpha)$
		simultaneous confidence surface for the vector  function $(m_1,m_2)^\top $. The details are omitted for the sake of brevity.
	}
\end{remark}

\section{Finite Sample Performance}\label{finitesample}

In this section we study the finite sample performance of the simultaneous confidence  surfaces proposed in the previous sections. 
We start giving some more details regarding the general implementation of the algorithms, and present the simulation study.

\subsection{Implementation} \label{sec51}

For the estimator  of the regression function in \eqref{estimating-m} we use  the kernel 
(of order $4$) in $[b_n, 1-b_n]$
\begin{align*}
	K(x)=(45/32-150x^2/32+105x^4/32)\mathbf 1(|x|\leq 1)~,
\end{align*}
and for the boundary we use the kernel function $K_l(x)=(420x^2-480x+120)x(1-x)\mathbf 1(0\leq x\leq 1)$.
	\textcolor{blue}{We choose the bandwidth as the minimizer of
		\begin{align}
			\label{det2} 
			MGCV(b)=\max_{1\leq s\leq p}\frac{\sum_{i=1 }^{n}(\hat m_b(\tfrac{i}{n},\tfrac{s}{p})-X_{i,n}(\tfrac{s}{p}))^2}{(1-{\rm tr} (Q_s(b))/n)^2}~,
		\end{align}
		$Q_s(b)$ is an $n \times n $ 
		matrix such that
		$$
		\big 
		(\hat m_b(\tfrac{1}{n},\tfrac{s}{p}), 
		\hat m_b(\tfrac{2}{n},\tfrac{s}{p}),
		\ldots ,\hat m_b(1,\tfrac{s}{p})\big )^\top =Q_s(b)\big ( X_{1,n}(\tfrac{s}{p}), \ldots ,X_{n,n}(\tfrac{s}{p})\big )^\top.
		$$
		Here  $\hat m_b(u,t)$ is the NW estimator 
		with bandwidth $b$ 
		defined in \eqref{estimating-m}.}

The criterion \eqref{det2} is motivated by the generalized cross validation criterion introduced by \cite{craven1978smoothing} and will be called
Maximal Generalized Cross Validation (MGCV) method throughout this paper.

For the estimator  of the  long-run variance 
in \eqref{2018-5.4} we use  $w=\lf n^{2/7}\rf$ and $\tau_n=n^{-1/7}$ as recommended in \cite{dette2019detecting}.
The window size in the multiplier bootstrap is then selected by the  minimal volatility method advocated by \cite{politis1999subsampling}.
For the sake of brevity, we discuss this method  only for Algorithm \ref{algorithm3} in detail
(the method  for Algorithm  \ref{algorithm1} is similar). 
We consider 
a grid of window sizes $\tilde m_1< \ldots < \tilde m_M$ (for some integer $M$).
We first calculate   $\hat S_{j \tilde{m}_s}^{\hat \sigma}=(\hat S_{j \tilde{m}_s,r}^{\hat \sigma}, 1\leq r\leq(n-2\lceil nb_n\rceil+1)p)$ defined in step (c) of Algorithm \ref{algorithm3} for each   $\tilde{m}_s$. 
Let $\hat S_{\tilde m_s}^{\hat \sigma,\diamond}$  denote the $(n-2\lceil nb_n \rceil+1)p$ dimensional vector with $r_{th}$ entry defined by   
$$
\hat S_{\tilde m_s,r}^{\hat \sigma,\diamond} =
\frac{1} { 2\lceil nb_n\rceil-\tilde m_s } 
\sum_{j=1}^{2\lceil nb_n \rceil-\tilde m_s}(\hat S_{j\tilde m_s,r}^{\hat \sigma})^2~,
$$
and consider the standard error of $\{\hat S_{\tilde m_s,r}^{\hat \sigma,\diamond}\}_{s=k-2}^{k+2}$, 
that is 
\begin{align*} 
	{\rm se} \big(\{\hat S_{\tilde m_s,r}^{\hat \sigma,\diamond}\}_{s=k-2}^{k+2}\big)=  \Big (\frac{1}{4}\sum_{s=k-2}^{k+2} \Big (\hat S_{\tilde m_s,r}^{\hat \sigma,\diamond}-\frac{1}{5}\sum_{s=k-2}^{k+2}\hat S_{\tilde m_s,r}^{\hat \sigma,\diamond}\Big )^2 \Big )^{1/2}.
\end{align*}
Then we choose $m_n'=\tilde m_j$ where $j$ is defined as the minimizer  of 
the function 
\begin{align*}
	MV(k)= 
	\frac{1}{(n-2\lceil nb_n\rceil+1)p} \sum_{r=1}^{(n-2\lceil nb_n\rceil+1)p}
	se\big(\{\hat S_{\tilde m_s,r}^{\hat \sigma,\diamond}\}_{s=k-2}^{k+2}\big)~
\end{align*}
in the set $\{ 3, \ldots ,  M-2\}$. 
Throughout this section   we consider $p=\lf \sqrt n\rf$.

\subsection{Simulated  data} \label{sec52}

We  consider two regression functions 
\begin{align*}
	\textcolor{blue}{m_1(u,t)}&\textcolor{blue}{=(u+2t)^2/2},  \\ 
	m_2(u,t)& =(1+u^2)(6(t-0.5)^2(1+\mathbf 1(t>0.3))+1)
\end{align*}
(note that  $m_2$ is discontinuous at the point $t=0.3$). For the definition of the error
processes let $\{\varepsilon_i\}_{i\in \mathbb Z}$ be 
a sequence of independent standard normally distributed random variables and  $\{\eta_i\}_{i\in \mathbb Z}$ be 
a sequence of independent $t$-distributed random variables with $8$ degrees of freedom. Define the functions
\begin{align*}
	&  ~a(t)  =0.5 \cos(\pi t/3) , ~~~~~~~~~~~~
	b(t) =0.4t , 
	~~~~~~~~~~~~  c(t) =0.3t^2 ,  \\  
	&  d_1(t) =1+0.5 \sin(\pi t) , 
	~~~~~~ ~ d_{2,1}(t)  =2t-1 , 
	~~~~~~   d_{2,2}(t) =6t^2-6t+1 , 
\end{align*}
and  
$\FF_i^1=( \ldots  , \varepsilon_{i-1  },  \varepsilon_i)$,
$\FF_i^2=( \ldots  , \eta_{i-1},\eta_i)$.
We consider the following two locally stationary functional time series errors $G_1$ and $G_2$ are defined by  
\begin{align*} 
	\textcolor{blue}{G_1(u,t,\FF_i^1)=G_0(u,t,\FF_i^1)d_1(t)/3, \text{where}~ G_0(u,t,\FF_i^1)=(a(u)-0.1t)G_0(u,t,\FF_i^1)+\epsilon_i,}\\
	G_2(u,t,\FF_i^1,\FF_i^2)=	\tilde G_1(u,\FF_i^1)d_{2,1}(t)/2+\tilde G_2(u,\FF_i^2)d_{2,2}(t)/2
\end{align*}
where the locally stationary time series $\tilde G_1 $ and $\tilde G_2$ are defined as
\begin{align*}
	\tilde	G_1(u,\FF_i^1)=a(u)\tilde G_1(u,\FF ^1_{i-1})+\varepsilon_i,~
	\tilde	G_2(u,\FF_i^2)=b(u)\tilde G_2(u,\FF^2_{i-1})+\eta_i-c(u)\eta_{i-1}.
\end{align*}
Note that 
$\tilde G_1$ is a  locally stationary AR($1$) process (or equivalently a
locally stationary MA($\infty$)  process), and that $\tilde G_2$ is a locally stationary ARMA($1,1$)  model. 
With these processes we define  the following   functional time series model (for $1\leq i\leq n$, $0\leq t\leq 1$)
\begin{align*}
	\text{(a)} ~~ X_{i,n} (t)=m_1(\tfrac{i}{n},t)+G_1(\tfrac{i}{n},t,\FF_i^1)
	~~	\text{(b)} X_{i,n} (t)=m_1(\tfrac{i}{n},t)+G_2(\tfrac{i}{n},t,\FF_i^1,\FF_i^2)
	\\
	\text{(c)} ~~ X_{i,n} (t)=m_2(\tfrac{i}{n},t)+G_1(\tfrac{i}{n},t,\FF_i^1)
	~~	\text{(d)}  X_{i,n} (t)=m_2(\tfrac{i}{n},t)+G_2(\tfrac{i}{n},t,\FF_i^1,\FF_i^2).
\end{align*}
In Figure \ref{SCB-simulation} we display typical 
$95\%$ simultaneous confidence surfaces of the form \eqref{2.3} from one simulation run
for  model (a) with sample size  $n=800$
and $B=1000$ bootstrap replications, which are calculated  by Algorithm \ref{algorithm1} (constant width) and  Algorithm \ref{algorithm3} (varying width). 
We observe that  there exist  differences between the  surfaces with constant and variable width, but they are not substantial.

\begin{figure}[h]
\centering
\includegraphics[width=\textwidth,height=6.5cm]
{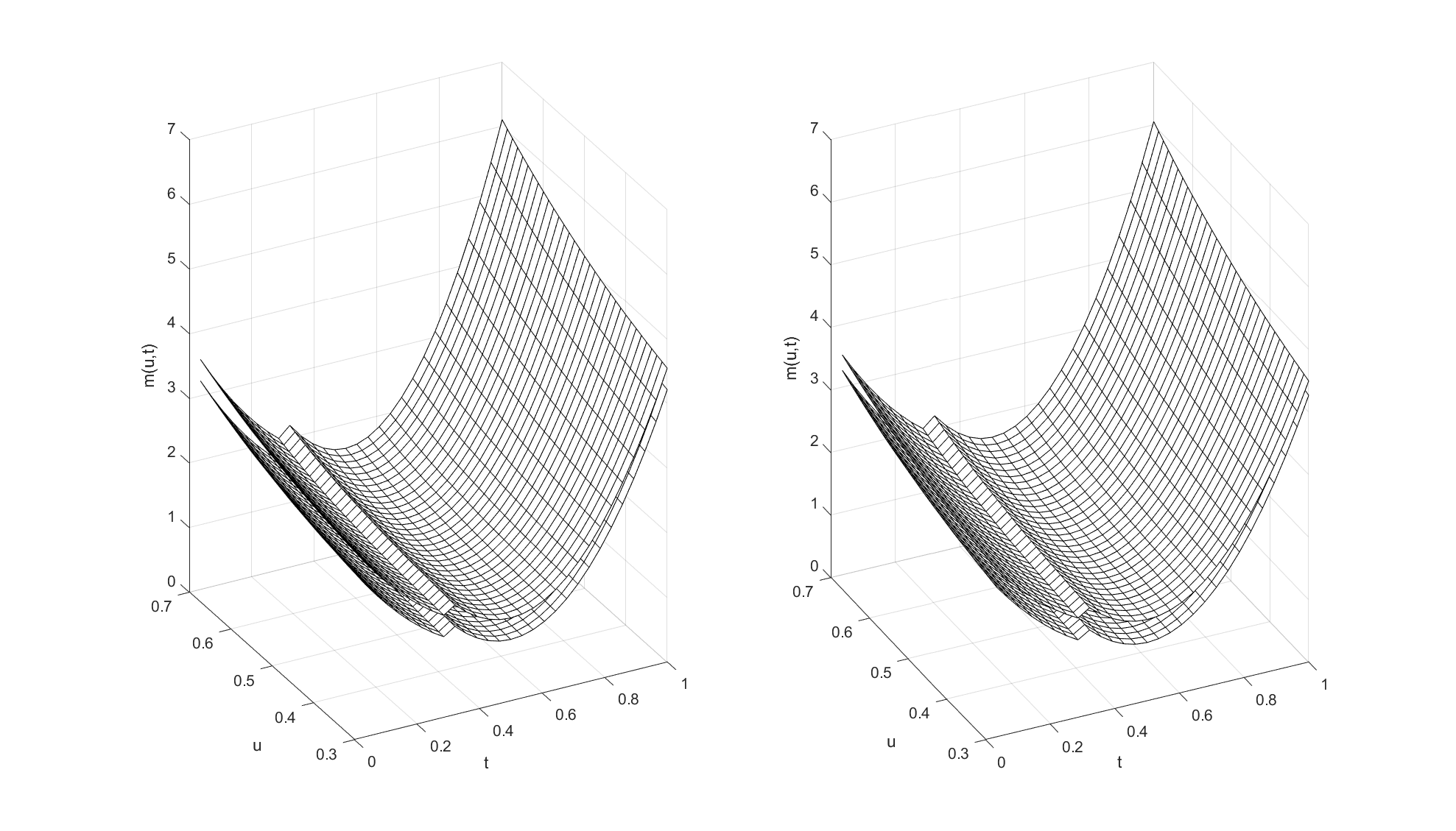}
\vspace{-0.5cm}
\caption{ \it 
	$95\%$ simultaneous confidence surfaces   \eqref{n02}  and 		\eqref{n1} 	
	for the regression function in model (c) from $n=800$
	observations. Left panel: constant width (Algorithm   \ref{algorithm1}); Right panel: varying width (Algorithm   \ref{algorithm3})
}
\label{SCB-simulation}
\end{figure}

We next investigate the  coverage probabilities 
of the different surfaces constructed in this paper for sample sizes $n=500$ and $n=800$.
All results are based on $1000$ simulation runs and $B=1000$ bootstrap replications.
The left part of  Table \ref{table1} shows the coverage probabilities of the surfaces with constant width	while the results in the right part correspond to the bands with varying width. 
We observe that the simulated coverage probabilities are  close to their nominal levels in all cases under consideration, which illustrates the validity of our methods for finite sample sizes.

We conclude this section mentioning 
that confidence bands  for the regression function $m$ for a fixed $u$ or a fixed $t$ can be constructed in a similar manner and
details and some additional numerical results for these bands  are  discussed in Section \ref{secA}.

\begin{table}[H]
\caption{\label{table1}\it Simulated coverage probabilities of the simultaneous confidence bands
	\eqref{n02}  and 
	\eqref{n1} 
	calculated by  Algorithm   \ref{algorithm1}
	(constant width) and  Algorithm   \ref{algorithm3} (varying width), respectively. 
}
\centering
\begin{tabular}{lcccccccc}
	\toprule
	& \multicolumn{4}{c}{constant width} & \multicolumn{4}{c}{varying width} \\
	\cmidrule(lr){2-5} \cmidrule(lr){6-9}
	& \multicolumn{2}{c}{model (a)} & \multicolumn{2}{c}{model (b)} & \multicolumn{2}{c}{model (a)} & \multicolumn{2}{c}{model (b)} \\
	\cmidrule(lr){2-3} \cmidrule(lr){4-5} \cmidrule(lr){6-7} \cmidrule(lr){8-9}
	level   & 90\% & 95\% & 90\% &  95\% & 90\% & 95\% & 90\% & 95\% \\
	\cmidrule(lr){1-9}
	n=500 & 88.0\% & 94.2\% & 90.1\% & 93.8\% & 91.2\% & 95.3\% & 87.9\% & 93.6\% \\
	n=800 & 89.9\% & 95.8\% & 88.3\% & 93.9\% & 90.9\% & 96.1\% & 90.7\% & 96.0\% \\
	\midrule
	& \multicolumn{2}{c}{model (c)} & \multicolumn{2}{c}{model (d)} & \multicolumn{2}{c}{model (c)} & \multicolumn{2}{c}{model (d)} \\
	\cmidrule(lr){2-3} \cmidrule(lr){4-5} \cmidrule(lr){6-7} \cmidrule(lr){8-9}
	level   & 90\% & 95\% & 90\% &  95\% & 90\% & 95\% & 90\% & 95\% \\
	\cmidrule(lr){1-9}
	n=500 & 87.9\% & 93.9\% & 91.3\% & 95.4\% & 87.5\% & 95.1\% & 87.7\% & 94.8\% \\
	n=800 & 88.6\% & 94.2\% & 89.9\% & 95.9\% & 90.8\% & 95.0\% & 90.1\% & 94.9\% \\
	\bottomrule
\end{tabular}
\end{table}

\subsection{Simulation results in the boundary}\label{Boundary-check}

\WC{	We examine the proposed method for the simultaneous inference in the boundary region in Remark  \ref{Boundary-Remark}. We summarize our results in table \ref{Boundary-Table}, and find that our method in boundary works reasonably well.
}
\begin{table}[h!]
\centering
\caption{Simulated coverage probabilities of simultaneous confidence surface in the boundary using methods in Remark \ref{Boundary-Remark}. }\label{Boundary-Table}
\begin{tabular}{lcccccccc}
	\toprule
	& \multicolumn{4}{c}{constant width} & \multicolumn{4}{c}{varying width} \\
	\cmidrule(lr){2-5} \cmidrule(lr){6-9}
	& \multicolumn{2}{c}{model (a)} & \multicolumn{2}{c}{model (b)} & \multicolumn{2}{c}{model (a)} & \multicolumn{2}{c}{model (b)} \\
	\cmidrule(lr){2-3} \cmidrule(lr){4-5} \cmidrule(lr){6-7} \cmidrule(lr){8-9}
	level   & 90\% & 95\% & 90\% &  95\% & 90\% & 95\% & 90\% & 95\% \\
	\cmidrule(lr){1-9}
	n=500 & 87.6\% & 93.9\% & 87.2\% & 93.7\% & 91.9\% & 96.0\% & 91.2\% & 96.2\% \\
	n=800 & 89.4\% & 94.4\% & 90.6\% & 96.1\% & 90.2\% & 95.4\% & 89.7\% & 95.0\% \\
	\midrule
	& \multicolumn{2}{c}{model (c)} & \multicolumn{2}{c}{model (d)} & \multicolumn{2}{c}{model (c)} & \multicolumn{2}{c}{model (d)} \\
	\cmidrule(lr){2-3} \cmidrule(lr){4-5} \cmidrule(lr){6-7} \cmidrule(lr){8-9}
	level   & 90\% & 95\% & 90\% &  95\% & 90\% & 95\% & 90\% & 95\% \\
	\cmidrule(lr){1-9}
	n=500 & 91.4\% & 95.1\% & 89.8\% & 94.6\% & 90.4\% & 95.5\% & 90.8\% & 95.3\% \\
	n=800 & 89.8\% & 95.7\% & 90.1\% & 94.9\% & 90.7\% & 95.4\% & 88.4\% & 94.1\% \\
	\bottomrule
\end{tabular}
\end{table}

\subsection{Empirical investigation of Theorem \ref{SCB}}\label{CheckThm1}
\textcolor{blue}{In this section we investigate the finite sample accuracy of the  Gaussian approximation  in Theorem \ref{SCB}.  We consider  model (d), the sample size  $n=500, 800$ and $b=0.1, 0.2$
and compare the simulated quantiles of  the maximum deviation of $\max_{\substack{b_n \leq u\leq 1-b_n \\ 0\leq t\leq 1}}\sqrt{nb_n}|\hat \Delta(u,t)| $ and that of the maximum norm of the sum of corresponding high-dimensional Gaussian vectors with the auto-covariance structure described in Theorem \ref{SCB}. The results are presented in
Figure \ref{Check-Theorem1}, which shows that the approximation accuracy of Theorem \ref{SCB} is quite high.} 
\begin{figure}
\centering
\includegraphics[width=14cm, height=14cm]{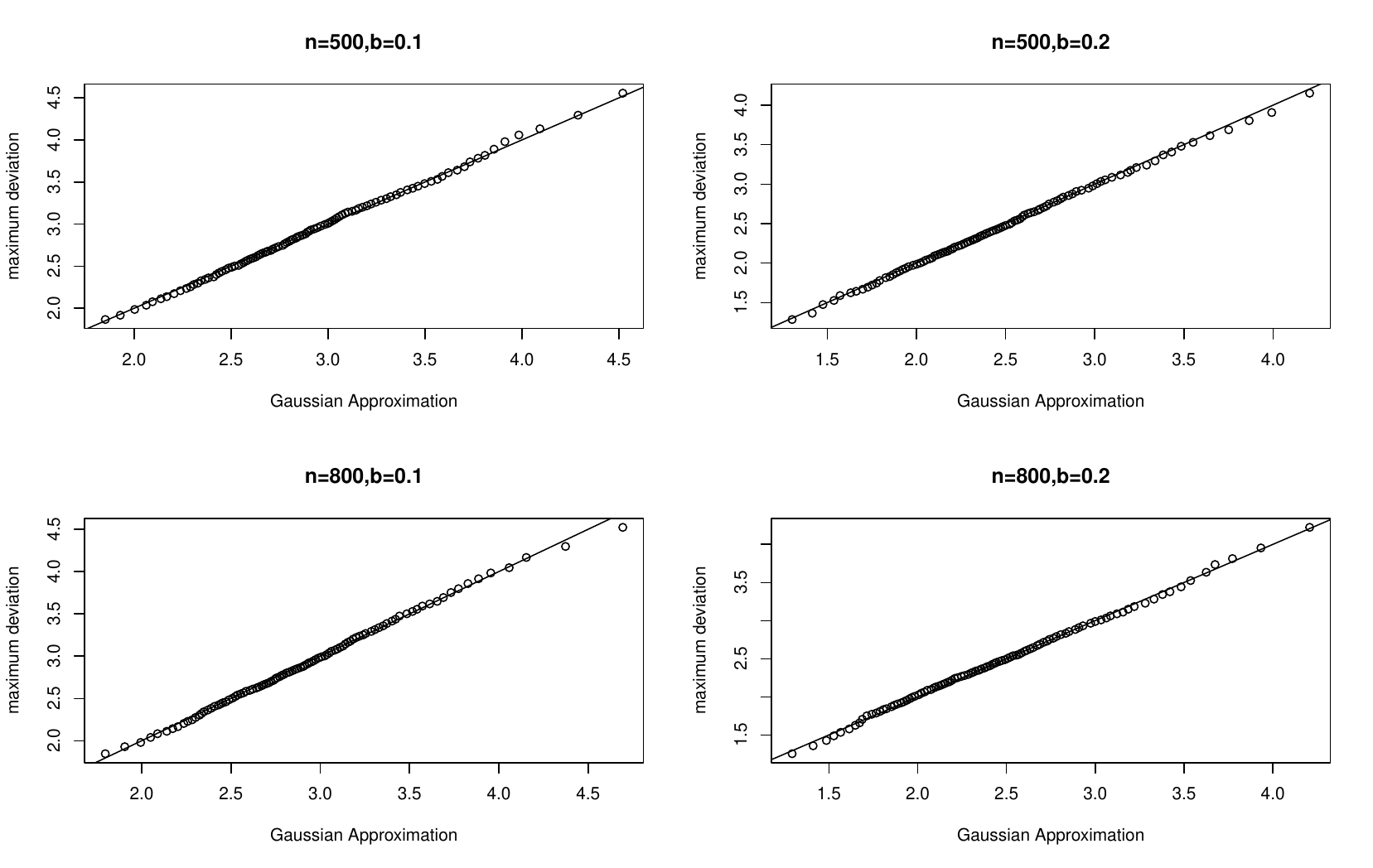}
\caption{\it Quantile-quantile plot of the $\max_{\substack{b_n \leq u\leq 1-b_n ,  0\leq t\leq 1}}\sqrt{nb_n}|\hat \Delta(u,t)|$  versus $\Big |\frac{1}{\sqrt{nb_n}}\sum_{i=1}^{2\lc nb_n\rc-1}\tilde Y_i \Big  |_\infty$ as described in Theorem \ref{SCB}.}\label{Check-Theorem1}
\end{figure}

\subsection{Empirical performance of the long-run variance estimator}
\label{checkLRV}
\textcolor{blue}{In this section we investigate the finite sample performance of the
difference based long-run variance estimator \eqref{lrv}. We examine the maximum  error 
\begin{align}
	\label{holger1}   
	\max_{1 \leq i\leq n ,1\leq j\leq p}|\hat \sigma(i/n,j/p)-\sigma(i/n,j/p)|
\end{align}
where $p=\lf n^{1/2}\rf$ as mentioned in Section \ref{sec51}. We consider model (a), (c), (b), (d) with sample size $n=500$ and $800$, respectively. The results are shown in Figure \ref{Boxplot}, where we display for each case the box plot of $2000$ simulations
of \eqref{holger1}. We observe that the  estimator works reasonably well and in all simulation scenarios the estimation error decreases as the sample size increases.}
\begin{figure}
\centering
\includegraphics[width=14cm, height=14cm]{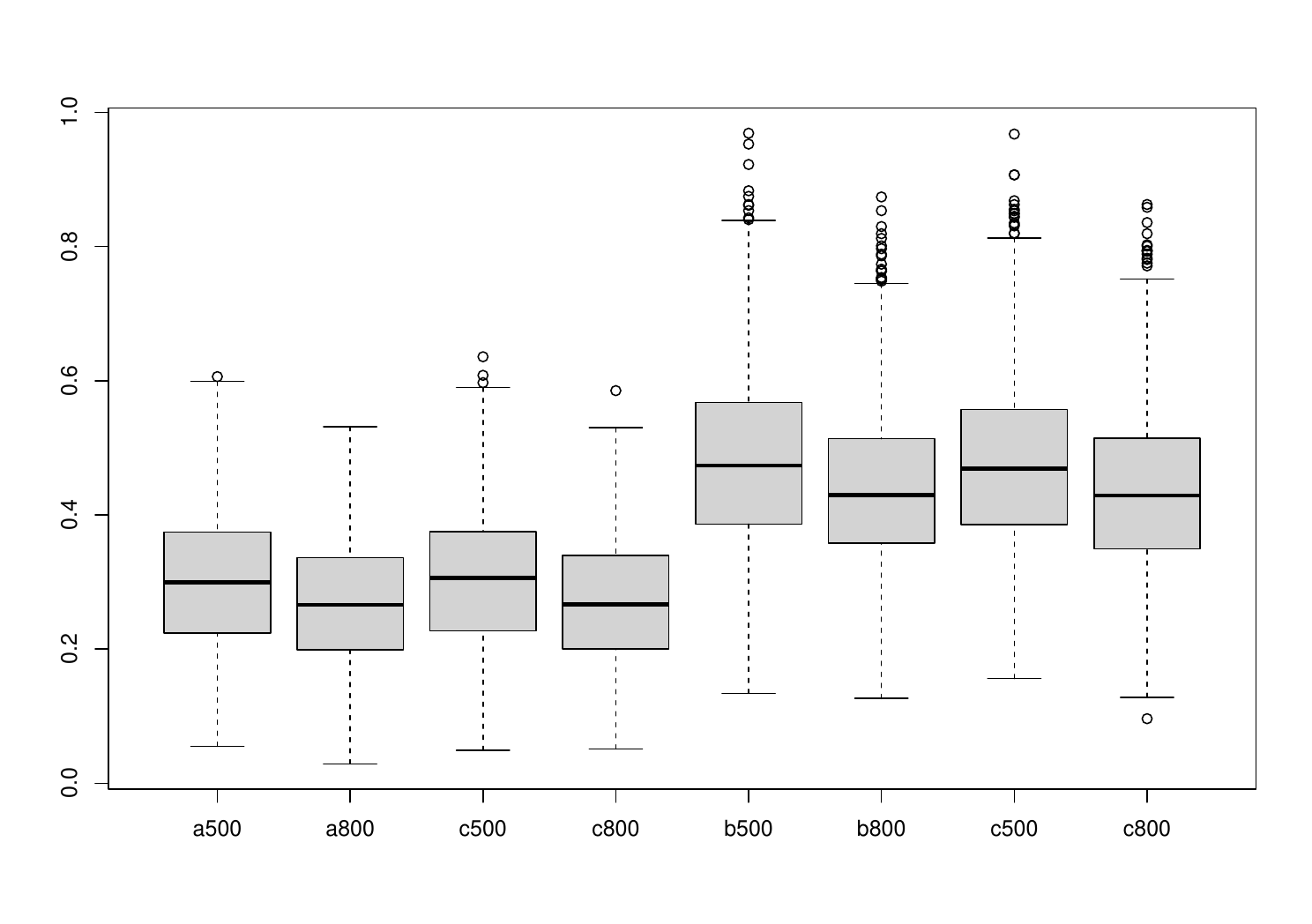}
\caption{\it Box plot of the simulated  estimation error \eqref{holger1} for model (a), (c), (b), (d) with sample size $500$ and $800$, respectively. The label $a500$ means model (a) for sample size $500$. Other labels can be understood similarly.}\label{Boxplot}
\end{figure}
\section{Discussion on the alternative assumptions of Theorem \ref{Thm2.1-2021}}
\setcounter{equation}{0}
In this section, we discuss alternative assumptions for Theorem \ref{Thm2.1-2021}. Some assumptions in the main paper can be relaxed yielding different approximation rates.
\begin{remark}\label{remrk3.1} ~
{\rm 
\begin{itemize}
	\item[(i)]
	A careful inspection of the proofs in Section \ref{proof-main} shows that it 
	is possible to prove similar results   under alternative moment assumptions. 
	For example, Theorem \ref{SCB} holds under the assumption 
	\begin{align}\label{3.9-new}
		\E \Big [  {\sup_{0\leq u,t\leq 1}(G(u,t,\FF_0)} )^4 \Big ] <\infty ~~ .
	\end{align}
	The details are omitted for the sake  of brevity. 
	Note that the $\sup$ in 
	\eqref{3.9-new} appears inside the expectation, while it appears outside 
	the expectation in \eqref{2.7a}. Thus neither \eqref{2.7a} implies 
	\eqref{3.9-new} nor vice versa.
	
	\item[(ii)]
	Assumption \ref{asserr}(2) 
	requires geometric 
	decay  of the dependence measure
	$\delta_q(G,i)$ and a careful inspection of the proofs in Section \ref{proof-main} 
	shows that similar (but weaker) results can be obtained under
	less restrictive assumptions.
	To be precise, define $\Delta_{k,q}=\sum_{i=k}^\infty \delta_q(G,i),$ $\Xi_{M}=\sum_{i=M}^\infty i\delta_2(G,i)$
	and  consider the following assumptions.
	\begin{description}
		\item (a) $\sum_{i=0}^\infty i\delta_3(G,i)<\infty$.
		\item (b) There exist constants  $M=M(n)>0$,  $\gamma=\gamma(n)\in (0,1)$ and $C_2> 0 $ such that 
		\begin{align*}
			(2\lceil nb_n\rceil )^{3/8}M^{-1/2}{l'_n}^{-5/8}\geq C_2l_n'
		\end{align*}
		where $l'_n=\max({\log (2\lceil nb_n\rceil (n-2\lceil nb_n \rceil+1) p/\gamma)},1)$.
	\end{description}
	Then  under the conditions of Theorem \ref{SCB} with Assumption  \ref{asserr}(2)
	replaced by (a) and (b), we have 
	\begin{align*}
		\mathfrak{P}_n
		= O\Big (\eta_n'+\Theta \Big (\sqrt{nb_n}(b_n^4+\frac{1}{n}), np \Big )+\Theta\big( \big((np)^{1/q^*}((nb_n)^{-1}+1/p)\big)^{\frac{q^*}{q^*+1}}, n p\big) \Big )
	\end{align*}
	with 
	\begin{eqnarray*}
		&&  \eta'_n=   (nb_n)^{-1 / 8} M^{1 / 2} {l'}_{n}^{7 / 8}+\gamma+\left((nb_n)^{1 / 8} M^{-1 / 2} {l'}_{n}^{-3 / 8}\right)^{q /(1+q)}\left(np \Delta_{M, q}^{q}\right)^{1 /(1+q)} \\ \nonumber
		&& \quad \quad +\Xi_{M}^{1 / 3}\left(1 \vee \log \left(np / \Xi_{M}\right)\right)^{2 / 3}.
	\end{eqnarray*}
	The same arguments as given in the proof of Theorem \ref{Thm2.1-2021} show that 
	(under the other conditions in this theorem) the set ${\cal C}_n$ defined by  \eqref{n02} 
	defines an (asymptotic) $(1-\alpha)$  simultaneous confidence surface if $\eta_n'=o(1)$. For example, if   $\delta_q(G,i)=O(i^{-1-\alpha})$
	for some $\alpha>0$, $p=n^{\beta}$ for some $\beta>0$ and  $b_n=n^{-\gamma}$ for some  $0<\gamma<1$,  then $\eta_n'=o(1)$ if $(1+\beta)-(1-\gamma)q\alpha/4<0$, which gives a lower bound on $q$.
\end{itemize}
}
\end{remark}

\setcounter{equation}{0}
\section{Simultaneous confidence bands for fixed $u$ or $t$}
\label{secA} 

\subsection{Theoretical background and algorithms} \label{seca1}

In this section, we present the simultaneous confidence band for the regression function
$(u,t) \to m(u,t)$ in model \eqref{1.1}, where one of the arguments  $u$ and $t$ is fixed. \WC{Let $\mathcal C^a$ be the class of functions with Lipschitz continuous $a_{th}$ order derivatives with bounded Lipschitz constant.}   
Consider 
\begin{itemize}
\item [(1)] simultaneous confidence bands {\it for fixed $t$}, which have the form
\begin{equation}
\label{2.1}
{\cal C}(t) =     \big \{
\WC{ f\in \WC{\mathcal C^{3}} }  ~| ~~ \hat L_1 (u,t) \leq f(u)  \leq  \hat U_1 (u,t) ~~\forall u  \big \} ~,
\end{equation}
where $\hat L_1$ and $\hat U_1$  are  appropriate lower and upper bounds calculated from the data.
As $t \in [0,1] $ is fixed  these bounds  can be derived 
generalizing results for confidence bands
in nonparametric regression 
from the independent 
\citep[see][among others]{kanakov1984,xia1998,proksch2016} to the locally stationary case 
\citep[see also][for results in
a model with a stationary  error  process]
{wu2007inference}.
An alternative approach based on multiplier bootstrap will be given below.
\item [(2)] simultaneous confidence bands {\it for fixed $u$}, which have the form
\begin{equation}
\label{2.2}
{\cal C}(u) =     \big \{
\WC{ f\in \WC{\mathcal C^{0}} } 
~| ~~ \hat L_2 (u,t) \leq f(t)  \leq  \hat U_2 (u,t) ~~\forall t \in [0,1] \big \} ~,
\end{equation}
where $\hat L_2$ and $\hat U_2$  are  appropriate lower and upper bounds calculated from the data.
Note that these bounds can not be directly calculated using results of
\cite{dettekokotaue2020}
as these 
authors develop their methodology  under the assumption of stationarity.
\end{itemize}

Recall the definition of the  residuals  $\hat \varepsilon_{i,n}(t)$  and  the long-run variance estimator $\hat \sigma$ 
in the main article.  
For  the construction of a  simultaneous confidence bands for a fixed $t\in[0,1]$ of the form  \eqref{2.1} we define  
\begin{align}
\hat Z_i(u,t)=K\Big(\tfrac{\frac{i}{n}-u}{b_n}\Big)\hat \varepsilon_{i,n}(t),\hat Z_{i,l}(t)=\hat Z_i(\tfrac{l}{n},t),
\notag\\
~~\hat Z^{\hat \sigma}_i(u,t)=K\Big(\tfrac{\frac{i}{n}-u}{b_n}\Big)\tfrac{\hat \varepsilon_{i,n}(t)}{\hat \sigma(\tfrac{i}{n},t)},
~~\hat Z^{\hat \sigma}_{i,l}(t)=\hat Z^{\hat \sigma}_i(\tfrac{l}{n},t). \nonumber 
\end{align}

Next we consider  the $(n-2\lceil nb_n \rceil+1)$-dimensional vectors
\begin{align}\label{new.7.14}
\hat{\tilde Z}_j(t)= \big (
\hat Z_{j,\lceil nb_n\rceil}(t),
\hat	Z_{j+1,\lceil nb_n\rceil+1}(t), 
\ldots ,
\hat	Z_{n-2\lceil nb_n\rceil+j, n-\lceil nb_n \rceil  }(t)
\big)^\top , \\
\hat{\tilde Z}^{\hat \sigma}_j(t)= \big (
\hat Z_{j,\lceil nb_n\rceil}^{\hat \sigma}(t),
\hat	Z_{j+1,\lceil nb_n\rceil+1}^{\hat \sigma}(t), 
\ldots ,
\hat	Z_{n-2\lceil nb_n\rceil+j, n-\lceil nb_n \rceil  }^{\hat \sigma}(t)
\big)^\top    \label{new.7.15}
\end{align}
$(1\leq j \leq 2\lceil nb_n \rceil-1)$, 
then a simultaneous confidence band for fixed $t\in [0,1]$ can be generated by the Algorithms \ref{algorithmt1}  (constant width)  and Algorithm \ref{algorithmt2} (varying width).
\newpage

\begin{small}
\begin{algorithm}[H]
\KwResult{ simultaneous confidence band of the form \eqref{2.1}  with fixed width 
}
\begin{itemize} 
	\item[(a)]  Calculate the 
	the $(n-2\lceil nb_n\rceil +1)$-dimensional vector   $\hat{\tilde Z}_j(t)$
	in \eqref{new.7.14}\;
	\item[(b)] For window size $m_n$, 
	let $m_n'=2\lfloor m_n/2\rfloor$, define 
	\begin{align}
		\label{hol10-fixt}
		\hat S_{jm_n'}(t)=
		\frac{1}{\sqrt {m_n'}}	\sum_{r=j}^{j+\lf m_n/2\rf -1 }\hat{\tilde  Z}_{r}(t)-	\frac{1}{\sqrt {m_n'}}\sum_{r=j+\lf m_n/2\rf }^{j+m'_n-1 }\hat{\tilde  Z}_{r}(t)
	\end{align} 
	Let  
	$\hat{ \varepsilon}_{j:j+m_n',k}(t)$ be the	
	$k_{th}$ component of
	$\hat S_{jm_n'}(t)$. 
	\item[(c)]  \For{r=1, \ldots , B  }{
		\smallskip
		- Generate independent standard normal distributed random variables $\{R^{(r)}_i\}_{i\in [1, n-m_n']}$.
		- Calculate 
		\begin{align*} \nonumber 
			T_k^{(r)}(t)& =\sum_{j=1}^{2\lceil nb_n \rceil-m_n'}
			\hat{ \varepsilon}_{j:j+m_n',k}(t) 
			{R^{(r)}_{k+j-1} }~,~~ k= 1 , \ldots  , n-2\lceil nb_n\rceil+1  , \\
			T^{(r)}(t) & =
			\max_{1\leq k\leq n-2\lceil nb_n\rceil+1} |T_k^{(r)}(t)|.
			\nonumber 
		\end{align*} 
		
		\bf{end}
	}
	\smallskip
	\item [(d)]
	Define $T_{\lf (1-\alpha)B\rf}(t)$ as the empirical $(1-\alpha)$-quantile of the  sample
	$T^{(1)}(t), \ldots , T^{(B)}(t)$ and 
	\begin{align*}
		\hat   L_3 (u,t)   = \hat m(u,t)  - \hat r_3(t) ~~,~~~~
		\hat  U_3 (u,t) =  \hat m(u,t) +   \hat r_3(t)
	\end{align*}
	where 
	$$
	\hat r_3(t) =\frac{\sqrt 2 T_{\lf (1-\alpha)B\rf} (t)}{\sqrt{nb_n}\sqrt{ 2\lceil nb_n\rceil-m'_n}
	} 
	$$
	{\bf Output:}   ~~
	$  {\cal C}_n (t)=     \big \{ f\in \WC{\WC{\mathcal C^{3}}} : [0,1]^2 \to \R ~| ~~ \hat L_3 (u,t) \leq f(u)  \leq  \hat U_3 (u,t) ~~\forall   u \in  {[b_n,1-b_n]}  \big \} 
	$
\end{itemize}
\caption{}\label{algorithmt1}
\end{algorithm}
\end{small}

\begin{small}

\begin{algorithm}[H]
\KwResult{ simultaneous confidence band of the form \eqref{2.1}  
	with varying width
}

\begin{itemize} 
	\item[(a)] Calculate the  estimate of the long-run variance  $\hat \sigma^2$
	in \eqref{2018-5.4}
	
	\item[(b)]  Calculate the 
	$(n-2\lceil nb_n\rceil +1)$-dimensional  
	vectors   $\hat{\tilde Z}^{\hat \sigma}_j(t)$
	in \eqref{new.7.15}
	\item[(c)] For window size $m_n$, 
	let $m_n'=2\lfloor m_n/2\rfloor$, define 
	\begin{align*}
		\hat S^{\hat \sigma}_{jm_n'}(t)& =
		\frac{1}{\sqrt {m_n'}}	\sum_{r=j}^{j+\lf m_n/2\rf -1 }\hat{\tilde  Z}^{\hat \sigma}_{r}(t)-	\frac{1}{\sqrt {m_n'}}\sum_{r=j+\lf m_n/2\rf}^{j+m'_n-1 }\hat{\tilde  Z}^{\hat \sigma}_{r}(t)
	\end{align*} 
	Let  $\hat S^{\hat \sigma}_{jm_n',k}(t)$ be the $k_{th}$ component of	$\hat S^{\hat \sigma}_{jm_n'}(t)$.
	Let  
	$\hat{ \varepsilon}^{\ \hat \sigma}_{j:j+m_n',k}(t)$ be the	
	$k_{th}$ component of
	$\hat S^{\hat \sigma}_{jm_n'}(t)$. 
	\item[(d)]  \For{r=1, \ldots , B  }{
		\smallskip
		- Generate independent standard normal distributed random variables $\{R^{(r)}_i\}_{i\in [1,n-m_n']}$.
		
		- Calculate 
		\begin{align*} \nonumber 
			T_k^{\hat \sigma,(r)}(t)& =\sum_{j=1}^{2\lceil nb_n \rceil-m_n'}
			\hat{ \varepsilon}^{\ \hat \sigma}_{j:j+m_n',k}(t) {R^{(r)}_{k+j-1}} ~,~~ k= 1 , \ldots  , n-2\lceil nb_n\rceil+1  , \\
			T^{\hat \sigma,(r)}(t) & =
			\max_{1\leq k\leq n-2\lceil nb_n\rceil+1} |T_k^{\hat \sigma,(r)}(t)|.
			\nonumber 
		\end{align*} 
		
		\bf{end}
	}
	\smallskip
	\item [(e)]
	Define $T^{\hat \sigma}_{\lf (1-\alpha)B\rf}(t)$ as the empirical $(1-\alpha)$-quantile of the sample
	$T^{\hat \sigma,(1)}(t), \ldots , T^{\hat \sigma,(B)}(t)$ and 
	\begin{align*}
		\hat   L^{\hat \sigma}_4 (u,t)  & = \hat m(u,t) -
		\hat r_4 (u,t),~~~~
		\hat  U^{\hat \sigma}_4 (u,t)  =  \hat m(u,t) +  \hat r_4 (u,t) 
	\end{align*}
	where
	$$
	\hat r_4 (u,t) =  \frac{\hat \sigma(u,t)\sqrt 2 T^{\hat \sigma } _{\lf (1-\alpha)B\rf}(t) }{\sqrt{nb_n}\sqrt{ 2\lceil nb_n\rceil-m'_n}}
	$$
	{\bf Output:} 
	\begin{equation*} 
		\label{3.12-v3}
		{\cal C}^{\hat \sigma}_n(t) =     \big \{ f\in \WC{\WC{\mathcal C^{3}}} : [0,1]^2 \to \R ~| ~~ \hat L^{\hat \sigma}_4 (u,t) \leq f(u)  \leq  \hat U^{\hat \sigma}_4 (u,t) ~~\forall   u \in  {[b_n,1-b_n]} \big \}. 
	\end{equation*}
\end{itemize}
\caption{ }\label{algorithmt2}
\end{algorithm}

\end{small}

The following result shows that the sets constructed  by  Algorithms  \ref{algorithmt1} and \ref{algorithmt2} 
are asymptotic $(1-\alpha)$-confidence bands of the form \eqref{2.1}.
The proof  is similar to but easier than the proof of Theorems \ref{Thm2.1-2021} and  \ref{sigma-asy}
is therefore omitted for the sake of brevity.

\begin{theorem}
\label{fixt-SCB}Assume that the conditions of Theorem \ref{SCB} hold.
Define
$$\vartheta^\dag_n=\frac{\log^2 n}{m_n}+\frac{m_n\log n}{nb_n}+\sqrt{\frac{m_n}{nb_n}}n^{4/q}.
$$
\begin{description}
\item (i) 
If 
$\vartheta^{\dag,1/3}_n\{1\vee \log (\frac{n}{\vartheta^\dag_n})\}^{2/3}+\Theta(\big(\sqrt{m_n\log n}(\frac{1}{\sqrt{nb_n}}+b_n^3)(n)^{\frac{1}{q}}\big)^{q/(q+1)},n)=o(1)$
we have that for any  $\alpha\in(0,1)$
and any $t \in [0,1]$ 
\begin{align*}
	\lim_{n\rightarrow \infty} \lim_{B\rightarrow \infty}  \p(m \in \mathcal C_n(t)  ~ |~{\cal F}_n)=1-\alpha
\end{align*}
in probability.
\item (ii) If further the conditions of Theorem \ref{sigma-asy} and Proposition \ref{prop-lrv} hold, then
\begin{align*}
	\lim_{n\rightarrow \infty} \lim_{B\rightarrow \infty}  \p(m \in \mathcal C^{\hat \sigma}_n(t)  ~ |~{\cal F}_n)=1-\alpha
\end{align*}
in probability.
\end{description}
\end{theorem}

The next theorem presents a Gaussian approximation in the case where $u$ is fixed.  It is the basis for the construction of a confidence band for fixed $u$ and its proof
follows by 
similar (but easier)  arguments as given in  the proof 
of  Theorem \ref{SCB}.

\begin{theorem}\label{single-SCB}
Let Assumptions \ref{assreg} -  \ref{asskern}
be satisfied and  assume that the bandwidth in \eqref{estimating-m}
satisfies that $n^{1+a}b_n^9=o(1)$, $n^{a-1}b_n^{-1}=o(1)$ for  some $0<a<4/5$.
For any fixed 		$u \in (0,1)$ there exists 
a sequence of centered   $p$-dimensional Gaussian vectors $ (Y_i(u) )_{i \in \mathbb{N}}$  with the same  covariance structure as  the vector  $Z_i(u)$  in \eqref{3.1}, such that
\begin{align*}
\mathfrak{P}_n(u)& :=
\sup_{x\in \mathbb R} \Big |\p \Big  (\max_{0\leq t\leq 1}\sqrt{nb_n}|\hat \Delta(u,t)|\leq x \Big )-\p \Big ( \Big  |\frac{1}{\sqrt{nb_n}}\sum_{i=1}^nY_i(u) \Big |_\infty\leq x\Big  )  \Big  | \\
& = O \Big ((nb_n)^{-(1-11\iota)/8}  +\Theta \Big (\sqrt{nb_n}\Big(b_n^4+\frac{1}{n}\Big), p\Big )+\Theta \Big  (p^{\frac{1-q^*}{1+q^*}}, p \Big  )\Big ) 
\end{align*}
for any sequence $ p \to \infty$ with 	$p = O (\exp(n^{\iota}))$ 	 for some $0\leq \iota <1/11$. 
In particular, $
\mathfrak{P}_n(u)= o(1)
$ 
if $p=n^c$ for some $c>0$ and the constant $q^*$  in Assumption \ref{asssim} is sufficiently large.
\end{theorem}

\newpage

\begin{small}
\begin{algorithm}[H]
\medskip

\SetAlgoLined
\KwResult{  simultaneous confidence band     for fixed $u\in [b_n,1-b_n]$ as defined in \eqref{2.2} 
}

\begin{itemize} 
	\item[(a)]
	Calculate the $p$-dimensional  vectors  $\hat{Z}_i(u)$
	in \eqref{3.10} 
	\item[(b)] For window size $m_n$, 
	let $m_n'=2\lfloor m_n/2\rfloor$,  define 
	\begin{align*}
		\hat S_{jm_n'}(u)=
		\frac{1}{\sqrt {m_n'}}	\sum_{r=j}^{j+\lf m_n/2\rf -1 }\hat{ Z}_{r}(u)-	\frac{1}{\sqrt {m_n'}}\sum_{r=j+\lf m_n/2 \rf }^{j+m'_n-1 }\hat{ Z}_{r}(u)
	\end{align*} 
	\item[(c)] 
	\For{r=1, \ldots , B  }{
		
		-  Generate independent standard normal distributed random variables $
		\{R_i^{(r)} \}_{i = \lceil nu-nb_n \rceil }^{ \lf nu+nb_n\rf}
		$
		- Calculate  the bootstrap statistic  
		\begin{align*}
			T^{(r)}(u) 
			= \Big |\sum_{j=\lceil nu-nb_n\rceil}^{\lf nu+nb_n\rf-m_n'+1} \hat S_{jm_n'}(u)R_j^{(r)}\Big |_\infty
		\end{align*}
		~
	}
	
	\smallskip 
	\item[(d)] Define  $T_{\lf (1-\alpha)B\rf}(u)$ as the empirical $(1-\alpha)$-quantile of the  sample $T^{(1)}(u), \ldots , T^{(B)}(u)$ and  
	\begin{align*}
		\hat   L_5 (u,t)   = \hat m(u,t) - \hat r_5(u)~~,~ ~~~
		\hat  U_5 (u,t) =  \hat m(u,t) + \hat r_5(u)~,
	\end{align*}
	where
	$$\hat r_5(u)=\frac{\sqrt 2 T_{\lf (1-\alpha)B\rf}(u) }{\sqrt{nb_n}\sqrt{( \lfloor nu+nb_n\rfloor-\lceil nu-nb_n \rceil -m'_n+2)}}
	$$
	{\bf Output:}
	\begin{equation} \label{nn01}
		{\cal C}_n (u)  =     \big \{ f\in \WC{\WC{\mathcal C^{0}}} : [0,1]^2 \to \R ~| ~~ \hat L_5 (u,t) \leq f(t)  \leq  \hat U_5 (u,t) ~~\forall  t \in  [0,1] \big \} .
	\end{equation}
\end{itemize}

\caption{}
\label{algorithm0}
\end{algorithm}
\end{small}


\begin{small}
\begin{algorithm} [H]

\medskip

\SetAlgoLined
\KwResult{simultaneous confidence band  of the form \eqref{2.2} with varying width.
}

\begin{itemize} 
	\item[(a)] 
	For given $u\in [b_n,1-b_n]$, calculate the the estimate of the long-run variance  $\hat \sigma^2(u,\cdot)$
	in \eqref{2018-5.4}
	
	\item[(b)]
	Calculate the vector   $\hat{Z}^{\hat \sigma_u}_i(u)$
	in \eqref{Z(u)sigma}\;
	\item[(c)] For window size $m_n$, 
	let $m_n'=2\lfloor m_n/2\rfloor$ and define 
	the $p$-dimensional  random vectors
	\begin{align*}
		\hat S^{\hat \sigma_u}_{jm_n'}(u)& =
		\frac{1}{\sqrt {m_n'}}	\sum_{r=j}^{j+\lf m_n/2\rf-1 }\hat{  Z}^{\hat \sigma_u}_{r}(u)-	\frac{1}{\sqrt {m_n'}}\sum_{r=j+\lf m_n/2\rf}^{j+m'_n-1 }\hat{  Z}^{\hat \sigma_u}_{r}(u)
	\end{align*} 
	\item[(d)] \For{r=1, \ldots , B  }{
		\smallskip
		
		-  Generate independent standard normal distributed random variables $
		\{R_i^{(r)} \}_{i = \lceil nu-nb_n \rceil}^{   \lf nu+nb_n\rf}
		$
		- Calculate  the bootstrap statistic  
		\begin{align*}
			T^{\hat \sigma_u, (r)}(u) 
			= \Big |\sum_{j=\lceil nu-nb_n\rceil}^{\lf nu+nb_n\rf-m_n'+1} \hat S^{\hat \sigma_u}_{jm_n'}(u)R_j^{(r)}\Big |_\infty
		\end{align*}

	}
	
	\smallskip 
	\item[(e)] Define  $T^{\hat \sigma_u}_{\lf (1-\alpha)B\rf}(u)$ as the empirical $(1-\alpha)$-quantile of the sample $T^{\hat \sigma_u,(1)}(u), \ldots , T^{\hat \sigma_u,(B)}(u)$ and 
	\begin{align*}
		\hat   L^{\hat \sigma_u}_6 (u,t)   = \hat m(u,t) - \hat r^{\hat \sigma_u}_6(u,t)~~,~ ~~~
		\hat  U^{\hat \sigma_u}_6 (u,t) =  \hat m(u,t) + \hat r^{\hat \sigma_u}_6(u,t),
	\end{align*}
	where
	$$\hat r^{\hat \sigma_u}_6(u,t)=\frac{\hat \sigma(u,t)\sqrt 2 T^{\hat \sigma_u}_{\lf (1-\alpha)B\rf}(u) }{\sqrt{nb_n}\sqrt{( \lfloor nu+nb_n\rfloor-\lceil nu-nb_n \rceil -m'_n+2)}}
	$$
	{\bf Output:}
		\begin{equation}
			\label{nn02}
			{\cal C}^{\hat \sigma_u}_n (u)  =     \big \{ f \in \WC{\WC{\mathcal C^{0}}}: [0,1]^2 \to \R ~| ~~ \hat L^{\hat \sigma_u}_6 (u,t) \leq f(t)  \leq  \hat U^{\hat \sigma_u}_6 (u,t) ~~\forall  
			{t\in [0,1]} \big \}.  
		\end{equation}  
\end{itemize}

\caption{}\label{algorithm0sigma}
\end{algorithm}
\end{small}


Next we present details of the algorithms for a simultaneous confidence band for a fixed $u$ (of the form \eqref{2.2})  with fixed and varying width. 
For this purpose we define the $p$-dimensional vector
\begin{align}\label{Z(u)sigma}
\hat Z^{\hat \sigma_u}_{i}(u)
& = (\hat Z^{\hat \sigma_u}_{i,1}(u), \ldots , \hat Z^{\hat \sigma_u}_{i,p}(u))^\top \\
\nonumber 
& = K \Big (\frac{\tfrac{i}{n}-u}{b_n} \Big ) \Big  (\tfrac{\hat \varepsilon_{i,n}(\tfrac{1}{p})}{\hat \sigma(u,\tfrac{1}{p})}
,\tfrac{\hat \varepsilon_{i,n}(\tfrac{2}{p})}{\hat \sigma(u,\tfrac{2}{p})}
,\ldots ,\tfrac{\hat \varepsilon_{i,n} (\tfrac{p-1}{p})}{\hat \sigma(u,\tfrac{p-1}{p})},
\tfrac{\hat \varepsilon_{i,n}(1)}{\hat \sigma(u,1)}
\Big )^\top~,
\end{align}
where $\hat \varepsilon_{i,n} $ and $\hat \sigma $ are defined 
in the main article, respectively. 
Algorithms  \ref{algorithm0} and \ref{algorithm0sigma} provides asymptotically correct the confidence bands of type \eqref{2.2}. The next Theorem \ref{fixu-SCB} yields the validity of Algorithms  \ref{algorithm0} and \ref{algorithm0sigma}, which is a consequence of Theorem  \ref{single-SCB}.

\begin{theorem}
\label{fixu-SCB}
Assume that the conditions of Theorem \ref{SCB} hold.
Define
$$\vartheta'_n=\frac{\log^2 n}{m_n}+\frac{m_n\log n}{nb_n}+\sqrt{\frac{m_n}{nb_n}}p^{4/q}
$$
and assume that 
$ p \to \infty$ such that  
$p = O (\exp(n^{\iota}))$ 	 for some $0\leq \iota <1/11$.
\begin{description}
\item(1) If $\alpha\in (0,1)$ and 
$$
\vartheta'^{1/3}_n
\Big \{1\vee \log \Big(\frac{p}{\vartheta'_n}\Big)\Big\}^{2/3}+\Theta \Big (\Big(\sqrt{m_n\log p} \Big (\frac{1}{\sqrt{nb_n}}+b_n^3 \Big )p^{\frac{1}{q}}\Big)^{q/(q+1)},p \Big )=o(1),
$$
then we have  for the confidence band in \eqref{nn01}
\begin{align*}
	\lim_{n\rightarrow \infty} \lim_{B\rightarrow \infty}  \p(m \in \mathcal C_n(u)  ~ |~{\cal F}_n)=1-\alpha
\end{align*}
in probability.
\item (ii) If further the conditions of Theorem \ref{sigma-asy} and Proposition \ref{prop-lrv} hold, then
for the confidence band in
\eqref{n02}
\begin{align*}
	\lim_{n\rightarrow \infty} \lim_{B\rightarrow \infty}  \p(m \in \mathcal C^{\hat \sigma}_n(u)  ~ |~{\cal F}_n)=1-\alpha
\end{align*}
in probability.
\end{description}

\end{theorem}

The proof of Theorem \ref{fixu-SCB}
follows by similar (but easier)  arguments as given in  the proof 
of   Theorem \ref{Thm2.1-2021} and Theorem  \ref{sigma-asy}.
\begin{remark} \label{assA1}
{\rm One can	prove similar results   under alternative moment assumptions.
In fact, Theorem \ref{single-SCB} remains valid  if condition  \eqref{2.7a}	is  replaced by  \begin{align}\label{3.8-new} \E \Big [ {\sup_{0\leq t\leq 1}(G(u,t,\FF_0)} )^4 \Big ]<\infty~~. 
\end{align}
Moreover, one can prove   Theorem \ref{single-SCB} under weaker assumptions than  Assumption  \ref{asserr} (ii), which requires geometrically decaying dependence measure. More precisely, If the assumptions of Theorem \ref{single-SCB} hold, 
where Assumption  \ref{asserr} (ii) is
replaced by assumption (a) in (ii) of Remark \ref{remrk3.1}  and the following conditions
\begin{description}
	\item (b1) There exist constants 
	$M=M(n)>0$, $\gamma=\gamma(n)\in (0,1)$   and $C_1> 0 $ such that 
	\begin{align*}
		(2\lceil nb_n\rceil )^{3/8}M^{-1/2}l_n^{-5/8}\geq C_1l_n
	\end{align*}
	where $l_n=\max(\log ({2\lceil nb_n \rceil}p/\gamma),1)$.
\end{description}
Recall the quantity $\Xi_M$ and $\Delta_{M,q} $ defined in Remark \ref{remrk3.1}.				Then we have  
\begin{align*}
	\mathfrak{P}_n(u) =O \Big(\eta_n  +\Theta \Big (\sqrt{nb_n}(b_n^4+\frac{1}{n}), p\Big )+\Theta \Big  (p^{\frac{1-q^*}{1+q^*}}, p \Big  )\Big)
\end{align*}
with
\begin{eqnarray*}
	&& \eta_n   =    (nb_n)^{-1 / 8} M^{1 / 2} l_{n}^{7 / 8}+\gamma+\left((nb_n)^{1 / 8} M^{-1 / 2} l_{n}^{-3 / 8}\right)^{q /(1+q)}\left(p \Delta_{M, q}^{q}\right)^{1 /(1+q)} \\ \nonumber
	&& \quad \quad  +\Xi_{M}^{1 / 3}\left(1 \vee \log \left(p / \Xi_{M}\right)\right)^{2 / 3}.
\end{eqnarray*}
By similar arguments   as given in Remark \ref{remrk3.1}, the sets ${\cal C}_n(u)$ and ${\cal C}^{\hat \sigma_u}_n(u)$ defined by  \eqref{nn01} and \eqref{nn02}, respectively,  
define an (asymptotic) $(1-\alpha)$  simultaneous confidence surface if $\eta_n=o(1)$. For example, if   $\delta_q(G,i)=O(i^{-1-\alpha})$
for some $\alpha>0$, $p=n^{\beta}$ for some $\beta>0$ and  $b_n=n^{-\gamma}$ for some  $0<\gamma<1$,  then $\eta_n=o(1)$ if $\beta-(1-\gamma)q\alpha/4<0$, which gives a lower bound on $q$.
}
\end{remark}


\subsection{Finite sample properties} \label{sec8}

In this section we provide numerical results for the confidence bands for 
the regression function $m$ with fixed $u$ or $t$
derived in Algorithms
\ref{algorithmt1} - \ref{algorithm0sigma}.
As in the main part of the paper we consider simulated and real data.

For the simultaneous confidence band for a fixed $t\in[0,1]$ in \eqref{2.1} and a fixed $u \in (0,1)$ in \eqref{2.2}, the tuning parameters are chosen 
in a similar way as described in Section \ref{sec51}. \WC{In particular for a fixed $u\in (0,1)$ 
use the bandwidth $b_n$ as the minmizer of the  loss 
function 
\begin{align} \label{det2a}
MGCV(b)=\max_{1\leq s\leq p}\frac{\sum_{i=\lceil nu-nb_n\rceil}^{\lfloor nu+nb_n\rfloor}(\hat m_b(\tfrac{i}{n},\tfrac{s}{p})-X_{i,n}(\tfrac{s}{p}))^2}{(1-{\rm tr} (Q_s(b,u))/(\lf nu+nb_n\rf-\lceil nu-nb_n\rceil+1))^2}.
\end{align}
and $Q_s(b,u)$ is the submatrix of $Q_s(b)$ defined in \eqref{det2} 
consisting of  $\lceil nu-nb_n\rceil:\lf nu+nb_n\rf_{th}$ rows and lines.}
The criterion \eqref{det2a} is also motivated by the generalized cross validation criterion introduced by \cite{craven1978smoothing}.

\subsection{Simulated data} \label{sec81}
For simulated data, the regression functions and locally stationary functional time series  are stated in Section \ref{sec52}. We begin  displaying  typical 
$95\%$ simultaneous confidence bands
obtained from one simulation run
for  model (a) with sample size  $n=800$. 
Figure \ref{simu-fixt-SCB} shows the simultaneous band of the type \eqref{2.1} with constant width (Algorithm \ref{algorithmt1})
and variable width 
(Algorithm \ref{algorithmt2}),
while in Figure  \ref{simu-fixu-SCB}   we display the
simultaneous confidence
bands   of the form \eqref{2.2} (for fixed $u$) 
with constant width (Algorithm \ref{algorithm0})
and variable width 
(Algorithm \ref{algorithm0sigma}).  We observe that in all cases there exist  differences between the bands with constant and variable width, but they are not substantial.\\

\begin{figure}[htbp]
\centering
\includegraphics[width=14cm, height=6cm]{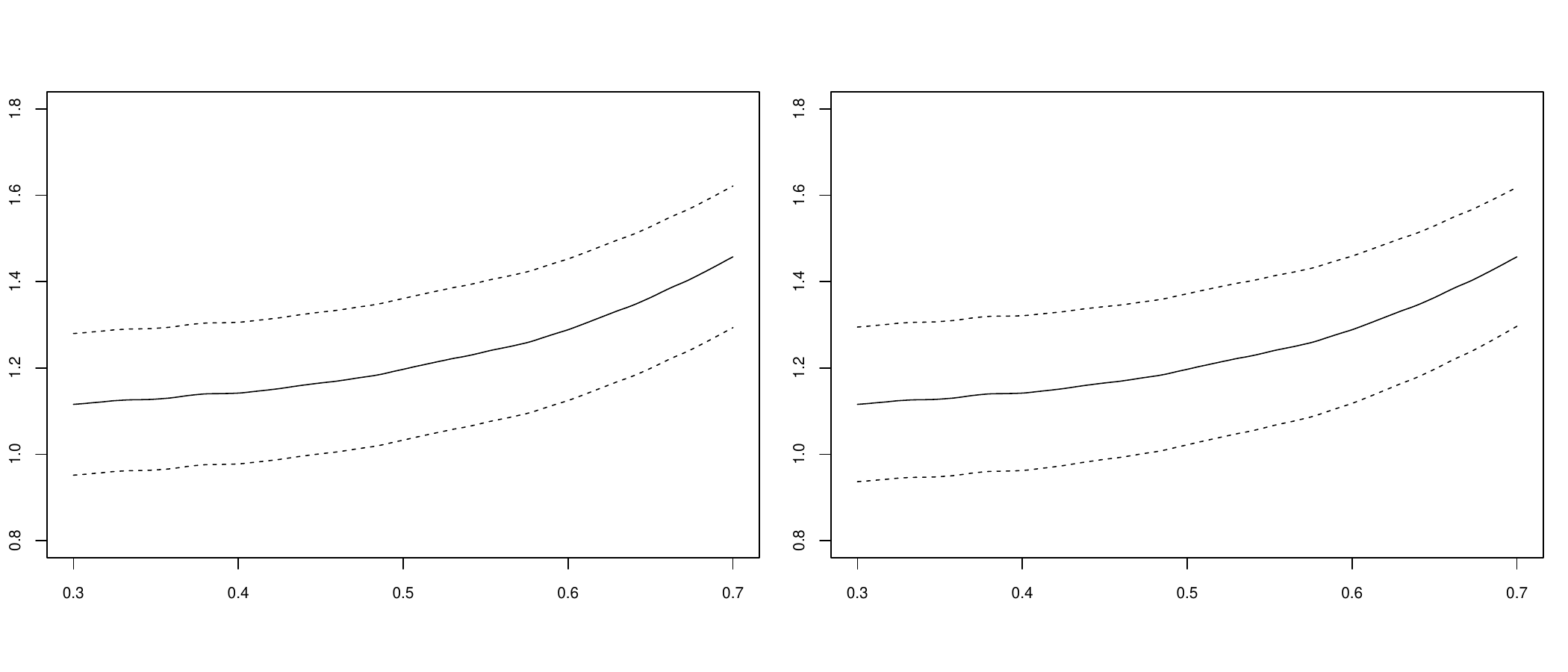}
\caption{\it $95\%$ simultaneous confidence bands of
the form \eqref{2.1} (fixed $t=0.5$) for the regression function in model (c) from  $n=800$
observations. Left panel: constant width (Algorithm \ref{algorithmt1}); Right panel:
varying width (Algorithm \ref{algorithmt2}).}
\label{simu-fixt-SCB}
\end{figure}

We next investigate the  coverage probabilities 
of  confidence bands constructed for fixed $t=0.5$ and $u=0.5$ for sample sizes $n=500$ and $n=800$.
All results presented in the following discussion are based on $1000$ simulation runs and $B=1000$ bootstrap replications. In all  tables  the left part shows the coverage probabilities of the bands with constant width while the results in the right part correspond to the bands with varying width.

In Table  \ref{tablet1} we give some results for the confidence bands of the form \eqref{2.1} (for fixed $t=0.5$)  with constant  and variable width 
(c.f. Algorithm {\ref{algorithmt1}}  and Algorithm {\ref{algorithmt2}}),
while we present in  Table \ref{tablef1}  the simulated coverage probabilities 
of the simultaneous confidence bands 
of the form \eqref{2.2}, where  $u=0.5$ is fixed
(c.f. Algorithm \ref{algorithm0} and Algorithm \ref{algorithm0sigma}).
We  observe that the simulated coverage probabilities are  close to their nominal levels in all cases under consideration, which illustrates the validity of our methods for finite sample sizes.
\begin{figure}[htbp]
\centering
\vspace{-0.5cm}
\includegraphics[width=14cm, height=7cm]{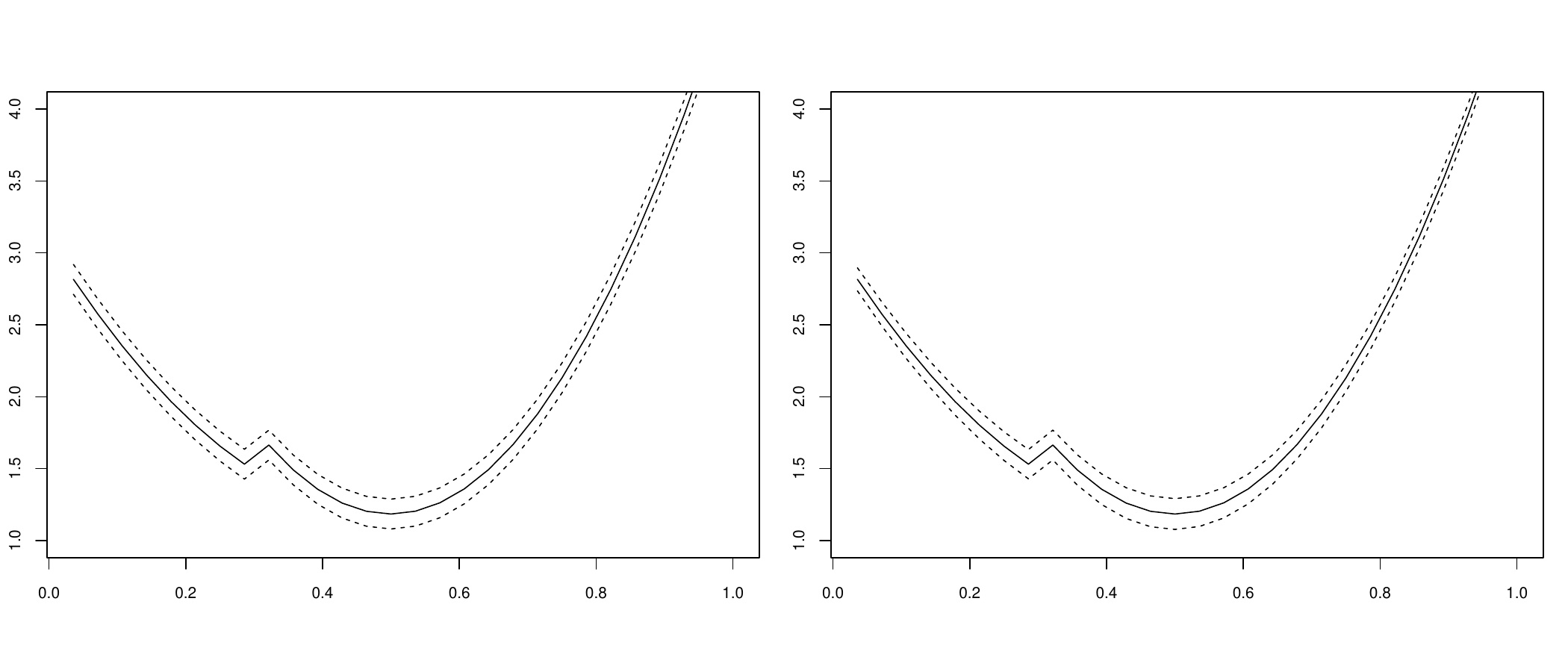}
\caption{\it $95\%$ simultaneous confidence band of
the form   \eqref{2.2} (fixed $u=0.5$) for the regression function in model (c) from  $n=800$
observations. Left panel: constant width (Algorithm   \ref{algorithm0}); Right panel:
varying width (Algorithm   \ref{algorithm0sigma}).}
\label{simu-fixu-SCB}
\end{figure}

\begin{table}[H]
\caption{\label{tablet1} \it Simulated coverage probabilities of the simultaneous confidence band
of the form   \eqref{2.1}  for   fixed $t=0.5$ calculated by Algorithm   \ref{algorithmt1}
(constant width) and   \ref{algorithmt2} (varying width). 
}
\centering
\begin{tabular}{lcccccccc}
\toprule
& \multicolumn{4}{c}{Constant Width} & \multicolumn{4}{c}{Varying Width} \\
\cmidrule(lr){2-5} \cmidrule(lr){6-9}
& \multicolumn{2}{c}{Model (a)} & \multicolumn{2}{c}{Model (b)} & \multicolumn{2}{c}{Model (a)} & \multicolumn{2}{c}{Model (b)} \\
\cmidrule(lr){2-3} \cmidrule(lr){4-5} \cmidrule(lr){6-7} \cmidrule(lr){8-9}
Level   & 90\% & 95\% & 90\% &  95\% & 90\% & 95\% & 90\% & 95\% \\
\cmidrule(lr){1-9}
n=500 & 90.3\% & 95.0\% & 91.7\% & 96.0\% & 91.2\% & 95.6\% & 91.3\% & 96.2\% \\
n=800 & 88.5\% & 95.4\% & 88.7\% & 94.4\% & 88.8\% & 94.5\% & 88.4\% & 94.0\% \\
\midrule
& \multicolumn{2}{c}{Model (c)} & \multicolumn{2}{c}{Model (d)} & \multicolumn{2}{c}{Model (c)} & \multicolumn{2}{c}{Model (d)} \\
\cmidrule(lr){2-3} \cmidrule(lr){4-5} \cmidrule(lr){6-7} \cmidrule(lr){8-9}
Level   & 90\% & 95\% & 90\% &  95\% & 90\% & 95\% & 90\% & 95\% \\
\cmidrule(lr){1-9}
n=500 & 91.7\% & 96.3\% & 91.4\% & 95.6\% & 91.5\% & 95.4\% & 90.4\% & 94.1\% \\
n=800 & 89.1\% & 94.8\% & 89.8\% & 94.5\% & 87.5\% & 93.4\% & 88.7\% & 94.4\% \\
\bottomrule
\end{tabular}
\end{table}

\begin{table}[H]
\caption{\label{tablef1} \it Simulated coverage probabilities of the simultaneous confidence band
of the form  \eqref{2.2}  for fixed $u=0.5$ calculated by  Algorithms  \ref{algorithm0} 
(constant width) and  \ref{algorithm0sigma} (varying width).}
\centering
\begin{tabular}{lcccccccc}
\toprule
& \multicolumn{4}{c}{Constant Width} & \multicolumn{4}{c}{Varying Width} \\
\cmidrule(lr){2-5} \cmidrule(lr){6-9}
& \multicolumn{2}{c}{Model (a)} & \multicolumn{2}{c}{Model (b)} & \multicolumn{2}{c}{Model (a)} & \multicolumn{2}{c}{Model (b)} \\
\cmidrule(lr){2-3} \cmidrule(lr){4-5} \cmidrule(lr){6-7} \cmidrule(lr){8-9}
Level   & 90\% & 95\% & 90\% &  95\% & 90\% & 95\% & 90\% & 95\% \\
\cmidrule(lr){1-9}
n=500 & 87.0\% & 93.4\% & 88.4\% & 93.5\% & 86.9\% & 92.2\% & 88.7\% & 93.7\% \\
n=800 & 88.7\% & 93.7\% & 88.4\% & 94.7\% & 89.4\% & 94.4\% & 88.9\% & 94.1\% \\
\midrule
& \multicolumn{2}{c}{Model (c)} & \multicolumn{2}{c}{Model (d)} & \multicolumn{2}{c}{Model (c)} & \multicolumn{2}{c}{Model (d)} \\
\cmidrule(lr){2-3} \cmidrule(lr){4-5} \cmidrule(lr){6-7} \cmidrule(lr){8-9}
Level   & 90\% & 95\% & 90\% &  95\% & 90\% & 95\% & 90\% & 95\% \\
\cmidrule(lr){1-9}
n=500 & 86.6\% & 92.3\% & 90.2\% & 94.0\% & 90.2\% & 94.5\% & 89.5\% & 94.2\% \\
n=800 & 89.6\% & 94.7\% & 87.8\% & 93.3\% & 88.9\% & 93.4\% & 89.8\% & 94.1\% \\
\bottomrule
\end{tabular}
\end{table}

\subsection{Real data}
\label{Section-Real Data}
In this section we further study the  well documented volatility smile for implied volatility of the European call option of SP500 data set considered in Section \ref{sec3} of the main article. In Figure \ref{Vol-fixt-SCB} of we display  $95\%$  simultaneous confidence bands of the form \eqref{2.1} for fixed $t=0.5$ (which corresponds to  Moneyness=$1.1$)
where the parameters are chosen as  $b_n=0.12$ and $m_n=18$. 
We observe that the implied volatility  changes with time (or precisely the time to maturity) when moneyness (or equivalently, the strike price and underlying asset price) is specified.
We also calculate confidence bands of the form \eqref{2.2}
for fixed $u=0.5$, by  
Algorithm \ref{algorithm0}  (constant width) and Algorithm \ref{algorithm0sigma} (varying width).
The parameter selection procedure yields  $b_n=0.1$ and $m_n=32$, and  the resulting simultaneous confidence bands of the form \eqref{2.2} are presented in Figure \ref{Vol-fixu-SCB}.
We observe  that both 95\% simultaneous confidence bands indicate that  the implied volatility is a quadratic function of moneyness, which supports the well documented phenomenon of 'volatility smile’.
We observe that the differences between the bands with constant and variable width are rather small.

\begin{figure}[htbp]
\centering
\includegraphics[width=14cm, height=6cm]{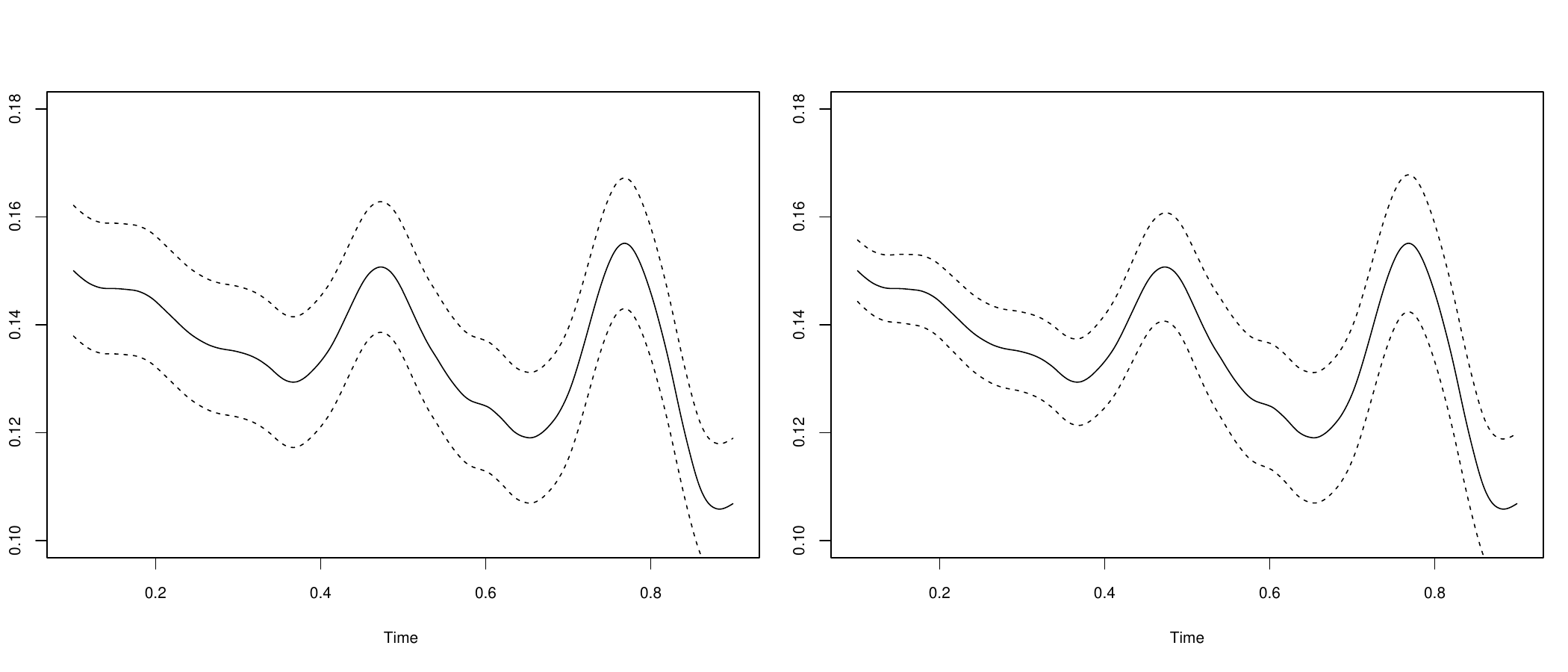}
\vspace{-.4cm}
\caption{\it  95\% simultaneous confidence bands  of the form \eqref{2.1} (fixed $t=0.5$)
for the data example in Section \ref{sec3}.
Left panel: constant width
(Algorithm \ref{algorithmt1}); Right panel: variable width (Algorithm  \ref{algorithmt2}).}
\label{Vol-fixt-SCB}
\end{figure}

\begin{figure}[htbp]
\centering
\includegraphics[width=14cm, height=6cm]{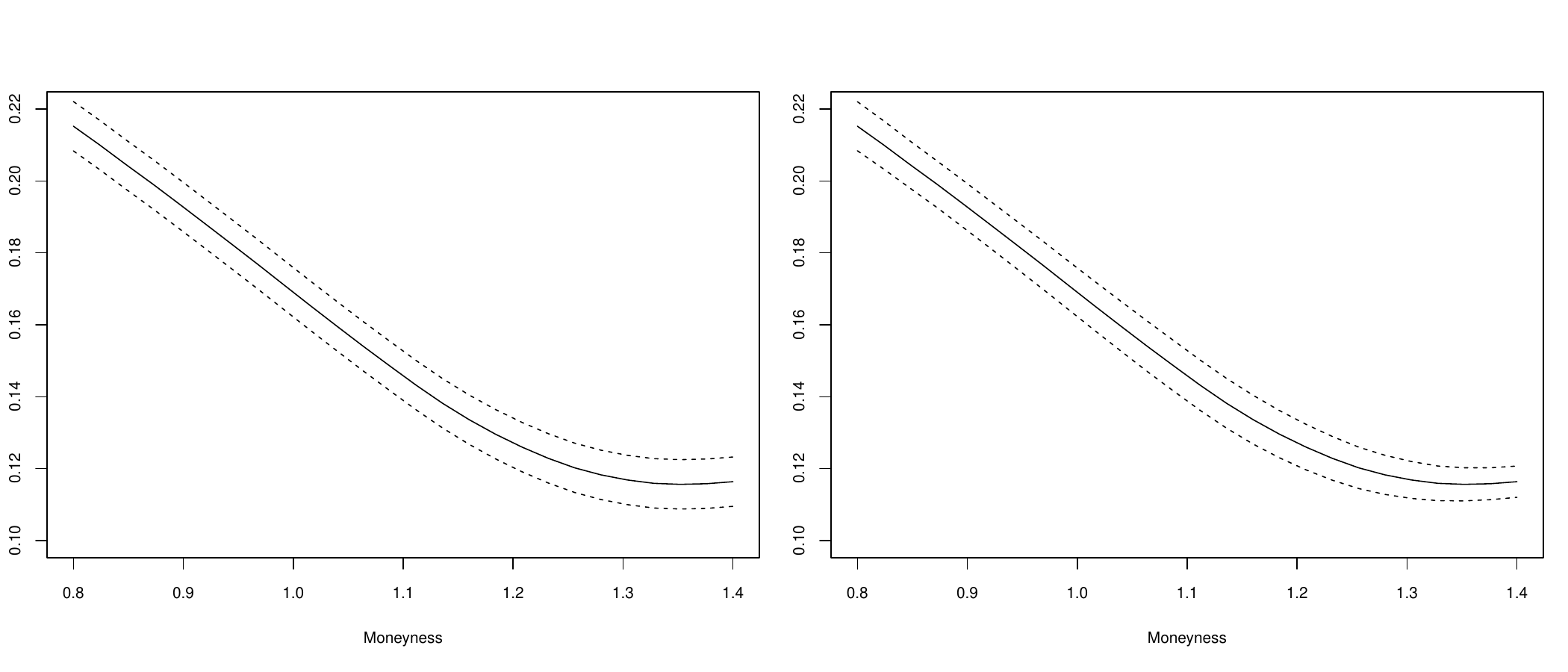}
\vspace{-.4cm}
\caption{\it  95\% simultaneous confidence bands  of the form   \eqref{2.2} (fixed $u=0.5$) for the  IV surface. Left panel: constant width
(Algorithm   \ref{algorithm0}); Right panel: variable width (Algorithm    \ref{algorithm0sigma}).}
\label{Vol-fixu-SCB}
\end{figure}


\setcounter{equation}{0}
\section{Examples of locally stationary error processes} \label{secex}

In  this section we present several examples for the error processes, which satisfy the assumptions  of the main article.

\begin{example}\label{example1.1}
{\rm 
Let  $(B_j)_{j \ge 0} $  denote a basis of $L^2\big([0,1]^2\big)$ and  let  $(\eta_{i,j})_{i \geq 0, j \geq 0}$ denote
an array of independent 
identically distributed centered random variables with variance $\sigma^2$. We
define the error process 
$$\epsilon_{i}(u,v)=\sum_{j=0}^{\infty} \eta_{i, j} B_{j}(u,v),$$
assume that
\begin{equation*} 
\sup_{u\in[0,1]}\int_0^1\E(\epsilon_i^2(u,v))dv = \sigma^2\sup_{u\in [0,1]}\sum_{s=0}^\infty \int B_s^2(u,v)dv<\infty  .
\end{equation*}
Next, consider 
the  locally stationary MA($\infty$) functional linear model
\begin{align} \label{2.12}
\varepsilon_{i,n} (t) 
=\sum_{j=0}^{\infty} \int_{0}^{1} a_{j}(t, v) \epsilon_{i-j}(\tfrac{i}{n},v) d v ~, 
\end{align}
where  	$( a_{j})_{j \geq 0}$
is a sequence of square integrable functions  $a_{j}:[0,1]^{2} \rightarrow \mathbb{R}$ satisfying 
$$
\sum_{j=0}^{\infty} \sup _{u, v \in[0,1]}\left|a_{j}(u, v)\right|<\infty.
$$ 
Define 
$\FF_i=(\ldots , \eta_{i-1},\eta_i)$, then we obtain from \eqref{2.12} the representation of the form
$\varepsilon_{i,n} (t)  =	G(\frac{i}{n},t,\FF_i) $, where 
$$G(u,t,\FF_i)  =\sum_{j=0}^{\infty} \int_{0}^{1} a_{j}(t, v) \sum_{s=0}^{\infty} \eta_{i-j, s} B_{s}(u,v) d v.
$$
Further, assume that $\|\eta_{1,1}\|_q<\infty$ for some $q>2$, then by Burkholder's and Cauchy's inequality the physical dependence measure defined in \eqref{2.7}  satisfies 
\begin{align*}
\delta_q(G,i) &= \sup_{u,t \in [0,1] }
\Big \|\sum_{s=0}^\infty\int_0^1a_i(t,v)B_s(u,v)dv(\eta_{0,s}-\eta'_{0,s})\Big \|_q \\
& =O \Big(  \sup_{u,t \in [0,1]}
\Big (\sum_{s=0}^\infty  \Big (\int_0^1a_i(t,v)B_s(u,v)dv  \Big )^2\Big ) ^{1/2}\Big )  \\
& =O \Big  (\sup_{t\in [0,1]}\Big  [\int_0^1a^2_i(t,v)dv \Big ]^{1/2}\Big )~.
\end{align*}
Therefore  Assumption \ref{asserr}(2) will be satisfied if 	
\begin{align*}
\sup_{t\in [0,1]}\Big[\int_0^1a^2_i(t,v)dv\Big]^{1/2}=O(\chi^i)~.
\end{align*}
Similarly, it  follows for $q\geq 2$ that 
\begin{align}\label{Quantifynorm}
\|G(u,t,\FF_0)\|_q^2 &\leq Mq\sum_{j=0}^\infty \sum_{s=0}^\infty (\int_0^1 a_j(t,v)B_s(u,v)dv)^2\|\eta_{1,1}\|_q^2\notag\\
& \leq Mq\sum_{j=0}^\infty\int_{0}^1a_j^2(t,v)dv\sum_{s=0}^\infty \int_0^1 B_s^2(u,v)dv \|\eta_{1,1}\|_q^2
\end{align}
for some sufficiently large constant $M$. Consequently,
the filter $G$ has finite   moment of order $q$, if  
\begin{align}\label{ajsquare}
\sum_{j=0}^\infty\int_{0}^1a_j^2(t,v)dv<\infty~.
\end{align}
Furthermore, if there exists  positive constants $M_0$ and $\alpha$ such that $\|\eta_{1,1}\|_q\leq M_0q^{1/2-\alpha}$, Assumption \ref{asserr}(1)  is also satisfied, because 
for any fixed $t_0$, the sequence \
\begin{align*}
\frac{t_0^q\|G(u,t,\FF_0)\|_q^q}{q!}=O\Big(\frac{C^qt_0^qq^{q-\alpha q}}{q!}\Big)=O\Big(\frac{1}{\sqrt{2\pi q}}(\frac{Ct_0e}{q^\alpha})^q\Big)
\end{align*}
is summable, where 
$$
C=\sup_{t\in[0,1],u\in[0,1]}M_0\sqrt{M\sum_{j=0}^\infty\int_0^1a_j^2(t,v)dv\sum_{s=0}^\infty\int_0^1B_s^2(u,v)dv}.
$$
Moreover,  if $b_s(u,v):=\frac{\partial}{\partial u}B_s(u,v)$ exists for $u\in (0,1), v\in [0,1]$, then  it follows observing \eqref{Quantifynorm}  that 
Assumption \ref{asserr}(3)  holds   under \eqref{ajsquare} and
$$\sup_{u\in [0,1]}\sum_{s=0}^\infty \int b_s^2(u,v)dv<\infty.
$$
Finally, if  $\|\eta_{1,1}\|_{q^*}<\infty$ and  
\begin{align*}
\sup_{t\in [0,1]}\Big[\int_0^1\Big(\frac{\partial}{\partial t}a_i(t,v)\Big)^2dv\Big]^{1/2}=O(\chi^i)~,
\end{align*}
it can be shown by similar arguments as given above that 
Assumption \ref{asssim} is satisfied.
}
\end{example}

\begin{example}\label{Basis-Function}
{\rm 
For  a given orthonormal basis  $(\phi_k(t))_{k \geq 1} $ of $L^2([0,1])$  consider the   functional time series  $(G(u,t,\FF_i))_{i \in \Z}$   defined by  
\begin{align}\label{KL}
G(u,t,\FF_i)=\sum_{k=1}^\infty H_k(u,\FF_{i})\phi_k(t)~,
\end{align} 
where   for each $k \in \N$ and $u \in [0,1]$  the random coefficients $(H_k(u,\FF_{i}))_{i \in \Z}$ are  
stationary time series.
A parsimonious choice of \eqref{KL} is to consider $\FF_i=\cup_{k=1}^\infty \FF_{i,k}$ where  $\{\FF_{i,k}\}_{k=1}^\infty$ are independent filtrations.
In this case we obtain 
\begin{align}\label{KL-Simple}
G(u,t,\FF_i)=\sum_{k=1}^\infty H_k(u,\FF_{i,k})\phi_k(t),
\end{align}
and  the random coefficients  $H_k(u,\FF_{i,k})$ are
stochastically independent. 
A  sufficient condition for Assumption \ref{asserr}(2) in  model \eqref{KL-Simple} is 
$$
\sup_{t \in [0,1]}  \sum_{k=0}^\infty|\phi_k(t)|\delta_{q}(H_k,i)=O(\chi^i)~,
$$ 
where $\delta_{q}(H_k,i):=\sup_{u\in [0,1]}\|H_k(u,\FF_{i,k})-H_k(u,\FF^*_{i,k})\|_q$. The $q${th} moment of the process  $G$ in \eqref{KL-Simple} exists for  $q\geq 2$, if 
$$\Delta_q:=\sup_{t\in [0,1],u\in[0,1]}\sum_{k=0}^\infty \phi_k^2(t)\|H_k(u,\FF_{0,k})\|^2_q<\infty.
$$
If further
$\Delta_q=O(q^{1/2-\alpha})$ for some $\alpha>0$, then similar arguments as given in  Example \ref{example1.1} show  that   Assumption \ref{asserr}(1) is satisfied as well. Finally,
if the inequality 
$$
\sum_{k=0}^\infty \phi_k^2(t)\Big \|\frac{\partial}{\partial u}H_k(u,\FF_{0,k}) \Big \|^2_q<\infty ~
$$
holds uniformly with respect to  $t, u\in [0,1]$, Assumption  \ref{asserr}(3)  is also satisfied. 
\\
On the other hand, in  model \eqref{KL} we have  $
H_k(u,\FF_i)=\int_0^1 G(u,t,\FF_i)\phi_k(t)dt$, and consequently 
the magnitude of $\|H_k\|_q$ and $\delta_q(H_k,i)$ can be determined by  Assumption \ref{asserr}. 
For example, if the basis of $L^2([0,1])$ is given by $\phi_k(t)=\cos(k\pi t)$ ${(k=0,1,\ldots })$ 
and the inequality 
$$
\left\|G(u,0,\FF_1)\right\|_{q}+\Big\|\frac{\partial}{\partial t}G(u,0,\FF_1)\Big\|_{q}+\sup _{t \in[0,1]}\Big\|\frac{\partial^2}{\partial t^2}G(u,t,\FF_1)\Big\|_{q}<\infty,
$$
holds for  $u\in[0,1]$, it follows by similar arguments as given in \cite{zhou2020statistical} that 
\begin{align}\label{PropertyH}
\sup_{u\in[0,1]}\big \|H_k(u,\FF_{k})\big \|_q=O(k^{-2}),~~ \delta_q(H_k,i)=O\big(\min(k^{-2},\delta_G(i,q))\big).
\end{align}
Similarly, assume that the basis  of $L^2([0,1])$ is given by the Legendre polynomials  and that 
\begin{align*}
\sup_{u\in [0,1]}\max_{s=1,2,3}\Big \|\int_{-1}^1 \frac{|\frac{\partial^s}{\partial t^s}G(u,t,\FF_0)|}{\sqrt{1-x^2}}dx\Big \|_q<\infty.
\end{align*}
If additionally for every  $\varepsilon > 0 $, there exists a constant  $\delta >0 $ 
such that 
$$
\sum_{s=1,2}\sum_k\Big\|\frac{\partial^s}{\partial t^s}G\left(u,x_{k},\FF_i\right)-\frac{\partial^s}{\partial t^s}G\left(u,x_{k-1},\FF_i)\right)\Big\|_q<\varepsilon
$$
for any finite sequence of pairwise disjoint sub-intervals  $(x_{k-1}, x_k) $ of the interval $(0,1)$
such that
$ \sum_{k}\left(x_{k}-x_{k-1}\right)<\delta $,  
it follows from  Theorem 2.1 of \cite{wang2012convergence} that  \eqref{PropertyH} holds as well.

Finally, if  
$$\sup_{t\in [0,1]} \sum_{k=0}^\infty|\phi'_k(t)|\delta_{q^*}(H_k,i)=O(\chi^i),$$ 
it can be shown by similar arguments as given above that 
Assumption \ref{asssim} is also  satisfied.
}		
\end{example}

\setcounter{equation}{0}
\section{Proofs of Theorems} \label{proof-main} 
In the proofs, for two real sequence $a_n$ and $b_n$ we write $a_n\lesssim b_n$, if there exists a universal positive constant $M$ such that $a_n\leq Mb_n$. Let $\mathbf 1(\cdot)$ be the usual indicator function. For simplicity let $\tilde K(u)=\frac{1}{nb_n}\sum_{i=1}^n K\big(\frac{i/n-u}{b_n}\big)$.
\subsection{Proof of Theorem \ref{SCB}}
\label{sec713}

For $ p \in \N$  define by $t_v =\frac{v}{p}$, ($v=0,\ldots ,p$) an equidistant partition of the interval $[0,1]$ and  let $M$ be a sufficiently large generic constant which may vary from line to line. 
Define
\begin{equation}
\label{3.3}
W_n(u,t)=\sqrt{nb_n} \big (\hat m(u,t)-\E(\hat m(u,t)) \big )
=\frac{1}{\sqrt {nb_n}\tilde K(u)}\sum_{i=1}^n G(\tfrac{i}{n},t,\FF_i)K\Big(\frac{\tfrac{i}{n}-u}{b_n}\Big),
\end{equation}
we have by triangle inequality
\begin{align}
\Big 
|\sup_{b_n\leq u\leq 1-b_n,0\leq t\leq 1}|W_n(u,t)|-
\max_{{\substack{\lceil nb_n\rceil \leq l_1\leq n-\lc nb_n\rc\\1\leq s\leq p}}} |W_n(\tfrac{l_1}{n},\tfrac{s}{p})| \Big  |\notag 
\leq \tilde W_n~,
\end{align}
where  
\begin{align*}
\tilde W_n 
\:= \max_{\substack{\lceil nb_n\rceil \leq l_1\leq n-\lc nb_n\rc,1\leq s\leq p,\\
|u-\tfrac{l_1}{n}|\leq 1/n, |t-\tfrac{s}{p}|\leq 1/p,u,t\in[0,1]}}|W_n(u,t)-W_n(\tfrac{l_1}{n},\tfrac{s}{p})|.
\end{align*}
By Assumption \ref{asssim}, Burkholder's inequality
and similar arguments as  given in the proof of Proposition 1.1 of \cite{dette2020prediction} we obtain
\begin{equation}\label{partialnorm}
\begin{split} 
& \sup_{u,t\in [0,1]}\Big\|\frac{\partial}{\partial u}W_n(u,t)\Big\|_{q^*}\leq {M\over  b_n} ,~\sup_{u,t\in [0,1]}\Big\|\frac{\partial}{\partial t}W_n(u,t)\Big\|_{q^*}\leq M, \\
& \sup_{u,t\in [0,1]}\Big\|\frac{\partial^2}{\partial u\partial t}W_n(u,t)\Big\|_{q^*}\leq {M\over  b_n}.
\end{split} 
\end{equation}
Note that  we have  for $\tau_s>0$, $s=1,2$ and $x,y\in [0,1)$,
\begin{align*} 
&\left\|\sup_{\substack{0\leq t_1\leq \tau_{1}\\0\leq t_2\leq \tau_{2}}} |W_n(t_1+x, t_2+y)-W_n(x,y)|\right\|_{q*} \leq \int_{0}^{\tau_{1}}\left\|\frac{\partial}{\partial u} W_{n}(x+u, y)\right\|_{q^*} d u\\&+\int_{0}^{\tau_{2}}\left\|\frac{\partial}{\partial t} W_{n}(x, y+v)\right\|_{q^*} d v\notag+ 
\int_{0}^{\tau_{1}} \int_{0}^{\tau_{2}}\left\|\frac{\partial^{2}}{\partial x \partial t} W_{n}(x+u, y+v)\right\|_{q^*} d u d v.
\end{align*}
Therefore,  \eqref{partialnorm}  and similar arguments as in  the proof of Proposition B.2  
of \cite{dette2019change} show 
\begin{align}\label{partialnorm2}
\|\tilde W_n\|_{q^*}
=O((np)^{1/q^*}((nb_n)^{-1}+1/p)).
\end{align}
Observing \eqref{eq25-2021-2-2} and \eqref{2023-10-3}.\ 
Lemma \ref{anti}  and \eqref{partialnorm2} it therefore follows
that  
\begin{align*}
\mathfrak{P}_n
& \lesssim (nb_n)^{-(1-11\iota)/8}+\Theta\Big(\sqrt{nb_n}\Big(b_n^4+\frac{1}{n}\Big),n p\Big)+\Theta(\delta, n p)+\p(\tilde W_n>\delta)\notag\\
& \lesssim (nb_n)^{-(1-11\iota)/8}+\Theta\Big(\sqrt{nb_n}\Big(b_n^4+\frac{1}{n}\Big),n p\Big)+\Theta(\delta, n p) \\
&  ~~~+\left((np)^{1/q^*}((nb_n)^{-1}+1/p)/\delta\right)^{q^*}.
\nonumber 
\end{align*}
Solving $\delta=\left((np)^{1/q^*}((nb_n)^{-1}+1/p)/\delta\right)^{q^*}$ we get $\delta=\big((np)^{1/q^*}((nb_n)^{-1}+1/p)\big)^{\frac{q^*}{q^*+1}}$  and the assertion of the theorem follows.  \hfill $\Box$
\subsection{Proof of Theorem \ref{Thm2.1-2021}}
\begin{proof}In the following discussion we use  the following notation. For any vector $y_n$ indexed by $n$, let $y_{n,r}$ be its $r_{th}$ component. For example,  $\hat S_{rm_n, j}$ is the $j_{th}$ entry of the vector $\hat S_{rm_n}$.

Let $T_k$  denote the  statistic  generated  by \eqref{3.13} 
in one bootstrap iteration of Algorithm \ref{algorithm1} and define for 
integers $a,b$ the quantities 
\begin{align*}
T_{ap+b}^{\diamond}& =\sum_{j=1}^{2\lceil nb_n \rceil-m_n'}\hat S_{jm_n',(a-1)p+b}{R_{k+j-1}},a=1,...n-2\lceil nb_n\rceil+1,1\leq b\leq p
\\
T^{\diamond} &  := ((T_1^{\diamond})^\top  , \ldots , (T_{(n-2\lceil nb_n\rceil+1 )p}^{\diamond})^\top)^\top 
= \big ({T_1}^\top ,\ldots , {T_{  n-2\lceil nb_n\rceil+1}}^\top \big )^\top  \\
T & =
|  T^{\diamond}|_{\infty} =
\max_{1\leq k\leq n-2\lceil nb_n\rceil+1} |T_k|_\infty
\end{align*}
It suffices to show that   the following inequality holds
\begin{align}
\sup_{x\in \mathbb R}\Big|\p(|T^\diamond/\sqrt{ 2\lceil nb_n\rceil-m'_n}|_\infty \leq x|\FF_n)-\p\Big(\frac{1}{\sqrt{2nb_n}}\Big|\sum_{i=1}^{2\lc nb_n\rc-1}\tilde Y_i\Big|_\infty\leq x\Big)\Big|\notag\\=O_p\Big(\vartheta^{1/3}_n\{1\vee \log (\frac{np}{\vartheta_n})\}^{2/3}+\Theta \Big(\Big(\sqrt{m_n\log np}\Big(\frac{1}{\sqrt{nb_n}}+b_n^3\Big) (np)^{\frac{1}{q}}\Big)^{q/(q+1)},np\Big)\Big).\label{3.41-2021}
\end{align}
If this estimate has been established,  Theorem \ref{Thm2.1-2021}  follows from  Theorem \ref{SCB}, which shows 
that the probabilities $\p \big ( 
\max_{b_n \leq u\leq 1-b_n,0\leq t\leq 1}\sqrt{nb_n}|\hat \Delta(u,t)|\leq x \big )$
can be approximated by the probabilities 
$$\p \big (\frac{1}{\sqrt{nb_n}}|\sum_{i=1}^{2\lc nb_n\rc-1}\tilde Y_i|_\infty\leq x \big )$$ 
uniformly with respect to $x \in \R$.\\


For a proof of \eqref{3.41-2021} we assume without loss of generality that $m_n$ is even so that $m_n'=m_n$.   For convenience, let $\sum_{i=a}^bZ_i=0$ if the indices $a$ and $b$ satisfy $a>b$. 
Given the data,  it follows for the conditional covariance 
\begin{align}\label{lag-Tdiamond}
&((2\lceil nb_n\rceil-1)-m_n+1)\sigma^{T^\diamond}_{(k_1-1)p+j_1,(k_2-1)p+j_2}:=\E(T^\diamond_{(k_1-1)p+j_1}T^\diamond_{(k_2-1)p+j_2}|\FF_n)\\&=\E\Big(\sum_{r=1}^{2\lceil nb_n \rceil-m_n}\hat S_{rm_n, (k_1-1)p+j_1}{R_{k_1+r-1}}\sum_{r=1}^{2\lceil nb_n \rceil-m_n}\hat S_{rm_n, (k_2-1)p+j_2}{R_{k_2+r-1}}\Big|\FF_n\Big)\notag\\
&=\sum_{r=1}^{2\lceil nb_n \rceil-m_n-(k_2-k_1)}\hat S_{(r+k_2-k_1)m_n, (k_1-1)p+j_1}\hat S_{rm_n, (k_2-1)p+j_2}.\notag
\end{align} 
where $1\leq k_1\leq k_2\leq  (n-2\lceil nb_n\rceil+1)$, $1\leq j_1,j_2\leq p$. Here,  without generality, we assume $k_1\leq k_2$. 
Define $\tilde T^\diamond$, and $\tilde S_{jm_n}$  in the same way as   $T^\diamond$,   and $\hat S_{jm_n}$ in  \eqref{3.13} and  \eqref{hol10}, respectively,
where the residuals $\hat {\tilde Z}_i$ defined in \eqref{3.11} and used in step (a) of Algorithm \ref{algorithm1}
have been replaced by  quantities $\tilde Z_i$ defined in \eqref{tildeZi}.Then we obtain by similar arguments
\begin{align}\label{eq64}
((2\lceil nb_n\rceil-1)-m_n+1)\sigma^{\tilde T^\diamond}_{(k_1-1)p+j_1,(k_2-1)p+j_2}:=\E(\tilde T^\diamond_{(k_1-1)p+j_1}\tilde T^\diamond_{(k_2-1)p+j_2}|\FF_n)\notag\\
=\sum_{r=1}^{\lceil 2nb_n \rceil-m_n-(k_2-k_1)}\tilde S_{(r+k_2-k_1)m_n, (k_1-1)p+j_1}\tilde S_{rm_n, (k_2-1)p+j_2}.
\end{align} 
Recall the definition of the random variable   $\tilde Y_j$ in Proposition \ref{Grid-Result} and denote by  $\tilde Z_{j,i}$, $\tilde Y_{j,i}$ the $i${th} component of the vectors $\tilde Z_j$ and $\tilde Y_j$, respectively ($1\leq i \leq  (n-2\lceil nb_n \rceil+1)p $, $1\leq j\leq 2\lceil nb_n\rceil-1$). Then we obtain
\begin{align}\label{eq65}
&\sigma^{\tilde Y}_{(k_1-1)p+j_1,(k_2-1)p+j_2}:=\E\Big(\frac{1}{2\lceil nb_n\rceil-1}\sum_{i_1=1}^{2\lceil nb_n\rceil-1}\tilde Y_{i_1,(k_1-1)p+j_1}\sum_{i_2=1}^{2\lceil nb_n\rceil-1}\tilde Y_{i_2,(k_2-1)p+j_2}\Big)\notag
\\&=\frac{\E(\sum_{i_1=1}^{2\lceil nb_n\rceil-1}\tilde Z_{i_1,(k_1-1)p+j_1}\sum_{i_2=1}^{2\lceil nb_n\rceil-1}\tilde Z_{i_2,(k_2-1)p+j_2})}{2\lceil nb_n\rceil-1}\notag\\
&=\frac{\E(\sum_{i_1=1}^{2\lceil nb_n\rceil-1} Z_{i_1+(k_1-1),\lceil nb_n\rceil+(k_1-1),j_1}\sum_{i_2=1}^{2\lceil nb_n\rceil-1}Z_{i_2+(k_2-1),\lceil nb_n\rceil+(k_2-1),j_2})}{2\lceil nb_n\rceil-1},
\end{align}
where  $Z_{i_1+(k_1-1),\lceil nb_n\rceil,j_1}$ is the $j_1${th} entry of the $p-$dimensional random vector  $Z_{i_1+(k_1-1),\lceil nb_n\rceil}$ and $Z_{i_2+(k_2-1),\lceil nb_n\rceil,j_2}$ is defined similarly.
We will show at the end of this section that 
\begin{align}
\Big \|\max_{k_1,k_2,j_1,j_2}|\sigma^{\tilde Y}_{(k_1-1)p+j_1,(k_2-1)p+j_2}-\sigma^{\tilde T^\diamond}_{(k_1-1)p+j_1,(k_2-1)p+j_2}\Big \|_{q/2}
=O(\vartheta_n).\label{eq67-2021}
\end{align}
If  \eqref{eq67-2021} holds, it  follows from Lemma \ref{Lemma1.3} 
that there exists a constant $\eta_0>0$ such that
\begin{align*}
\p\Big(\min_{\substack{1\leq k\leq  (n-2\lceil nb_n\rceil+1),\\ 1\leq j\leq p}}\sigma^{\tilde T^\diamond}_{(k-1)p+j,(k-1)p+j} \geq \eta_0\Big)\geq 1-O(\vartheta_n^{q/2}).
\end{align*}
Then, by  Theorem 2 of \cite{chernozhukov2015comparison}, we have
\begin{align}\label{Tdiamondbound1}
\sup_{x\in \mathbb R}\Big  |\p \Big  (\frac{|\tilde T^\diamond |_\infty}{\sqrt{ 2\lceil nb_n\rceil-m_n}}  \leq x ~\Big |\FF_n \Big )-\p \Big (\frac{1}{\sqrt{2nb_n}}\Big |\sum_{i=1}^{2\lc nb_n\rc-1}\tilde Y_i\Big |_\infty\leq x\Big )\Big |\notag\\=O_p(\vartheta^{1/3}_n\{1\vee \log (\frac{np}{\vartheta_n})\}^{2/3}).
\end{align}
Since conditional on  $ \FF_n$, $\big (\tilde T^\diamond -T^\diamond \big )$
is an  $(n-2\lceil nb_n\rceil+1)p$ dimensional Gaussian random vector
we obtain by the (conditional) Jensen  inequality and  conditional inequality for the  concentration of the maximum of a Gaussian process \citep[see Chapter 5 in Appendix A of][where a similar result has been derived in Lemma A.1]{chatterjee2014superconcentration}
that
\begin{align}\label{diffTdiamond}
\E(|\tilde T^\diamond-T^\diamond|^q_\infty|\FF_n)&
\leq M|\sqrt{\log np}\max_{r=1}^{(n-2\lceil nb_n\rceil+1)p } \Big ( \sum_{j=1}^{2\lceil nb_n \rceil-m_n'}(\hat S_{jm'_n,r}- S_{jm'_n,r})^2\Big)^{1/2}|^q
\end{align}
for some large constant $M$ almost surely.
Observing that
\begin{align}\label{maxtrick}\max_{1\leq i\leq n}|Z_i|^l\leq \sum_{1\leq i\leq n}|Z_i|^l
~~~~~~~
\text{for any $l>0, n\in \mathbb N$}
\end{align}
and using 
a similar argument as given in the proof of Proposition 1.1 in \cite{dette2020prediction}
and the fact that $K_l$ and $K_r$ are both three order kernels,  we have
\begin{align*}
\frac{1}{\sqrt{ 2\lceil nb_n\rceil-m_n}}
\Big\|\max_{r=1}^{(n-2\lceil nb_n\rceil+1)p }\Big (\sum_{j=1}^{\lceil 2nb_n \rceil-m_n'}(\hat S_{jm'_n,r}- S_{jm'_n,r})^2\Big )^{1/2} \Big\|_{q}=O\Big(\sqrt{m_n}\Big(\frac{1}{\sqrt{nb_n}}+b_n^3\Big)(np)^{\frac{1}{q}}\Big), 
\end{align*}
and combining this result  with the (conditional version) of Lemma \ref{anti} 
and   \eqref{diffTdiamond} yields 
\begin{align}
\sup_{x\in \mathbb R} \Big |\p \Big ( 
\frac{|T^\diamond|_\infty}{\sqrt{ 2\lceil nb_n\rceil-m_n}}
> x~\Big |\FF_n \Big )-\p \Big (\frac{1}{\sqrt{2nb_n}}\Big |\sum_{i=1}^{2\lc nb_n\rc-1}\tilde Y_i\Big  |_\infty>x\Big )\Big |
\notag\\ 
\leq 
\sup_{x\in \mathbb R}\Big |\p \Big ( \frac{|\tilde T^\diamond|_\infty}{\sqrt{ 2\lceil nb_n\rceil-m_n}} > x~ \Big |\FF_n\Big )-\p \Big (\frac{1}{\sqrt{2nb_n}}\Big |\sum_{i=1}^{2\lc nb_n\rc-1}\tilde Y_i\Big |_\infty>x\Big )\Big | \notag\\
+\p \Big ( \frac{ |\tilde T^\diamond-T^\diamond|_\infty}{\sqrt{ 2\lceil nb_n\rceil-m_n}} >\delta\Big|\FF_n \Big  )+O\big(\Theta(\delta, np)\big)\notag\\
\leq
\sup_{x\in \mathbb R}|\p \Big ( \frac{ |\tilde T^\diamond|_\infty}{\sqrt{2\lceil nb_n\rceil-m_n}}  > x~\Big |\FF_n \Big )-\p \Big (\frac{1}{\sqrt{2nb_n}} \Big  |\sum_{i=1}^{2\lc nb_n\rc-1}\tilde Y_i\Big |_\infty>x\Big  )\Big  | \notag\\
+O_p\big(\delta^{-q}\big (\sqrt{m_n\log np}\big(\frac{1}{\sqrt{nb_n}}+b_n^3\big)(np)^{\frac{1}{q}}\big)^{q}\big)+O\big(\Theta(\delta, np)\big ),
\label{det22}
\end{align}
where we have used the Markov's inequality. Taking $\delta=\big(\sqrt{m_n\log np}(\frac{1}{\sqrt{nb_n}}+b_n^3)(np)^{\frac{1}{q}}\big)^{q/(q+1)}$ in \eqref{det22}, and combining this estimate with \eqref{Tdiamondbound1} yields \eqref{3.41-2021} completes the proof.

\par
{\bf Proof of  \eqref{eq67-2021}.}
To simplify the notation, write 
\begin{align*} 
G_{j,i,k}=G(\tfrac{i+k-1}{n},j/p,\FF_{i+k-1}),~~G_{j,i,k,u}=G(\tfrac{i+k-1+u}{n},j/p,\FF_{u})
\end{align*} 
Without loss of generality, we consider the case $k_1\leq k_2$. We  calculate $\sigma^{\tilde Y}_{(k_1-1)p+j_1,(k_2-1)p+j_2}$
observing the  representation 
\begin{align*} 
Z_{i_1+(k_1-1),\lceil nb_n\rceil+(k_1-1),j_1}=G_{j_1,i_1,k_1}K\Big (\tfrac{i_1-\lceil nb_n\rceil}{nb_n}\Big ).
\end{align*}
By  Lemma \ref{Lemma1.2}  it follows  that 
\begin{align}\label{Cov-Z}\E  \big [ Z_{i_1+(k_1-1),\lceil nb_n\rceil+(k_1-1),j_1}Z_{i_2+(k_2-1),\lceil nb_n\rceil+(k_2-1),j_2} \big ] =O(\chi^{|i_1-i_2+k_1-k_2|}).
\end{align}
uniformly for $1\leq i_1, i_2\leq 2\lceil nb_n \rceil-1$, $1\leq j_1,j_2\leq p$, $1\leq k_1, k_2\leq n-2\lceil nb_n \rceil+1$. We first show that \eqref{eq67-2021} holds whenever $k_2-k_1>2\lceil nb_n \rceil-m_n$.  On the one hand, observing and \eqref{lag-Tdiamond} and \eqref{eq64} that if $2\lceil nb_n \rceil-m_n-(k_2-k_1)<0 $ then \begin{align}\label{B-19-new}\sigma^{\tilde T^\diamond}_{(k_1-1)p+j_1,(k_2-1)p+j_2}=0~~ a.s.\end{align} 
Moreover, by \eqref{eq65} and \eqref{Cov-Z}, straightforward calculations show that
\begin{align}\label{B-20-new}
\sigma^{\tilde Y}_{(k_1-1)p+j_1,(k_2-1)p+j_2}=\frac{1}{2\lceil nb_n\rceil-1}O \Big (\sum_{i_1=1}^{2\lceil nb_n\rceil-1}\sum_{i_2=1}^{2\lceil nb_n\rceil-1}\chi^{|i_1-i_2+k_1-k_2|}\Big  )=O\Big (\frac{m_n}{nb_n} \Big ).
\end{align}
Combining  \eqref{B-19-new}, \eqref{B-20-new} and by applying similar argument to $k_1\geq k_2$, we obtain
\begin{align}\label{B-21-new}
\Big \|\max_{\substack{k_1,k_2,j_1,j_2\\|k_2-k_1|>2\lceil nb_n\rceil-m_n}}|\sigma^{\tilde Y}_{(k_1-1)p+j_1,(k_2-1)p+j_2}-\sigma^{\tilde T^\diamond}_{(k_1-1)p+j_1,(k_2-1)p+j_2}\Big \|_{q/2}
=O\Big(\frac{m_n}{nb_n}\Big).
\end{align}

Now consider the case that $k_2-k_1\leq 2\lceil nb_n \rceil-m_n$. Without losing generality we consider $k_1\leq k_2$. Again by \eqref{eq65}
\begin{align*}
&\E \Big (\sum_{i_1=1}^{k_2-k_1}Z_{i_1+(k_1-1),\lceil nb_n\rceil+(k_1-1),j_1}\sum_{i_2=1}^{2\lceil nb_n\rceil-1}Z_{i_2+(k_2-1),\lceil nb_n\rceil+(k_2-1),j_2} \Big )
\notag\\&=O \Big ( \sum_{i_1=1}^{k_2-k_1}\sum_{i_2=1}^{2\lceil nb_n\rceil-1}\chi^{|i_2-i_1+k_2-k_1|} \Big )=O \Big ( \sum_{i_1=1}^{k_2-k_1}\sum_{i_2=1}^{2\lceil nb_n\rceil-1}\chi^{i_2-i_1+k_2-k_1} \Big )=O(1),\\
&\E \Big (\sum_{i_1=1}^{2\lceil nb_n\rceil-1}Z_{i_1+(k_1-1),\lceil nb_n\rceil+(k_1-1),j_1}\sum_{i_2=2\lceil nb_n\rceil-(k_2-k_1)}^{2\lceil nb_n\rceil-1}Z_{i_2+(k_2-1),\lceil nb_n\rceil+(k_2-1),j_2} \Big ) 
\notag\\
&=O\Big ( \sum_{i_1=1}^{2\lceil nb_n\rceil-1}\sum_{i_2=2\lceil nb_n\rceil-(k_2-k_1)}^{2\lceil nb_n\rceil-1}\chi^{|i_2-i_1+k_2-k_1|}\Big )=O\Big ( \sum_{i_1=1}^{2\lceil nb_n\rceil-1}\sum_{i_2=2\lceil nb_n\rceil-(k_2-k_1)}^{2\lceil nb_n\rceil-1}\chi^{i_2-i_1+k_2-k_1}\Big )=O(1).
\end{align*}
Let $a=\lf M\log n\rf$ for a sufficiently large constant $M$.
Using \eqref{eq65}, it follows   (considering the lags up to $a$)
that   
\begin{align}
&\sigma^{\tilde Y}_{(k_1-1)p+j_1,(k_2-1)p+j_2}\notag\\&=\tfrac{1}{2\lceil nb_n\rceil -1}\E \Big (\sum_{i_1=k_2-k_1+1}^{2\lceil nb_n\rceil-1} Z_{i_1+(k_1-1),\lceil nb_n\rceil+(k_1-1),j_1}\sum_{i_2=1}^{2\lceil nb_n\rceil-(k_2-k_1)-1}Z_{i_2+(k_2-1),\lceil nb_n\rceil+(k_2-1),j_2} \Big )\notag\\
& +O((nb_n)^{-1})\notag\\
&  =\tfrac{1}{2\lceil nb_n\rceil -1}\E\Big(\sum_{i_1,i_2=1}^{2\lceil nb_n\rceil-(k_2-k_1)-1} G_{j_1,i_1,k_2}K(\tfrac{i_1+k_2-k_1-\lceil nb_n\rceil}{nb_n})
G_{j_2,i_2,k_2}K(\tfrac{i_2-\lceil nb_n\rceil}{nb_n})\Big)+O((nb_n)^{-1})\notag
\\
&=A + B +O(nb_n\chi^a+(nb_n)^{-1}),\label{eq71-2021}
\end{align}
where the terms $A$ and $ B$  are defined by
\begin{align}
\label{eq72-2021} A 
&:=\tfrac{1}{(2\lceil nb_n\rceil -1)}\sum_{i=1}^{2\lceil nb_n \rceil-(k_2-k_1)-1}A_i,
\\
\nonumber 
A_i &   = \E(G_{j_1,i,k_2,0}G_{j_2,i,k_2,0})
K(\tfrac{i+k_2-k_1-\lceil nb_n\rceil}{nb_n})K(\tfrac{i-\lceil nb_n\rceil}{nb_n})\\
B&=\tfrac{1}{(2\lceil nb_n\rceil -1)}\sum_{u=1}^a(B_{1,u}+B_{2,u}), 
\nonumber 
\\
B_{1,u}&=\sum_{i=1}^{2\lceil nb_n \rceil-(k_2-k_1)-1-u}B_{1,u,i},\label{eq74-2021} \\
B_{2,u}&=:\sum_{i=1}^{2\lceil nb_n \rceil-(k_2-k_1)-1-u}B_{2,u,i}.\label{eq75-2021}
\end{align}
and 
\begin{align*}
B_{1,u,i}&=
\E(G_{j_1,i,k_2,u}G_{j_1,i,k_2,0}) K(\tfrac{i+u+k_2-k_1-\lceil nb_n\rceil}{nb_n})K(\tfrac{i-\lceil nb_n\rceil}{nb_n}))
\\
B_{2,u,i}&=\E(G_{j_1,i,k_2,0}G_{j_2,i,k_2,u})
K(\tfrac{i+k_2-k_1-\lceil nb_n\rceil}{nb_n})K(\tfrac{i+u-\lceil nb_n\rceil}{nb_n})
\end{align*}

Therefore, by \eqref{eq71-2021}, we have that
\begin{align}\label{eq75}
\sigma^{\tilde Y}_{(k_1-1)p+j_1,(k_2-1)p+j_2}=\frac{1}{2\lceil nb_n\rceil -1}\Big(\sum_{i=1}^{2\lceil nb_n\rceil-1-(k_2-k_1)}A_i+\sum_{u=1}^a\sum_{i=1}^{2\lceil nb_n\rceil-1-(k_2-k_1)-u}(B_{1,u,i}+B_{2,u,i})\Big)\notag\\+O(nb_n\chi^a+(nb_n)^{-1}).
\end{align}

Now for  the term in \eqref{eq64} we have 
\begin{align*}
&m_n\tilde S_{(r+k_2-k_1)m_n, (k_1-1)p+j_1}\tilde S_{rm_n, (k_2-1)p+j_2}=\Big(\sum_{i=r+k_2-k_1}^{r+k_2-k_1+m_n/2-1}-\sum_{i=r+k_2-k_1+m_n/2}^{r+k_2-k_1+m_n}\Big)Z_{i+k_1-1, \lceil nb_n\rceil+k_1-1,j_1}\notag\\&\times\Big(\sum_{i=r}^{r+m_n/2-1}-\sum_{i=r+m_n/2}^{r+m_n}\Big)Z_{i+k_2-1, \lceil nb_n\rceil+k_2-1,j_2}\notag\\
=&\Big(\sum_{i=r}^{r+m_n/2-1}-\sum_{i=r+m_n/2}^{r+m_n}\Big)G_{j_1,i,k_2}K(\tfrac{i+k_2-k_1-\lceil nb_n\rceil}{nb_n})\times\Big(\sum_{i=r}^{r+m_n/2-1}-\sum_{i=r+m_n/2}^{r+m_n}\Big)G_{j_2,i,k_2}K(\tfrac{i-\lceil nb_n\rceil}{nb_n}).
\end{align*}

By Lemma \ref{Lemma1.2}, it follows that uniformly for $|k_2-k_1|\leq 2\lceil nb_n\rceil-m_n $ and $1\leq r\leq \lceil 2nb_n \rceil-m_n-(k_2-k_1)$, \begin{align}
&m_n\E\tilde S_{(r+k_2-k_1)m_n, (k_1-1)p+j_1}\tilde S_{rm_n, (k_2-1)p+j_2}\notag\\&=
\sum_{i=r}^{r+m_n}\E(G_{j_1,i,k_2}G_{j_2,i,k_2})K(\tfrac{i+k_2-k_1-\lceil nb_n\rceil}{nb_n})K(\tfrac{i-\lceil nb_n\rceil}{nb_n})\notag
\\&+\sum_{u=1}^a\Big(\sum_{i=r}^{r+m_n-u}\big(\E(G_{j_1,i,(k_2+u)}G_{j_2,i,k_2})K(\tfrac{i+u+k_2-k_1-\lceil nb_n\rceil}{nb_n})K(\tfrac{i-\lceil nb_n\rceil}{nb_n})\notag\\
&+\E(G_{j_2,i,(k_2+u)}G_{j_1,i,k_2})K(\tfrac{i+k_2-k_1-\lceil nb_n\rceil}{nb_n})K(\tfrac{i+u-\lceil nb_n\rceil}{nb_n})\big)
\Big)+O(m_n\chi^a+a^2),\label{eq71}
\end{align}
where the the term $m_n\chi^a$ corresponds to the error of omitting terms in the sum 
with a large index $a$,  and the term $a^2$ summarizes the error
due to ignoring  different signs in the product $\tilde S_{(r+k_2-k_1)m_n, (k_1-1)p+j_1}\tilde S_{rm_n, (k_2-1)p+j_2}$ (for each index $u$, we omit   $2u$).
Furthermore, by  Assumption \ref{asskern} and   \ref{asserr}(3)  it follows that uniformaly for $|u|\leq a$
\begin{align}
& \frac{1}{m_n}\sum_{i=r}^{r+m_n}\E(G_{j_1,i,k_2}G_{j_2,i,k_2})K(\tfrac{i+k_2-k_1-\lceil nb_n\rceil}{nb_n})K(\tfrac{i-\lceil nb_n\rceil}{nb_n})=A_r+O(\frac{m_n}{nb_n}),\label{eq81-2021}
\\
& \frac{1}{m_n}\sum_{i=r}^{r+m_n-u}\E(G_{j_1,i,(k_2+u)}G_{j_2,i,k_2})K(\tfrac{i+u+k_2-k_1-\lceil nb_n\rceil}{nb_n})K(\tfrac{i-\lceil nb_n\rceil}{nb_n})=B_{1,u,r}+O(\frac{m_n}{nb_n}+\frac{a}{m_n}),\label{eq82-2021}\\
& \frac{1}{m_n}\sum_{i=r}^{r+m_n-u}\E(G_{j_2,i,(k_2+u)}G_{j_1,i,k_2})K(\tfrac{i+k_2-k_1-\lceil nb_n\rceil}{nb_n})K(\tfrac{i+u-\lceil nb_n\rceil}{nb_n})=B_{2,u,r}+O(\frac{m_n}{nb_n}+\frac{a}{m_n})\label{eq83-2021}
\big),
\end{align}
where terms $A_r, B_{1,u,r}$ and $B_{2,u,r}$ are defined in equations \eqref{eq72-2021}, \eqref{eq74-2021} and \eqref{eq75-2021}, respectively. Notice that \eqref{eq64} and expressions 
\eqref{eq71}, \eqref{eq81-2021}, \eqref{eq82-2021} and \eqref{eq83-2021} yield that
\begin{align}\label{Esigmadia}
\E\sigma^{\tilde T^\diamond}_{(k_1-1)p+j_1,(k_2-1)p+j_2}=\frac{1}{2\lceil nb_n\rceil -m_n}\Big \{ \sum_{r=1}^{2\lceil nb_n\rceil-m_n-(k_2-k_1)}(A_r+O(\frac{m_n}{nb_n}))\notag\\+\sum_{u=1}^a\sum_{r=1}^{2\lceil nb_n\rceil-m_n-(k_2-k_1)}(B_{1,u,r}+B_{2,u,r}+O(\frac{m_n}{nb_n}+\frac{a}{m_n}))\Big \} 
+O \big (\chi^a+\frac{a^2}{m_n} \big ).
\end{align}
Lemma \ref{Lemma1.2} implies
$$
\max_{\substack{1\leq r\leq 2\lceil nb_n\rceil-(k_2-k_1)-1,\\1\leq k_1\leq k_2\leq  (n-2\lceil nb_n\rceil+1),s=1,2}}B_{s,u,r}=O(\chi^u), 
$$
which yields in combination with  equations
\eqref{eq75}, \eqref{Esigmadia} with $a=M\log n $ for a sufficiently large constant $M$,  and a similar argument applied to the case that $k_1\geq k_2$,
\begin{align}\label{eq87-2021}
\max_{\substack{1\leq k_1,k_2\leq  (n-2\lceil nb_n\rceil+1) \\|k_2-k_1|\leq 2\lceil nb_n \rceil-m_n, 1\leq j_1,j_2\leq p}}\Big|\E\sigma^{\tilde T^\diamond}_{(k_1-1)p+j_1,(k_2-1)p+j_2}-\sigma^{\tilde Y}_{(k_1-1)p+j_1,(k_2-1)p+j_2}\Big|=O \Big (
\frac{\log^2 n}{m_n}+\frac{m_n\log n}{nb_n} \Big ).
\end{align}
Furthermore,  
using \eqref{maxtrick},
the  Cauchy-Schwartz inequality,  a similar argument as given in  the proof of Lemma 1 of
\cite{zhou2013heteroscedasticity} and Assumption \ref{asserr}(2) yield  that
\begin{align}\label{eq88-2021}
\Big \|\max_{\substack{1\leq k_1\leq k_2\leq  (n-2\lceil nb_n\rceil+1),\\ 1\leq j_1,j_2\leq p}}|\E\sigma^{\tilde T^\diamond}_{(k_1-1)p+j_1,(k_2-1)p+j_2}-\sigma^{\tilde T^\diamond}_{(k_1-1)p+j_1,(k_2-1)p+j_2}| \Big \|_{q/2}=O\Big(\sqrt{\frac{m_n}{nb_n}}(np)^{4/q}\Big).
\end{align}
Combining  \eqref{eq87-2021} and \eqref{eq88-2021}, we obtain
\begin{align}\label{B-32-new}
\Big \|\max_{\substack{k_1,k_2,j_1,j_2\\|k_2-k_1|\leq 2\lceil nb_n\rceil-m_n}}|\sigma^{\tilde Y}_{(k_1-1)p+j_1,(k_2-1)p+j_2}-\sigma^{\tilde T^\diamond}_{(k_1-1)p+j_1,(k_2-1)p+j_2}\Big \|_{q/2}
\notag\\=O\Big(\frac{\log^2n}{m_n}+\frac{m_n\log n}{nb_n}+\sqrt{\frac{m_n}{nb_n}}(np)^{4/q}\Big).
\end{align}

Therefore the estimate  \eqref{eq67-2021} follows combining  \eqref{B-21-new} and \eqref{B-32-new}.  
\end{proof}

\subsection{Proof of Theorem \ref{sigma-asy}}
\label{631}
Similarly to  
\eqref{mean-inference} and \eqref{mean-inference2} in the proof of Proposition \ref{Grid-Result}  we obtain
\begin{align} \label{6-17-115}
\sup_{\substack{u\in [b_n,1-b_n]\\t\in[0,1]}} \frac{1}{\sigma(u,t)}\Big | \E(\hat m(u,t))-m(u,t)
\Big |\leq M \Big (\frac{1}{n}+b_n^4 \Big )
\end{align}
for some constant $M$, where we have used the fact that, by Assumption \ref{asskern}, $\int K(v)v^2dv =0 $. 
Moreover,
by a  similar but simpler argument as given in the proof of equation (B.7) in Lemma B.3 of \cite{dette2019change}
we have   for  the quantity
\begin{align*}
\frac{(\hat m(u,t)-\E(\hat m(u,t)))}{\sigma(u,t)}= \frac{1}{nb_n\tilde K(u)}\sum_{i=1}^n \frac{G(\tfrac{i}{n},t,\FF_i)}{\sigma(u,t)}K \Big (\frac{\tfrac{i}{n}-u}{b_n}\Big )  := \Psi^\sigma(u,t) 
\end{align*} 
the estimate 
\begin{align}\label{6-17-118}
\Big \|\sup_{u\in[b_n,1-b_n],t\in[0,1]}\sqrt{nb_n}|\Phi^\sigma(u,t)-\Psi^\sigma(u,t)| \Big \|_q=O(b_n^{1-2/q}),
\end{align}
where
\begin{align*}
\Phi^\sigma(u,t)=\frac{1}{nb_n\tilde K(u)}\sum_{i=1}^n \frac{G(\tfrac{i}{n},t,\FF_i)}{\sigma(\tfrac{i}{n},t)}K(\frac{\tfrac{i}{n}-u}{b_n}).
\end{align*}
Following the proof of  Theorem \ref{SCB} we find that  
\begin{align*}
\sup_{x\in \mathbb R}
\Big |\p
\Big (\max_{b_n \leq u\leq 1-b_n,0\leq t\leq 1}\sqrt{nb_n}\big|{\Phi^\sigma(u,t)}\big|\leq x \Big )-\p\Big (\Big |\frac{1}{\sqrt{nb_n}}\sum_{i=1}^{2\lc nb_n\rc-1}\tilde Y^\sigma_i\Big |_\infty\leq x\Big )\Big  |\notag\\=O\Big((nb_n)^{-(1-11\iota)/8}+\Theta\Big(\big((np)^{1/q^*}((nb_n)^{-1}+1/p)\big)^{\frac{q^*}{q^*+1}}, n p\Big)\Big).
\end{align*}
Combining this result  with    Lemma \ref{anti} 
(with
$X=\max_{\substack{b_n \leq u\leq 1-b_n\\0\leq t\leq 1}}\sqrt{nb_n}\big|{\Phi^\sigma(u,t)}\big|$,
\\
$Y=\frac{1}{\sqrt{nb_n}}\sum_{i=1}^{2\lc nb_n\rc-1}\tilde Y^\sigma_i$, $X'=\max_{b_n \leq u\leq 1-b_n,0\leq t\leq 1}\sqrt{nb_n}\big|{\Psi^\sigma(u,t)}\big|$ 
) and \eqref{6-17-118}
gives
\begin{align}\label{6-17-120}
\sup_{x\in \mathbb R}\Big|\p\Big(\max_{b_n \leq u\leq 1-b_n,0\leq t\leq 1}\sqrt{nb_n}\big|{\Psi^\sigma(u,t)}\big|\leq x\Big)-\p\Big(\Big|\frac{1}{\sqrt{nb_n}}\sum_{i=1}^{2\lc nb_n\rc-1}\tilde Y^\sigma_i\Big|_\infty\leq x\Big)\Big|\notag\\
=O\Big((nb_n)^{-(1-11\iota)/8}+\Theta(\big((np)^{1/q^*}((nb_n)^{-1}+1/p)\big)^{\frac{q^*}{q^*+1}}, n p)\notag\\
+\p\Big(\sup_{u\in[b_n,1-b_n],t\in[0,1]}\sqrt{nb_n}|\Phi^\sigma(u,t)-\Psi^\sigma(u,t)|>\delta\Big) +\Theta(\delta, np)\Big) \notag\\
=O\Big((nb_n)^{-(1-11\iota)/8}+\Theta(\big((np)^{1/q^*}((nb_n)^{-1}+1/p)\big)^{\frac{q^*}{q^*+1}}, n p)+\Theta(\delta, np)+\frac{b_n^{q-2}}{\delta^q}\Big).
\end{align}
Taking $\delta=b_n^{\frac{q-2}{q+1}}$ we obtain for the last two terms 
in \eqref{6-17-120}
\begin{align*}
\Theta(\delta, np)+\frac{b_n^{q-2}}{\delta^q}=O\Big(\Theta(b_n^{\frac{q-2}{q+1}},np)\Big).
\end{align*}
On the other hand, \eqref{6-17-115}, \eqref{6-17-120} and  Lemma \ref{anti}  (with
$X=\max_{b_n \leq u\leq 1-b_n,0\leq t\leq 1}\sqrt{nb_n}\big|{\Psi^\sigma(u,t)}\big|$, $Y=\frac{1}{\sqrt{nb_n}}\sum_{i=1}^{2\lc nb_n\rc-1}\tilde Y^\sigma_i$, $X'=\max_{b_n \leq u\leq 1-b_n,0\leq t\leq 1}\sqrt{nb_n}\big|\hat \Delta^\sigma(u,t)\big|$ and   $\delta=M\sqrt{nb_n} (\frac{1}{n}+b_n^4)$
with a sufficiently large constant $M$)  
yield
\begin{align*}
\sup_{x\in \mathbb R}
\Big |\p \Big 
(\max_{b_n \leq u\leq 1-b_n,0\leq t\leq 1}\sqrt{nb_n}\big|\hat \Delta^\sigma(u,t)\big|\leq x \Big )-\p \Big ( \Big |\frac{1}{\sqrt{nb_n}}\sum_{i=1}^{2\lc nb_n\rc-1}\tilde Y^\sigma_i\Big |_\infty\leq x)\Big |\notag\\
=O \Big  ((nb_n)^{-(1-11\iota)/8}+\Theta
\big  (\big((np)^{1/q^*}((nb_n)^{-1}+1/p)\big)^{\frac{q^*}{q^*+1}}, n p \big ) \\
+\Theta \Big 
(\sqrt{nb_n}(b_n^4+\frac{1}{n}),np \Big )+\Theta(b_n^{\frac{q-2}{q+1}},np)\Big).
\end{align*}
\hfill $\Box$

\subsection{Proof of Theorem \ref{sigma-bootstrap}}

\label{633}

\begin{proof}
Recall that $g_n=\frac{w^{5/2}}{n}\tau_n^{-1/q'}+w^{1/2}n^{-1/2}\tau_n^{-1/2-2/q'}+w^{-1}$ and let 
$\eta_n$ be a sequence of positive numbers such that  $\eta_n\rightarrow \infty$ and  $(g_n+\tau_n)\eta_n\rightarrow 0$ (note that $g_n+\tau_n$ is the convergence rate of the estimator $\hat \sigma^2$ in Proposition \ref{prop-lrv}). 
Define the $\FF_n$ measurable event  
$$
A_n= \Big \{\sup_{u\in [0,1],t\in[0,1]}|\hat \sigma^2(u,t)-\sigma^2(u,t)|>(g_n+\tau_n)\eta_n \Big  \}~,
$$ 
then  Proposition \ref{prop-lrv}  
and  Markov's inequality
yield
\begin{align}\label{6-18-124}
\p(A_n)=O \big  (  {\eta_n^{- q'}} \big ).
\end{align}	
Then by Theorem \ref{sigma-asy}, Proposition \ref{prop-lrv} and Lemma \ref{anti} 
we have \begin{align}\label{Deltahatsigma}
\mathfrak{P}^{\hat{\sigma}} =
\sup_{x\in \mathbb R} \Big |\p \Big ( 
\max_{b_n \leq u\leq 1-b_n,0\leq t\leq 1}\sqrt{nb_n}|\hat \Delta^{\hat \sigma(u,t)}|\leq x \Big )-\p 
\Big (\Big |\frac{1}{\sqrt{nb_n}}\sum_{i=1}^{2\lc nb_n\rc-1}\tilde Y^\sigma_i \Big  |_\infty\leq x\Big )\Big |=o_p(1).
\end{align}
Let $T^{\hat \sigma}_k$  denote the  statistic $T^{\hat \sigma, (r) }_k$ in step (d) of Algorithm \ref{algorithm3}  generated  by  one bootstrap iteration  and define for 
integers $a,b$ the quantities  \begin{align*}
T_{ap+b}^{\hat \sigma,\diamond}&=\sum_{j=1}^{2\lceil nb_n \rceil-m_n'}\hat S^{\hat \sigma}_{jm_n',(a-1)p+b} {R_{k+j-1}},a=1,...n-2\lceil nb_n\rceil+1,1\leq b\leq p\\
T^{\hat \sigma,\diamond} &  := ((T_1^{\hat \sigma,\diamond})^\top  , \ldots , (T_{(n-2\lceil nb_n\rceil+1 )p}^{\hat \sigma,\diamond})^\top)^\top
= \big ({T_1^{\hat \sigma}}^\top ,\ldots , {T_{  n-2\lceil nb_n\rceil+1}^{\hat \sigma}}^\top \big )^\top \end{align*}and therefore
\begin{align*} 
T^{\hat \sigma} & =
|  T^{\hat \sigma,\diamond}|_{\infty} =
\max_{1\leq k\leq n-2\lceil nb_n\rceil+1} |T_k^{\hat \sigma}|_\infty
\end{align*}
We recall the notation \eqref{tildeZisigma},    introduce
the $(n-2\lceil nb_n\rceil +1)p$-dimensional  random vectors $ \hat S^{ \sigma,*}_{jm_n}  =\sum_{r=j}^{j+m_n-1 } {\tilde  Z}^{ \sigma}_{r},$ and 
\begin{align*}
\hat S^{ \sigma}_{jm_n'}& =\frac{1}{\sqrt{m'_n}}\hat  S^{ \sigma,*}_{j,\lfloor m_n/2\rfloor} - \frac{1}{\sqrt{m'_n}}\hat S^{ \sigma,*}_{j+\lfloor m_n/2\rfloor+1,\lfloor m_n/2\rfloor}~,
\end{align*} 
and consider 
\begin{align} \nonumber 
T_k^{ \sigma}& =\sum_{j=1}^{2\lceil nb_n \rceil-m_n'}\hat S^{ \sigma}_{jm'_n,[(k-1)p+1:kp]} {R_{k+j-1}} ~,~~ k= 1 , \ldots  , n-2\lceil nb_n\rceil+1  , \\
T^{ \sigma,\diamond} &  := ((T_1^{ \sigma,\diamond})^\top  , \ldots , (T_{(n-2\lceil nb_n\rceil+1 )p}^{ \sigma,\diamond})^\top)^\top 
= \big ({T_1^{ \sigma}}^\top ,\ldots , {T_{  n-2\lceil nb_n\rceil+1}^{ \sigma}}^\top \big )^\top ~\nonumber ,
\end{align} 
where  $ T^{ \sigma,\diamond}$ is obtained from $  T^{\hat \sigma,\diamond} $ by replacing
$\hat \sigma $ by $\sigma$.
Similar arguments as given in the 
proof of Theorem \ref{Thm2.1-2021} show, that it 
is sufficient to show the estimate
\begin{align}\label{4.23-7-31}
&	\sup_{x\in \mathbb R}\Big|\p(|T^{\hat \sigma, \diamond}/\sqrt{ 2\lceil nb_n\rceil-m'_n}|_\infty \leq x|\FF_n)-\p\Big(\frac{1}{\sqrt{2nb_n}}\Big|\sum_{i=1}^{2\lc nb_n\rc-1}\tilde Y^\sigma_i\Big|_\infty\leq x\Big)\Big|\notag\\&=O_p\Big(\vartheta^{1/3}_n\{1\vee \log (\frac{np}{\vartheta_n})\}^{2/3}+\Theta\big(\sqrt{m_n\log np}(\frac{1}{\sqrt{nb_n}}+b_n^3)(np)^{\frac{1}{q}}\big)^{q/(q+1)},np\big)\notag\\&+\Theta\big(\big(\sqrt{m_n\log np}((g_n+\tau_n)\eta_n)(np)^{\frac{1}{
	q}}\big)^{q/(q+1)},np\big)+
	\eta_n^{-q'}\Big)
\end{align}
where $\vartheta_n$ is defined in Theorem  \ref{Thm2.1-2021}.
The  assertion of Theorem \ref{sigma-bootstrap} then   follows from \eqref{Deltahatsigma}. 

Now we prove \eqref{4.23-7-31}.
By the first step in the proof of Theorem   \ref{Thm2.1-2021} 
it follows  that 
\begin{align}\label{6-18-123}
&	\sup_{x\in \mathbb R}\Big|\p(|T^{\sigma, \diamond}/\sqrt{ 2\lceil nb_n\rceil-m'_n}|_\infty \leq x|\FF_n)-\p\Big(\frac{1}{\sqrt{2nb_n}}\Big|\sum_{i=1}^{2\lc nb_n\rc-1}\tilde Y^\sigma_i\Big|_\infty\leq x\Big)\Big|\notag\\&=O_p\Big(\vartheta^{1/3}_n\{1\vee \log (\frac{np}{\vartheta_n})\}^{2/3}\notag\\&+\Theta\Big(\big(\sqrt{m_n\log np}(\frac{1}{\sqrt{nb_n}}+b_n^3)(np)^{\frac{1}{q}}\big)^{q/(q+1)},np\Big)
\Big) .
\end{align}
By similar arguments as given in the proof of Theorem \ref{Thm2.1-2021} we have
\begin{align}\label{6-18-125}
\E \big (| T^{\sigma,\diamond}-T^{\hat \sigma, \diamond}|^q_\infty\mathbf 1(A_n)\big |\FF_n\big )
&\leq M\Big|\sqrt{\log np}\max_{r=1}^{(n-2\lceil nb_n\rceil+1 )p} \Big(\sum_{j=1}^{\lceil 2nb_n \rceil-m_n'}( \hat S^\sigma_{jm'_n,r}- \hat S^{\hat \sigma}_{jm'_n,r})^2\mathbf 1(A_n)\Big)^{1/2}\Big|^q 
\end{align}

for some large constant $M$ almost surely, and the triangle inequality, a similar argument as given in the proof of Proposition 1.1 in \cite{dette2020prediction} and 
\eqref{maxtrick} yield  
\begin{align*}
\frac{1}{\sqrt{ 2\lceil nb_n\rceil-m_n}}
\Big \|\max_{r=1}^{(n-2\lceil nb_n\rceil+1 )p} \Big(\sum_{j=1}^{\lceil 2nb_n \rceil-m_n'}(\hat S^\sigma_{jm'_n,r}- \hat S^{\hat \sigma}_{jm'_n,r})^2\mathbf 1(A)\Big)^{1/2}\Big \|_{q} =O\big(\sqrt{m_n}(g_n+\tau_n)\eta_n(np)^{\frac{1}{q}}\big).
\end{align*}

This together with the (conditional version) of Lemma \ref{anti} and  \eqref{6-18-125} shows that 
\begin{align}
\sup_{x\in \mathbb R} \Big |\p \Big (\frac{|T^{\hat \sigma, \diamond}|_\infty}{\sqrt{2\lceil nb_n\rceil-m_n}}
> x \Big |\FF_n \Big )-\p  \Big (\frac{1}{\sqrt{2nb_n}}\Big |\sum_{i=1}^{2\lc nb_n\rc-1}\tilde Y^\sigma_i \Big|_\infty>x \Big )\Big |
\notag\\ \leq
\sup_{x\in \mathbb R} \Big |\p \Big(
\frac{|T^{\sigma, \diamond}|_\infty }{\sqrt{ 2\lceil nb_n\rceil-m_n}}
> x \Big |\FF_n \Big )-\p\Big (\frac{1}{\sqrt{2nb_n}}\Big |\sum_{i=1}^{2\lc nb_n\rc-1}\tilde Y^\sigma_i \Big|_\infty>x)\Big |\notag\\+\p\Big ( \frac{|T^{\diamond,\sigma}-T^{\diamond,\hat \sigma}|_\infty}{\sqrt{ 2\lceil nb_n\rceil-m_n}}   >\delta \Big|\FF_n\Big )+O\big(\Theta(\delta, np) \big )\notag\\
\leq
\sup_{x\in \mathbb R} \Big |\p \Big ( \frac{| T^{\sigma,\diamond}|_\infty }{\sqrt{ 2\lceil nb_n\rceil-m_n}} > x\Big |\FF_n \Big )-\p \Big(\frac{1}{\sqrt{2nb_n}} \Big |\sum_{i=1}^{2\lc nb_n\rc-1}\tilde Y^\sigma_i \Big |_\infty>x\Big )\Big |\notag\\
+O_p\big(\delta^{-q}\big(\sqrt{m_n\log np}((g_n+\tau_n)\eta_n)(np)^{\frac{1}{q}}\big)^{q}\big)+O\big(\Theta(\delta, np)+\eta_n^{-q'}\big),
\nonumber 
\end{align}
where we used  Markov's inequality and  \eqref{6-18-124}. Taking $$\delta=\big(\sqrt{m_n\log np}((g_n+\tau_n)\eta_n)(np)^{\frac{1}{q}}\big)^{q/(q+1)}$$  and observing \eqref{6-18-123} yields \eqref{4.23-7-31} and proves the assertion.
\end{proof}


\setcounter{equation}{0}





\section{Proposition \ref{Grid-Result} and Proof of Proposition \ref{prop-lrv}}\label{Proof-Prop}

\subsection{Proposition \ref{Grid-Result}}
The proof of Theorems \ref{SCB} is  based on the following auxiliary result providing a Gaussian approximation for the maximum deviation of the quantity  $\sqrt{nb_n}|\hat \Delta(u,t_v)|$
over the grid of $\{1/n,...,n/n\}\times \{t_1,...,t_p\}$ where $t_v= \frac{v}{p}$ ($v=1, \ldots , p$).   
\begin{proposition}\label{Grid-Result}
Assume that $n^{1+a}b_n^9=o(1)$, $n^{a-1}b_n^{-1}=o(1)$ for  some $0<a<4/5$,
and let Assumptions  \ref{assreg},  \ref{asserr} and \ref{asskern} be satisfied.
\begin{itemize}
\item[(i)] 
For a fixed  $u\in (0,1)$, let $Y_1(u) , \ldots , Y_n (u)$ 
denote  a sequence of centered  $p$-dimensional Gaussian vectors  such that  $ Y_i(u) $ has the same auto-covariance structure of the vector  $Z_i(u)$ defined in \eqref{3.1}. If 
$p = O (\exp(n^{\iota}))$	 
for some $0\leq \iota <1/11$, then
\begin{align*}
\mathfrak{P}_{p,n} (u)	 &:=    \sup_{x\in \mathbb R}\Big|\p \Big(\max_{1\leq v\leq p}\sqrt{nb_n}|\hat \Delta(u,t_v)|\leq x \Big)-\p \Big( \Big |\frac{1}{\sqrt{nb_n}}\sum_{i=1}^nY_i(u) \Big |_\infty\leq x \Big ) \Big |\notag\\
& ~~~~~~~~~~~  = O \Big ((nb_n)^{-(1-11\iota)/8}+\Theta\Big(\sqrt{nb_n}\Big(b_n^4+\frac{1}{n}\Big), p\Big)\Big)
\end{align*}

\item[(ii)]	Let
$\tilde Y_1, \ldots , \tilde Y_{2\lc nb_n\rc-1}$ denote independent
$(n-2\lceil nb_n\rceil +1)p$-dimensional centered Gaussian vectors with the same auto-covariance structure as  the vector $\tilde Z_i$
in \eqref{tildeZi}.
If $np=O(\exp(n^\iota))$  for some $0\leq \iota <1/11$,
then
\begin{align*}
\mathfrak{P}_{p,n}  	 &:=       \sup_{x\in \mathbb R}\Big|\p\Big(\max_{\lc nb_n\rc\leq l\leq n-\lc nb_n\rc,1\leq v\leq p}\sqrt{nb_n}|\hat \Delta(\tfrac{l}{n},t_v)|\leq x\Big)-\p\Big(\Big|\frac{1}{\sqrt{nb_n}}\sum_{i=1}^{2\lc nb_n\rc-1}\tilde Y_i\Big|_\infty\leq x\Big)\Big|\notag\\
& ~~~~~~~~~~ =O\Big((nb_n)^{-(1-11\iota)/8}+\Theta\Big(\sqrt{nb_n}\Big(b_n^4+\frac{1}{n}\Big), np\Big)\Big)
\end{align*}
\end{itemize}
\end{proposition}
\begin{proof}
Using Assumptions   \ref{assreg}, \ref{asskern} and  a  Taylor expansion  we obtain  
\begin{align}\label{mean-inference}
\sup_{\stackrel{u\in [b_n,1-b_n]}{t\in[0,1]}}\Big|\E(\hat m(u,t))-m(u,t)-b_n^2\int K(v)v^2dv\frac{\partial^2 }{\partial u^2}m(u,t)/2\Big|\leq M \Big(\frac{1}{n}+b_n^4\Big)
\end{align}
for some constant $M$. Notice that by assumption $\int K(v)v^2dv=0$.
Notice that for $u\in [b_n,1-b_n]$, 
\begin{align}
\label{mean-inference2}
\hat m(u,t)-\E(\hat m(u,t)) &=\frac{1}{nb_n\tilde K(u)}\sum_{i=1}^n G(\tfrac{i}{n},t,\FF_i)K\Big(\frac{\tfrac{i}{n}-u}{b_n}\Big) \\
& =\frac{1}{nb_n\tilde K(u)}
\sum_{i=\lceil n(u-b_n)\rceil}^{\lf n(u+b_n)\rf }
G(\tfrac{i}{n},t,\FF_i)K\Big(\frac{\tfrac{i}{n}-u}{b_n}\Big).
\nonumber
\end{align} 
Therefore, observing the definition of $Z_i(u)$ 
in \eqref{3.1} we have
(notice that $Z_i(u)$ is a vector of zero if $|\tfrac{i}{n}-u|\geq b_n$) 
\begin{align*}
\max_{1\leq v\leq p}
\sqrt {nb_n}|\hat m(u,t_v)-\E(\hat m(u,t_v))|\tilde K(u)=\Big|\frac{1}{\sqrt{nb_n}} 
\sum_{i=\lceil n(u-b_n)\rceil}^{\lf n(u+b_n)\rf }
Z_i(u)\Big|_\infty.
\end{align*}
We will now apply Corollary 2.2 of  \cite{zhang2018gaussian}  and check its assumptions first.
By Assumption \ref{asserr}(2)  and  the fact that the kernel is bounded it follows that
\begin{align*}
\max_{1\leq l\leq p}\sup_i\|Z_{i,l}(u)-Z_{i,l}^{(i-j)}(u)\|_2=O(\chi^j),
\end{align*}
where for any  (measurable function) $g=g(\FF_i)$, we define for $j\leq i$ the function
$g^{(j)}$ by  $g^{(j)}=g(\FF_i^{(j)})$, where $\FF_i^{(j)}=(\ldots,\eta_{j-1}, \eta_j',\eta_{j+1}, \ldots , \eta_i)$ and $\{\eta'_i\}_{i\in \mathbb Z}$ is an independent  copy of $\{\eta_i\}_{i\in \mathbb Z}$ (recall that  $\FF_i=(\eta_{-\infty},...,\eta_i)$). 
Lemma \ref{Lemma1.3} in Section \ref{aux}
shows that condition  (9) in the paper of  \cite{zhang2018gaussian} is satisfied. Moreover Assumption \ref{asserr}(1)
implies condition  (13) in this reference. 
Observing that for  random vector $v=(v_1,...,v_p)^\top $ and all $x\in \mathbb R$
$$
\{|v|_\infty\leq x\} =\Big \{\max(v_1,...,v_p,-v_{1},...,-v_{p})\leq x \Big \},
$$ 
we can use 
Corollary 2.2 of \cite{zhang2018gaussian} 
\begin{align}\label{approx_Yi(u)}
\sup_{x\in \mathbb R}\Big|\p\Big(\Big|\frac{1}{\sqrt {nb_n}}\sum_{i=1}^nY_i(u)\Big|_\infty\leq x\Big)-\p\Big(\frac{1}{\sqrt {nb_n}}\Big|\sum_{i=1}^nZ_i(u)\Big|_\infty\leq x\Big)\Big|=O((nb_n)^{-(1-11\iota’)/8}).
\end{align}
Therefore by \eqref{mean-inference}, \eqref{approx_Yi(u)} and Lemma \ref{anti}, and the fact that $\tilde K(u)=1+O(\frac{1}{nb_n})$ for $b_n\leq u\leq 1-b_n$
\begin{align}\label{2023-10-1}
& \sup_{x\in \mathbb R}\Big|\p \Big(\max_{1\leq v\leq p}\sqrt{nb_n}|\hat \Delta(u,t_v)|\leq x \Big)-\p \Big( \Big |\frac{1}{\sqrt{nb_n}}\sum_{i=1}^nY_i(u) \tilde K(u)\Big |_\infty\leq x \Big ) \Big |\notag\\
& = O \Big ((nb_n)^{-(1-11\iota)/8}+\Theta\Big(\sqrt{nb_n}\Big(b_n^4+\frac{1}{nb_n}\Big), p\Big)\Big).
\end{align}
Using  Theorem 2 of \cite{chernozhukov2015comparison}, it follows that
\begin{align}\label{2023-10-2}
& \sup_{x\in \mathbb R}\Big|\p \Big( \Big |\frac{1}{\sqrt{nb_n}}\sum_{i=1}^nY_i(u)\Big |_\infty\leq x \Big )-\p \Big( \Big |\frac{1}{\sqrt{nb_n}}\sum_{i=1}^nY_i(u) \tilde K(u)\Big |_\infty\leq x \Big ) \Big |\notag\\
& = O \Big ((nb_n)^{-1/3}\log^{2/3} (npb_n)\Big).
\end{align}
Since $p = O (\exp(n^{\iota}))$ ,
Then part (i) of the assertion  follows from \eqref{2023-10-1} and \eqref{2023-10-2}.
For part (ii), notice that $\tilde K(i/n)=\tilde K(j/n)$ for $i,j\in \mathbb Z$ such that $b_n\leq i/n,j/n\leq 1-b_n$. Let $\tilde K=\tilde K(\lf n/2\rf/n)$. Further note that by the definition of the vector  $\tilde Z_i$ in \eqref{tildeZi}  we have that  (Recall the notation $W_n(u,t)$ in \eqref{3.3})
\begin{align}\label{neweq}
\max_{1\leq v\leq p}\max_{\lceil nb_n\rceil\leq l\leq n-\lceil nb_n\rceil }
\tilde K| W_n(\tfrac{l}{n},t_v)|&=\max_{\lceil nb_n\rceil\leq l\leq n-\lceil nb_n\rceil }\Big|\frac{1}{\sqrt{nb_n}}\sum_{i=1}^n Z_i(\frac{l}{n})\Big|_\infty
=\Big|\frac{1}{\sqrt {nb_n}}\sum_{i=1}^{2\lc nb_n\rc-1}\tilde Z_i\Big|_\infty~.
\end{align}
Let
$\tilde Z_{i,s}$ denote the $s$th entry of the vector 
$\tilde Z_i$ defined in \eqref{tildeZi} ($1\leq s\leq (n-2\lc nb_n\rc+1)p$). 
By Assumption  \ref{asserr}(2)  
it follows that
\begin{align*}
\max_{1\leq s\leq (n-2\lceil nb_n\rceil +1)p}\sup_i\|\tilde Z_{i,s}-\tilde Z_{i,s}^{(i-j)}\|_2=O(\chi^j).
\end{align*}
By  Lemma \ref{Lemma1.3} in Section \ref{aux} we obtain  the inequality \begin{align*}
c_1\leq \min_{1\leq j\leq (n-2\lc nb_n\rc+1)p}\tilde \sigma_{j,j}\leq \max_{1\leq j\leq (n-2\lc nb_n\rc+1)p}\tilde \sigma_{j,j}\leq c_2
\end{align*}
for the quantities
$$
\tilde \sigma_{j,j}:=
{1 \over 2\lc nb_n\rc -1} \sum_{i,l=1}^{2\lc nb_n\rc-1}{\rm Cov}(\tilde Z_{i,j},\tilde Z_{l,j}).
$$
Therefore condition  (9) in the paper of  \cite{zhang2018gaussian} holds, and condition  (13) 
in this reference 
follows from  Assumption \ref{asserr}(1). As a consequence,  Corollary 2.2  in  \cite{zhang2018gaussian} (the validity of Corollary 2.2 of \cite{zhang2018gaussian} for $\tilde Z_i$ can be verified via the argument  of Proposition 2.1, A.1 and Theorem 2.1 of that paper and via \eqref{neweq}; details are omitted  for the sake of brevity)
yields 
\begin{align}\label{eq25-2021-2-2}
\sup_{x\in \mathbb R}\Big|\p\Big(\max_{
{\substack{\lceil nb_n\rceil \leq l_1\leq n-\lc nb_n\rc \\ 1\leq l_2\leq p}}}\tilde  K |W_n(\tfrac{l_1}{n},\tfrac{l_2}{p})|\leq x\Big)-\p\Big(\Big|\frac{1}{\sqrt{nb_n}}\sum_{i=1}^{2\lc nb_n\rc-1}\tilde Y_i\Big|_\infty\leq x\Big)\Big|=O((nb_n)^{-(1-11\iota)/8}).
\end{align}
Using  Theorem 2 of \cite{chernozhukov2015comparison} and the the fact that $\tilde K=1+O(\frac{1}{nb_n})$, it follows that
\begin{align}
\sup_{x\in \mathbb R}\Big|\p \Big( \Big |\frac{1}{\sqrt{nb_n}}\sum_{i=1}^n\tilde Y_i\Big |_\infty\leq x \Big )-\p \Big( \Big |\frac{1}{\sqrt{nb_n}}\sum_{i=1}^n\tilde Y_i \tilde  K\Big |_\infty\leq x \Big ) \Big |
= O \Big ((nb_n)^{-1/3}\log^{2/3} (n^2pb_n)\Big).\label{2023-10-3}
\end{align}
Consequently part  (ii) follows 
by the same arguments given in the  proof of part (i) via an application of  Lemma \ref{anti}. 
\end{proof}

\subsection{Proof of Proposition \ref{prop-lrv}}
\label{632}
\begin{proof}
Define
$\tilde S^G_{k,r}(t)=\frac{1}{\sqrt r}\sum_{i=k}^{k+r-1}G(i/n,t,\FF_i)$, and  define  for $u\in[w/n,1-w/n]$
$$
\tilde \Delta_j(t)=\frac{\tilde S^G_{j-w+1,w}(t)-\tilde S^G_{j+1,w}(t)}{\sqrt w}
~,~~
\tilde{\sigma}^2(u,t)=\sum_{j=1}^n\frac{w\tilde \Delta_j^2(t)}{2}\bar \omega(u,j)
$$  
as the analogs of  $\Delta_j(t)$ defined in the main article and the quantities in   \eqref{2018-5.4}, respectively.
We also use the convention $\tilde\sigma^2(u,t)=\tilde\sigma^2(w/n,t)$  and 
$\tilde \sigma^2(u,t)=\tilde \sigma^2(1-w/n,t)$ if 
$u\in[0,w/n)$ and   $u\in(1-w/n,1]$, respectively.
Assumption \ref{assreg} and the mean value theorem yield
\begin{align}\label{6-16-99}
\max_{w\leq j\leq n-w}\sup_{0\leq t\leq 1}|\tilde \Delta_j(t)-\Delta_j(t)|=\max_{w\leq j\leq n-w}\sup_{0\leq t\leq 1} \Big |\sum^{j}_{r=j-w+1}m(r/n,t)-\sum_{r=j+1}^{j+w}m(r/n,t)
\Big |=O(w/n).
\end{align}
On the other hand,   Assumption \ref{asserr} and Assumption \ref{asssim}
and   similar arguments as given in  the proof of Lemma 3 of \cite{zhou2010simultaneous} give 
\begin{align}\label{6-16-100}
\max_j\|\tilde \Delta_j(t)\|_{q'}=O(\sqrt w), ~\max_j \Big \|\frac{\partial}{\partial t}\tilde \Delta_j(t)\Big  \|_{q'}=O(\sqrt w).
\end{align}
{Here  we use the convention that $\frac{\partial }{\partial t}\tilde \Delta_j|_{t=0}=\frac{\partial }{\partial t}\tilde \Delta_j|_{t=0+}$, $\frac{\partial }{\partial t}\tilde \Delta_j|_{t=1}=\frac{\partial }{\partial t}\tilde \Delta_j|_{t=1-}$.}
Moreover, Proposition B.1. of \cite{dette2019change} yields
\begin{align}
\max_j \Big  \|\sup_t|\tilde \Delta_j(t)|\Big  \|_{q'}=O(\sqrt w).\label{6-16-101}
\end{align}
Now we introduce the notation $C_j(t)=\tilde \Delta_j(t)-\Delta_j(t)$ (note that 
this quantity is not random) and obtain  by \eqref{6-16-99} the representation 
\begin{align}\label{preWdiamond}
\tilde \sigma^2(u,t)-\hat \sigma^2(u,t)&=\sum_{j=1}^n\frac{w(2\tilde \Delta_j(t)-C_j(t))C_j(t)}{2}\bar w(u,j)\notag
\\&=\sum_{j=1}^nw\tilde \Delta_j(t)C_j(t)\bar \omega(u,j)+O(w^3/n^2)
\end{align}
uniformly with respect to $u,t$. Furthermore,  by \eqref{6-16-99}
we have 
\begin{align}
\sup_{t\in[0,1]}\Big |\sum_{j=1}^nw\tilde \Delta_j(t)C_j(t)\bar \omega(u,j)
\Big  |\leq
W^\diamond(u) := M(w/n)\sum_{j=1}^nw\sup_{t\in[0,1]}|\Delta_j(t)|\bar \omega(u,j)~,
\end{align}
where $M$ is a sufficiently large constant. Notice that $W^\diamond(u)$ is differentiable with respect to the variable  $u$. Therefore it follows from the  triangle inequality, \eqref{6-16-101} and Proposition B.1 of \cite{dette2019change},  that
\begin{align}\label{Wdiamond}
\Big \|\sup_{u\in[\gamma_n,1-\gamma_n]}|W^\diamond(u)|\Big  \|_{q'}=O \Big  (\frac{w^{5/2}}{n}\tau_n^{-1/q'}\Big  ).
\end{align}
Combining \eqref{preWdiamond}-- \eqref{Wdiamond}, we obtain
\begin{align}\label{6-17-105}
\Big  \|\sup_{\substack{u\in[\gamma_n,1-\gamma_n]\\t\in [0,1]}} |\tilde \sigma^2(u,t)-\hat \sigma^2(u,t)|\Big \|_{q'}=O \Big (\frac{w^{5/2}}{n}\tau_n^{-1/q'}+w^3/n^2\Big ).
\end{align} 
By Burkholder inequality  \citep[see for example][]{wu2005nonlinear} in $\mathcal L^{q'/2}$ norm, \eqref{6-16-100} and similar arguments as given in the proof of Lemma 3 in \cite{zhou2010simultaneous} we have
\begin{align}
\notag
&\sup_{\substack{u\in[\gamma_n,1-\gamma_n]\\t\in[0,1]}}\big \|\tilde \sigma^2(u,t)-\E(\tilde \sigma^2(u,t))\big \|_{q'/2}=O\big (w^{1/2}n^{-1/2}\tau_n^{-1/2}\big ),\\  \notag
&\sup_{\substack{u\in[\gamma_n,1-\gamma_n]\\t\in[0,1]}}\Big \|\frac{\partial }{\partial t}(\tilde \sigma^2(u,t)-\E(\tilde \sigma^2(u,t))) \Big \|_{q'/2}=O\big (w^{1/2}n^{-1/2}\tau_n^{-1/2}\big ),\\  
&\sup_{\substack{u\in[\gamma_n,1-\gamma_n]\\t\in[0,1]}}\Big  \|
\big  (\frac{\partial }{\partial u}+\frac{\partial^2 }{\partial u\partial t}
\big (\tilde \sigma^2(u,t)-\E(\tilde \sigma^2(u,t)) \big  )\Big  \|_{q'/2}=O\big (w^{1/2}n^{-1/2}\tau_n^{-1/2-1}\big ).\label{6-17-109-1}
\end{align}
It can be shown by  similar but simpler argument as given in the proof of Proposition B.2 of \cite{dette2019change}   that these estimates imply
\begin{align}
\big \|	\sup_{\substack{u\in[\gamma_n,1-\gamma_n]\\t\in[0,1]}}|\tilde \sigma^2(u,t)-\E(\tilde \sigma^2(u,t))|\big \|_{q'/2}=O\big (w^{1/2}n^{-1/2}\tau_n^{-1/2-2/q'}\big ).\label{6-17-109}
\end{align}
Moreover, it follows from the proof of Theorem 4.4 of \cite{dette2019detecting} 
that   
\begin{align} \notag
\sup_{\substack{u\in[\gamma_n,1-\gamma_n]\\t\in[0,1]}}\Big|\E\tilde \sigma^2(u,t)-\sigma^2(u,t)\Big|=O\Big(\sqrt{w/n}+w^{-1}+\tau_n^2\Big),\\
\sup_{\substack{u\in[0,\gamma_n)\cup(1-\gamma_n,1]\\t\in[0,1]}}\Big|\E\tilde \sigma^2(u,t)-\sigma^2(u,t)\Big|=O\Big(\sqrt{w/n}+w^{-1}+ {\tau_n}\Big)\label{6-17-111}
\end{align} 
and the assertion is a consequence of  \eqref{6-17-105}, \eqref{6-17-109} 
and \eqref{6-17-111}. 
\end{proof}

\section{Some auxiliary results} \label{aux}
\setcounter{equation}{0}

This section contains several technical lemmas, which will be used in the proofs of the main results in Section \ref{proof-main}.

\begin{lemma}\label{anti}
For any random vectors $X,X',Y$, and $\delta\in \mathbb R$, we have that
\begin{align}\label{eq39-2021}
\sup_{x\in \mathbb R}|\p(|X'|> x)-\p(|Y|>x)|\leq
\sup_{x\in \mathbb R}|\p(|X|> x)-\p(|Y|>x)|\notag\\+\p(|X-X'|>\delta)+2\sup_{x\in \mathbb R}\p(|Y-x|\leq \delta).
\end{align}
Furthermore, if $Y=(Y_1,...,Y_p)^\top$ is a $p$-dimensional  Gaussian  vector and there exist positive constants $c_1\leq c_2$ such that for all $1\leq j\leq p$, $c_1\leq \E(Y^2_j)\leq c_2$, then
\begin{align}\label{eq39-Gaussian-2021}
\sup_{x\in \mathbb R}|\p(|X'|> x)-\p(|Y|_\infty>x)|\leq
\sup_{x\in \mathbb R}|\p(|X|> x)-\p(|Y|_\infty>x)|+\p(|X-X'|>\delta)\notag\\+C\Theta(\delta, p),
\end{align}
where $C$ is a constant only dependent on $c_1$ and $c_2$.
\end{lemma}

\begin{proof}[Proof of Lemma \ref{anti}] By triangle inequality, we shall see that
\begin{align}
\p(|X'|>x)-\p(|Y|>x)\leq \p(|X'-X|>\delta)+\p(|X|>x-\delta)-\p(|Y|>x),\label{eq40-2021}\\
\p(|X'|>x)-\p(|Y|>x)\geq -\p(|X'-X|>\delta)+\p(|X|>x+\delta)-\p(|Y|>x)\label{eq41-2021}.
\end{align}
Notice that  right-hand side of \eqref{eq40-2021} is
\begin{align*}
\p(|X'-X|>\delta)+\p(|X|>x-\delta)-\p(|Y|>x-\delta)+\p(|Y|>x-\delta)-\p(|Y|>x).
\end{align*}
The absolute value of the above expression is then uniformly bounded   by
\begin{align}\label{eq43-2021}
\p(|X'-X|>\delta)+\sup_{x\in \mathbb R}|\p(|X|>x)-\p(|Y|>x)|+2\sup_{x\in \mathbb R}\p(|Y-x|\leq \delta).
\end{align}
Similarly, the absolute value of right-hand side of \eqref{eq41-2021} is also uniformly bounded by \eqref{eq43-2021}, which proves \eqref{eq39-2021}. Finally, \eqref{eq39-Gaussian-2021} follows from \eqref{eq39-2021} and  an application of Corollary 1 in \cite{chernozhukov2015comparison}. Note that in this result the constant $C$ is determined by $\max_{1\leq j\leq p}\E(Y^2_j)\leq c_2$ and  $\min_{1\leq j\leq p}\E(Y^2_j) \geq c_1$. 
\end{proof} 

The following result is a consequence of 
of  Lemma 5 of  \cite{zhou2010simultaneous}.

\begin{lemma}\label{Lemma1.2} Under the assumption \ref{asserr}(2), we have that
\begin{align*}
\sup_{u_1,u_2,t_1, t_2\in[0,1]}|\E(G(u_1,t,\FF_i)G(u_2,t_2,\FF_j))|=O(\chi^{|i-j|}).
\end{align*} 
\end{lemma}

\begin{lemma}\label{Lemma1.3} 
Define 
$$
\sigma_{j,j}(u)=\frac{1}{nb_n}\sum_{i,l=1}^n
{\rm Cov} (Z_{i,j}(u),Z_{l,j}(u))
$$
where $Z_{i,j}$ are the components of the vector $Z_i(u)$ defined in \eqref{3.1}.
If $b_n=o(1)$,  $\frac{\log n}{nb_n}=o(1)$
and  Assumption   \ref{asserr} and Assumption  \ref{asskern}  are satisfied, then there exist positive constants $c_1$ and $c_2$ such that
for sufficiently large $n$
\begin{align*}
0<c_1\leq \min_{1\leq j\leq p}\sigma_{j,j}(u)\leq \max_{1\leq j\leq p}\sigma_{j,j}(u)\leq c_2<\infty .
\end{align*}
for all $u\in [b_n,1-b_n]$.
Moreover, we have for 
\begin{align}
\label{holg1}
\tilde \sigma_{j,j}:=
{1 \over 2\lc nb_n\rc -1} \sum_{i,l=1}^{2\lc nb_n\rc-1}{\rm Cov}(\tilde Z_{i,j},\tilde Z_{l,j}),
\end{align}
the estimates 
\begin{align*}
c_1\leq \min_{1\leq j\leq (n-2\lc nb_n\rc+1)p}\tilde \sigma_{j,j}\leq \max_{1\leq j\leq (n-2\lc nb_n\rc+1)p}\tilde \sigma_{j,j}\leq c_2.
\end{align*}
\end{lemma}

\begin{proof}[Proof of Lemma \ref{Lemma1.3}]
By definition,
\begin{align*}
\sigma_{j,j}(u)=\frac{1}{nb_n}\sum_{i,l=1}^n\E\Big(G(\tfrac{i}{n},t_j,\FF_i)K\Big(\frac{\tfrac{i}{n}-u}{b_n}\Big)G(\tfrac{l}{n},t_j,\FF_l)K\Big(\frac{\tfrac{l}{n}-u}{b_n}\Big)\Big).
\end{align*}
Observing Assumption \ref{asserr}  and Lemma \ref{Lemma1.2}, we have 
\begin{align*}
\E(G(\tfrac{i}{n},t_j,\FF_i)G(\tfrac{l}{n},t_j,\FF_l)-G(u,t_j,\FF_i)G(u,t_j,\FF_l))=O\big(\min(\chi^{|l-i|},b_n)\big)
\end{align*}
uniformly with respect to $u\in [b_n,1-b_n]$, $|\tfrac{i}{n}-u|\leq b_n$ and $|\tfrac{l}{n}-u|\leq b_n$.
Consequently, observing Assumption  \ref{asskern}
it follows that 
\begin{align}\label{1-12-29}
\sigma_{j,j}(u) &=\frac{1}{nb_n}\sum_{i,l=1}^n\E\Big(G(u,t_j,\FF_i)K\Big(\frac{\tfrac{i}{n}-u}{b_n}\Big)G(u,t_j,\FF_l)K\Big(\frac{\tfrac{l}{n}-u}{b_n}\Big)\Big)+O(-b_n\log b_n)
\end{align}
On the other hand, if  $r_n$ is  a sequence  such that $r_n=o(1)$ and $nb_nr_n\rightarrow \infty$, $A(u,r_n):=\{l:|\frac{\tfrac{l}{n}-u}{b_n}|\leq 1-r_n, u\in [b_n,1-b_n]\}$
we obtain by \eqref{1-12-29} and Lemma \ref{Lemma1.2} that 

\begin{align}
\sigma_{j,j}(u) &=\frac{1}{nb_n}\sum_{l=1}^n\sum_{i=1}^n\mathbf 1(|i-l|\leq nb_nr_n)\E\Big(G(u,t_j,\FF_i)K\Big(\frac{\tfrac{i}{n}-u}{b_n}\Big)G(u,t_j,\FF_l)K\Big(\frac{\tfrac{l}{n}-u}{b_n}\Big)\Big)\notag\\&+O(-b_n\log b_n+\chi^{nb_nr_n})\notag\\
&=\frac{1}{nb_n}\sum_{l=1}^nK^2\Big(\frac{\tfrac{l}{n}-u}{b_n}\Big)\sum_{\substack{1\leq i\leq n,\\ |i-l|\leq nb_nr_n}}\E\Big(G(u,t_j,\FF_i)G(u,t_j,\FF_l)\mathbf 1\Big(\Big|\frac{\tfrac{i}{n}-u}{b_n}\Big|\leq 1\Big)\Big)\notag\\&+O(-b_n\log b_n+\chi^{nb_nr_n}+r_n)\notag\\
&=\frac{1}{nb_n}\sum_{\substack{1\leq l\leq n,\\l\in A(u,r_n)}}K^2\Big(\frac{\tfrac{l}{n}-u}{b_n}\Big)\sum_{\substack{1\leq i\leq n,\\ |i-l|\leq nb_nr_n}}\E\Big(G(u,t_j,\FF_i)G(u,t_j,\FF_l)\mathbf 1\Big(\Big|\frac{\tfrac{i}{n}-u}{b_n}\Big|\leq 1\Big)\Big)\notag\\&+O(-b_n\log b_n+\chi^{nb_nr_n}+r_n)\label{C7-new}
\end{align}
uniformly for $j \in \{ 1, \ldots ,  p\} $.
We obtain by  the definition of the long-run variance $\sigma^2(u,t)$ in Assumption \ref{asserr}(4)  and Lemma \ref{Lemma1.2} that 
\begin{align}
\Big|\sum_{i=1}^n\E\Big(G(u,t_j,\FF_i)G(u,t_j,\FF_l)\mathbf 1\Big(\Big|\frac{\tfrac{i}{n}-u}{b_n}\Big|\leq 1, |i-l|\leq nb_nr_n\Big)\Big)-\sigma^2(u,t_j)\Big|
=O(\chi^{nb_nr_n})\label{1-12-30}
\end{align}
uniformly with respect to $l\in A(u,r_n)=\{l:|\frac{\tfrac{l}{n}-u}{b_n}|\leq 1-r_n, u\in [b_n,1-b_n]\}$ and  $j \in \{ 1, \ldots ,  p\} $.
Combining \eqref{C7-new} and \eqref{1-12-30}
and using Lemma \ref{Lemma1.2} yields 
\begin{align}
\sigma_{j,j}(u)&=\frac{1}{nb_n}\sum_{l=1}^nK^2\Big(\frac{\tfrac{l}{n}-u}{b_n}\Big)\sigma^2(u,t_j)+O(-b_n\log b_n+\chi^{nb_nr_n}+r_n)
\notag\\
&=\sigma^2(u,t_j)\int_{-1}^1K^2(t)dt+O\Big(-b_n\log b_n+\chi^{nb_nr_n}+r_n+\frac{1}{nb_n}\Big).
\end{align}
Let $r_n=\frac{a\log n}{nb_n}$ for some sufficiently large positive constant $a$,  
then the assertion of the lemma follows in view of Assumption \ref{asserr}(4)). 

For the second assertion, consider the case that $j=k_1p+k_2$ for some $0\leq k_1\leq n-2\lceil nb_n\rceil $ and $1\leq k_2\leq p$.
Therefore by definition \eqref{tildeZi} in the main article,
$$\tilde Z_{i,k_1p+k_2}=G(\tfrac{i+k_1}{n},\tfrac{k_2}{p},\FF_{i+k_1})K(\tfrac{i-\lceil nb_n\rceil }{nb_n}),$$
which gives for the quantity in \eqref{holg1}
\begin{align*}\tilde \sigma_{k_1p+k_2,k_1p+k_2}=
{1 \over 2\lc nb_n\rc -1} \sum_{i,l=1}^{2\lc nb_n\rc-1}\E\Big(G(\tfrac{i+k_1}{n},\tfrac{k_2}{p},\FF_{i+k_1})K(\tfrac{i-\lceil nb_n\rceil }{nb_n})G(\tfrac{l+k_1}{n},\tfrac{k_2}{p},\FF_{l+k_1})K(\tfrac{l-\lceil nb_n\rceil }{nb_n})\Big)\end{align*}
Consequently, putting  $i+k_1=s_1$ and $l+k_1=s_2$ and using  a change of variable, we obtain  that
\begin{align}
\tilde \sigma_{k_1p+k_2,k_1p+k_2}=\sigma_{k_2,k_2}\Big(\tfrac{k_1+\lceil nb_n\rceil}{n}\Big) ,
\end{align}
which finishes the proof.
\end{proof}



\end{document}